\documentclass[10pt]{article}
\usepackage{amsmath,amsthm,amsfonts,amssymb,bm}
\usepackage[pagebackref=true]{hyperref}
\usepackage[light]{antpolt}    \usepackage[T1]{fontenc} 
\usepackage{color,graphicx,graphics}
\usepackage{mathrsfs,amssymb,bm,enumerate}
\usepackage[bbgreekl]{mathbbol}
\DeclareSymbolFontAlphabet{\mathbb}{AMSb}
\DeclareSymbolFontAlphabet{\mathbbl}{bbold}

\usepackage{ulem}
\usepackage{xcolor}
\usepackage[pagewise,mathlines,switch]{lineno}

\usepackage{amsbsy}
\DeclareMathAlphabet{\mathpzc}{OT1}{pzc}{m}{it}

 \usepackage{bbold}
 \usepackage{bbm}
 
 \newcommand{\scal}[1]{\left\langle #1 \right\rangle}
 \newcommand{\sett}[1]{\left\{   #1   \right\}}
 \newcommand{\norm}[1]{\left\|   #1   \right\|}
 \newcommand{\abso}[1]{\left|   #1   \right|}
 \newcommand{\paar}[1]{\left( #1 \right)}

 \newcommand{\fd}{\mathfrak{D}}

 \newcommand{\fh}{\mathfrak{H}}

 \newcommand{\fp}{\mathfrak{P}}

\newcommand{\vertiii}[1]{{\left\vert\kern-0.25ex\left\vert\kern-0.25ex\left\vert #1 
		\right\vert\kern-0.25ex\right\vert\kern-0.25ex\right\vert}}
\numberwithin{equation}{section}
\newtheorem{theorem}{\quad Theorem}[section]
\newtheorem{lemma}[theorem]{\quad Lemma}
\newtheorem{corollary}[theorem]{\quad Corollary}
\newtheorem{remark}[theorem]{\quad Remark}

\newtheorem{proposition}[theorem]{\quad Proposition}

\newcommand{\nii}{\mathscr{I}}

\newcommand{\ps}{\mathpzc{k}}

\newcommand{\lt}{{L^2(\mathbb{R}^2)}}
\newcommand{\dd}{\;{\rm d}}
\newcommand{\dz}{\;{\rm d}x{\rm d}y}
\newcommand{\ff}{\varphi}
\newcommand{\ee}{{\rm e}}
\newcommand{\ii}{{\rm i}}
\newcommand{\la}{\langle}
\newcommand{\ra}{\rangle}

\newcommand{\N}{\mathbb{N}}

\newcommand{\e}{\varepsilon}

\newcommand{\lam}{\lambda}
\newcommand{\x}{{X_\alpha}}
\newcommand{\xx}{{\dot{X}_\alpha}}
\newcommand{\what}{\widehat}
\newcommand{\nd}{{\partial_x^{-1}}}
\newcommand{\dx}{D_x^{\alpha}}
\newcommand{\mo}{\mu_1}  \newcommand{\po}{{p_1}}
\newcommand{\moo}{\mu_2}  \newcommand{\poo}{{p_2}}
\newcommand{\vr}{{\varrho}}

\newcommand{\al}{\alpha}
\newcommand{\rr}{\mathbb{R}}
\newcommand{\rt}{{\mathbb{R}^2}}
\newcommand{\rrr}{{R_{b,d}}}

\allowdisplaybreaks

\usepackage{authblk}

\title{On the  Kadomtsev–Petviashvili equation  with double-power  	nonlinearities}  
\author[1]{Amin Esfahani$^\ast$}
\author[2]{Steve Levandosky}
\author[3,4]{Gulcin M. Muslu}
\affil[1]{Department of Mathematics, Nazarbayev University, Astana 010000, Kazakhstan}
\affil[2]{Mathematics and Computer Science Department, College of the Holy Cross, Worcester, MA 01610, USA }
\affil[3]{Istanbul Technical University, Department of Mathematics, Maslak 34469, Istanbul, Turkey}
\affil[4]{Istanbul Medipol University, School of Engineering and Natural Sciences, Beykoz 34810, Istanbul, Turkey}

\date{}

\begin{document}
	\maketitle
	\begin{abstract}
		\let\thefootnote\relax\footnotetext{Amin Esfahani (saesfahani@gmail.com, amin.esfahani@nu.edu.kz)}
		\let\thefootnote\relax\footnotetext{Steve Levandosky (slevando@holycross.edu)}
		\let\thefootnote\relax\footnotetext{Gulcin M. Muslu (gulcin@itu.edu.tr,
			gulcinmihriye.muslu@medipol.edu.tr)}
		
		\let\thefootnote\relax\footnotetext{$^\ast$Corresponding Author}
		
	In this paper, we delve into the study of the generalized KP equation, which incorporates double-power nonlinearities. Our investigation covers various aspects, including the existence of solitary waves, their nonlinear stability, and instability. Notably, we address a broader class of nonlinearities represented by $\mu_1|u|^{p_1-1}u+\mu_2|u|^{p_2-1}u$, with $p_1>p_2$, encompassing cases where $\mu_1>0$ and $\mu_1<0<\mu_2$.
		One of the distinct features of our work is the absence of scaling, which introduces several challenges in establishing the existence of ground states. To overcome these challenges, we employ two different minimization problems, offering novel approaches to address this issue. Furthermore, our study includes a nuanced analysis to ascertain the stability of these ground states.
		Intriguingly, we extend our stability analysis to encompass cases where the convexity of the Lyapunov function is not guaranteed. This expansion of stability criteria represents a significant contribution to the field.
		Moving beyond the analysis of solitary waves, we shift our focus to the associated Cauchy problem. Here, we derive criteria that determine whether solutions exhibit finite-time blow-up or remain uniformly bounded within the energy space.
		Remarkably, our study unveils a notable gap in the existing literature, characterized by the absence of both theoretical evidence of blow-up and uniform boundedness. To explore this intriguing scenario, we employ the integrating factor method, providing a numerical investigation of solution behavior. This method distinguishes itself by offering spectral-order accuracy in space and fourth-order accuracy in time.
		Lastly, we rigorously establish the strong instability of the ground states, adding another layer of understanding to the complex dynamics inherent in the generalized KP equation. 
		
		
	\end{abstract}

	
	
	
	
	\section{Introduction}
	The Kadomtsev-Petviashvili (KP) equation is a nonlinear partial differential equation of crucial importance in the study of various physical phenomena. It was originally derived by Kadomtsev and Petviashvili in their seminal work \cite{kp}. The KP equation takes the form:
	\begin{equation}\label{kp}
		(u_t + u_{xxx} + uu_x)_x + \varepsilon u_{yy} = 0, \qquad \varepsilon = \pm 1,
	\end{equation}
	This equation was initially developed to investigate the transverse stability of the solitary wave solution of the Korteweg-de Vries (KdV) equation, which is given by:
	\begin{equation}\label{kdv}
		u_t + u_x + \left(\frac13 - B\right)u_{xxx} + uu_x = 0,
	\end{equation}
	The parameter $B \geq 0$ in Equation \eqref{kdv} represents the Bond number, reflecting different surface tension effects in the context of surface hydrodynamical waves.
	
	Equation \eqref{kp}, when considered with $\varepsilon = -1$, is known as the KP-II equation. It models gravity surface waves in a shallow water channel where the water depth is less than 0.46 cm. This specific form of the KP equation is particularly relevant for the study of these waves and their behavior under varying conditions.
	On the other hand, when Equation \eqref{kp} is considered with $\varepsilon = 1$, it becomes the KP-I equation. This version of the equation describes capillary waves on the surface of liquid or oblique magneto-acoustic waves in a plasma. These capillary waves are characterized by their interactions with surface tension forces and are observed in a variety of physical scenarios.
	The KP equation, with its variations based on the value of $\varepsilon$, continues to be a fundamental tool for understanding the dynamics of various types of waves in different media. Its applications extend to fields such as fluid dynamics, plasma physics, and surface wave phenomena, making it a central equation in the study of nonlinear wave phenomena.

	When considering the effects of higher-order nonlinearity, the KdV equation transforms into the following equation:
	\begin{equation}\label{gardner}
		u_t + u_x + \left(\frac{1}{3} - B\right)u_{xxx} + uu_x - \varsigma u^2u_x = 0.
	\end{equation}
	This equation is known as the Gardner equation, which characterizes large-amplitude internal waves (see \cite{grimshaw} and references therein). The sign of the coefficient $\varsigma$ in \eqref{gardner} varies depending on the physical scenario under examination. In the context of internal waves, this polarity hinges on the stratification \cite{grimshaw}; specifically, it is consistently positive in the scenario of a two-layer fluid. Beyond its significance in plasma physics \cite{rtp, watanabe}, Equation \eqref{gardner} also emerges from asymptotic theory concerning internal waves in a two-layer liquid exhibiting a density jump at the interface \cite{ky, miles}. Remarkably, through a linear transformation, the linear term $u_x$ can be eliminated, leading to an explicit solitary wave solution:
	\begin{equation}\label{gardner solitons}
		u(x, t) = \frac{6AB_0}{1 + R\cosh\left(\sqrt{A}(x - x_0 - AB_0t)\right)},
	\end{equation}
	where $B_0 = \frac{1}{3} - B$ and $R = \pm\sqrt{1 - 6\varsigma A}$ (see \cite{gpps}).
	
	\begin{figure}[ht]
		\begin{center}
			\scalebox{0.2}{\includegraphics{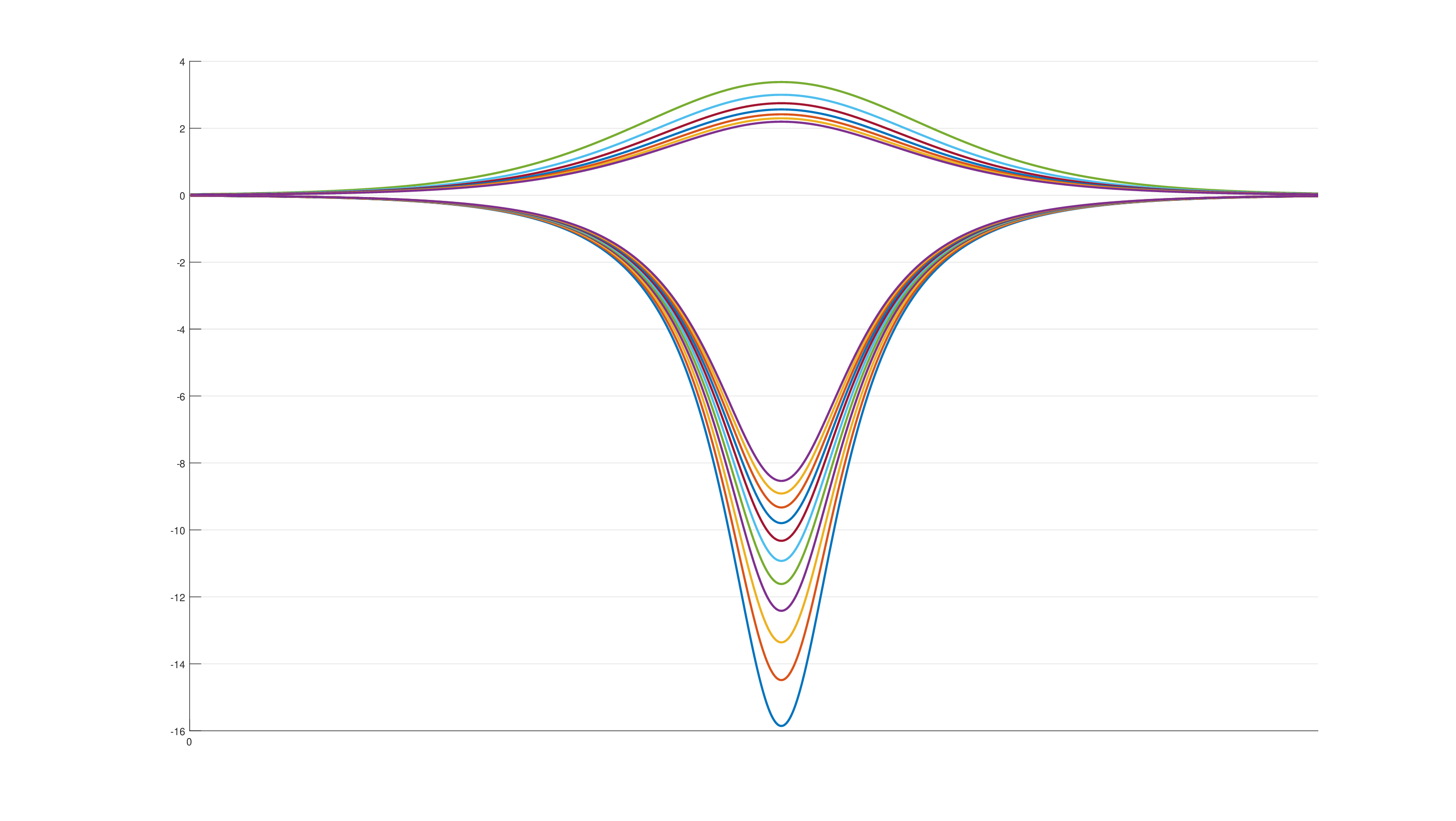}}
			\caption{Solitary waves of the Gardner equation \eqref{gardner} given by \eqref{gardner solitons} with various values of $\varsigma$. }\label{gard}
	\end{center}\end{figure}
	
	Similar to \eqref{kp}, a high-dimensional Gardner equation was derived in \cite{osb} to describe the propagation of weakly nonlinear and weakly dispersive dust ion acoustic waves in a collisionless unmagnetized plasma. This plasma consists of warm adiabatic ions, static negatively charged dust grains, nonthermal electrons, and isothermal positrons.
	
	This paper focuses on the following generalized Kadomtsev–Petviashvili (KP) equation with double-power nonlinearities \cite{osb,sbd,skt}:
	\begin{equation}\label{main-fifthkp}
		\left(u_t-D_x^{2\al} u_x+(f(u))_x\right)_x+ \e u_{yy}=0, \qquad \e=\pm1,
	\end{equation}
	where $u=u(x,y,t)$ is a real-valued function, and $f(u)=\mo f_1(u)+\moo f_2(u)$ where $\mo,\moo\in\rr$, and $f_j(u)=|u|^{p_j-1}u$ ($j=1,2$).
	
	The operator $D^{2\al}_x$ with $\al>0$ represents the Riesz potential of order $-2\al$ in the $x$-direction. It is defined by the usual Fourier multiplier operator with the symbol $|\xi|^{2\al}$. When $\al=0$, the equation \eqref{main-fifthkp} reduces to:
	\begin{equation}\label{dispersionless}
		(u_t-u_x+(f(u))_x)_x + \e u_{yy}=0,
	\end{equation}
	where the operator $\nd$ is defined via the Fourier transform $\widehat{\nd}=(i\xi)^{-1}$. This equation models the propagation of short pulses in some media and is known as the Khokhlov–Zabolotskaya (or the dispersionless) equation (see \cite{roud} and references therein).
	It is noteworthy that when $\alpha=1/2$, Equation \eqref{main-fifthkp} becomes the relevant KP version of the Benjamin-Ono equation. This version has been derived from the two-fluid system in the weakly nonlinear regime (see \cite{choicam}). It was also derived in \cite{abs,cgh} for long weakly nonlinear internal waves in stratified fluids of large depth. When $\alpha=2$, \eqref{main-fifthkp} arises as a two-dimensional model for capillary-gravity water waves in the regime of critical surface tension when the Bond number is close to the critical value $1/3$ (see \cite{abst,nonlinearity,karbel}), and can be regarded as a two-dimensional version of the Kawahara equation.
	
	Hence, equation \eqref{main-fifthkp} bears some connection with the full dispersion KP equation derived in \cite{lann-1} and discussed in \cite{lann-2} as an alternative model to KP (with fewer unphysical shortcomings) for gravity-capillary surface waves in the weakly transverse regime. It is known that the KP-I equation features Zaitsev traveling waves, which are localized in the $x$ direction and periodic in $y$. Suitable transformations of parameters produce solutions that are periodic in $x$ and localized in $y$. Due to its integrability properties, the KP-I equation possesses a localized, finite energy, explicit solitary wave called the lump \cite{mzbim}:
	\[
	\psi(x-ct,y)=
	\frac{8c\left(1-\frac c3(x-ct)^2+\frac{c^2}{3}y^2\right)}{\left(1+\frac c3(x-ct)^2+\frac{c^2}{3}y^2\right)^2}.
	\]
	
	Notice that equation \eqref{main-fifthkp} formally conserves energy (Hamiltonian) denoted as $E$, along with momentum quantities $M$ and $\mathbb{P}$, which are defined as follows:
	\begin{equation}\label{mass}
		M(u)=\frac12\|u\|_\lt^2,
	\end{equation}
	\[
	\mathbb{P}(u)=\int_\rt u(\nd u_y)\dz,
	\]
	\begin{equation}\label{energy}
		E(u)=\frac{1}{2}\int_\rt\left(|\dx u|^2-\e|\nd u_y|^2\right)\dz-K(u),
	\end{equation}
	where 
	\[
	K(u)=\mu_1K_1(u)+\mu_2K_2(u)= \int_\rt F (u)\dz=\mu_1\int_\rt F_1(u)\dz+\mu_2\int_\rt F_2(u)\dz,
	\]
	and $F_j$ for $j=1,2$ represents the primitive function of $f_j$ with $F_j(0)=0$.
	
	In this paper, our primary focus encompasses the investigation of both the existence and stability of solitary waves of \eqref{main-fifthkp}, as well as the analysis of the long-term behavior of solutions derived from \eqref{main-fifthkp}. A solitary wave is understood to be a solution in the form of $u(x,y,t)=v(x-ct,y)$. It is important to recognize that based on the concepts presented in \cite{dbs-0}, \eqref{main-fifthkp} does not exhibit nontrivial solitary wave solutions when $\e=+1$. Consequently, in the subsequent analysis, we consider \eqref{main-fifthkp} with the assumption that $\e=-1$. In this case, the solution $v$ satisfies the equation:
	\begin{equation}\label{gkp}
		\left(cv+D_x^{2\al} v- f(v) \right)_{xx}  + v_{yy}=0.
	\end{equation}
	It is worth noting that well-posedness results concerning the Cauchy problem associated with \eqref{main-fifthkp} have been established in specific scenarios. For instance, when $\mu_2=0$ and $p_1=2$, various works such as \cite{grun,lps,ylhd,sanwaschi} have provided insights into well-posedness. In the context of homogeneous nonlinearity with $\mu_2=0<\mu_1$ and $\al=1,2$, de Bouard and Saut \cite{dbs-0} introduced a minimization procedure for the associated norm of \eqref{kp}, subject to the constraint $K(u)=\lam>0$. They employed the concentration-compactness principle to establish the existence of solitary wave solutions. Through appropriate scaling of the Lagrange multiplier, they demonstrated that the minimizer satisfies \eqref{main-fifthkp}. However, in the presence of nonhomogeneous nonlinearity where $\mu_1,\mu_2\neq0$ this method becomes inapplicable due to the absence of scaling invariance. 
		To overcome this challenge, we divide our existence results into parts. In the case $\mo>0$ with $\po>\poo$, we employ the Pankov-Nehari manifold method \cite{pankov}. While a similar result can potentially be achieved using the Mountain-pass argument (as detailed in \cite[Chapter 7]{willem}) under certain constraints on $f$, our approach offers an additional advantage: the stability analysis of solitary waves can be inferred. 
	
In the case where $\mo < 0 < \moo$ with $\po > \poo$, the situation is significantly different, and the previously mentioned argument is not applicable. This is because the leading nonlinearity term, $\mo |u|^{\po-1}u$, is a positive term in the energy. To address this challenge, we introduce a new minimization problem (see \eqref{new-min-highlow}). While this new minimization problem helps establish the existence of ground states, it complicates the proof of stability. To address this challenge, we introduce a novel minimization problem (see \eqref{new-min-highlow}) that helps establish the existence of ground states, although it complicates the proof of stability. This novel approach, distinct from previous methods, is inspired by the work of Colin and Ohta \cite{colin}, allowing us to achieve another stability result without relying on the convexity condition of the Lyapunov function (see Theorem \ref{stab-theopq}).
		An important observation is that there exists a natural scaling invariance associated with \eqref{main-fifthkp} when $\mu_2=0$. More specifically, the scaling transformation
	\begin{equation}\label{scaling}
		u_\lam(x,y,t)=\lam^{\frac{2\al}{p_1-1}}u(\lam^{2\al+1}x,\lam y,\lam^{\al+1}t),\quad\lam>0
	\end{equation}
	preserves the form of \eqref{main-fifthkp}. This implies that $u_\lam$ is also a solution provided that $u$ is a solution of \eqref{main-fifthkp}. In contrast, when dealing with combined nonlinearities, no such scaling exists that leaves \eqref{main-fifthkp} invariant. This lack of invariance adds complexity to the systematic study of the Cauchy problem associated with \eqref{main-fifthkp}. A similar situation is encountered in the context of the following nonlinear Schrödinger equation with combined power-type nonlinearities:
	\begin{equation}\label{nls}
		\ii u_t+\Delta u=\mu_1 |u|^{p_1-1}u+\mu_2|u|^{p_2-1}u.
	\end{equation}
	Inspired by the comprehensive study presented by Tao et al. in \cite{tvs}, our focus in this paper is on establishing sharp criteria concerning the dichotomy of global existence versus finite-time blow-up. This includes obtaining insights into the local and global well-posedness, finite-time blow-up in weighted spaces, asymptotic behavior in both energy and weighted spaces, and the scattering versus blow-up phenomenon for specific cases of \eqref{nls}. Building upon the methodologies presented in \cite{cmz,mzz}, we aim to provide novel insights into understanding the intricate dynamics of \eqref{nls}. Drawing inspiration from these works, our aim is to develop a deeper understanding of the ground states of \eqref{main-fifthkp}, their stability, and their behavior in the face of perturbations and nonlinear interactions.

	The structure of the paper unfolds as follows: Section \ref{sect-exist} delves into the examination of the existence of ground states within the context of \eqref{main-fifthkp}. The subsequent section, Section \ref{sect-stability}, probes into the stability of these ground states. In Section \ref{sect-bound-blow}, we derive the criteria dictating whether the solutions remain uniformly bounded in the energy space or eventually blow up. The exploration of strong instability occupies Section \ref{sect-strong-inst}.
	Furthermore, in Section \ref{sect-numerical-s}, we present an intricately devised numerical approach that amalgamates the Fourier pseudo-spectral method with the integration factor method. This approach is proficiently employed to solve the generalized KP equation featuring double-power nonlinearities. The accuracy of this numerical method is meticulously monitored through both mass conservation error and Fourier coefficients. Significantly, the method exhibits a spectral-order precision in space and a fourth-order precision in time.
	The investigation then proceeds to the dynamic evolution of solutions for the generalized KP equation involving supercritical, subcritical, and critical nonlinearities. Particular attention is devoted to the analysis of a gap within Section \ref{sect-bound-blow}. This gap, wherein neither a blow-up nor a uniform boundedness outcome has been theoretically established, is a focal point of inquiry.


	\subsection*{Notation}
	We  denote the $L^2(\rr^2)$-inner product by $\langle\cdot,\cdot\rangle$.  We shall also denote by $\widehat\ff$ the Fourier transform of $\ff$, defined as
	\[
	\widehat\ff(\zeta)=\int_{\rr^2}\;\ff(z)e^{-\ii z\cdot \zeta}\;\dd z.
	\]
	For $s\in\rr$, we denote by
	$H^s\left(\rr^2\right)$, the nonhomogeneous Sobolev space defined by
	\[
	H^s\left(\rr^2\right)=\left\{\ff\in\mathscr{S}'\left(\rr^2\right)\;:\;\|\ff\|_{H^s\left(\rr^2\right)}<\infty\right\},
	\]
	where
	\[
	\|\ff\|_{H^s\left(\rr^2\right)}
	=\left\|\left(1+|\zeta|^2\right)^\frac{s}{2}\widehat{\ff}(\zeta)\right\|_{L^2\left(\rr^2\right)},
	\]
	and $\mathscr{S}'\left(\rr^2\right)$ is the space of tempered distributions. Let ${X_\al}$ be the closure of $\partial_x(C_0^\infty(\rr^2))$ for the norm
	\begin{equation}
		\|\ff_x\|_{X_\al}^2= \|\dx\ff_x\|_{L^2(\rr^2)}^2+\|\ff_y\|_{L^2(\rr^2)}^2
		+\|\ff_{x}\|_{L^2(\rr^2)}^2.
	\end{equation}
	The homogeneous space $\xx$ is defined by the norm
	\begin{equation}
		\|\ff\|_{\xx}^2= \|\dx\ff\|_{L^2(\rr^2)}^2+\|\nd\ff_y\|_{L^2(\rr^2)}^2.
	\end{equation}
	Let
	\[
	X^s=\left\{u\in H^s(\rr^2);\;\left(\xi^{-1}\what{u}(\xi,\eta)\right)^\vee\in H^s(\rr^2)\right\},
	\]
	with the norm
	\[
	\|u\|_{X^s}=\|u\|_{H^s(\rr^2)}+\left\|\left(\xi^{-1}\what{u}\right)^\vee\right\|_{H^s(\rr^2)},
	\]
	where $`\vee$' is the Fourier inverse transform.

	\section{Existence of Ground States}\label{sect-exist}
	In this section, we establish the existence of traveling wave solutions of equation \eqref{main-fifthkp}. We define the functional $I(u)$  as follows:
	\[
	I(u)=\frac{1}{2}\int_\rt\left(cu^2+|\dx u|^2+|\nd u_y|^2\right)\dz.
	\]
	A function $u$ is a weak solution of \eqref{gkp} if and only if \(I'(u) = K'(u)\), or in other words, if \(u\) is a critical point of the functional \(S = I - K\). In \cite{dbs-1}, critical points of \(S\) were shown to exist for homogeneous nonlinearities by demonstrating the existence of minimizers of \(I(u)\) under the constraint \(K(u) =\lambda > 0\). These minimizers satisfy \(I'(u) =\theta K'(u)\), where \(\theta\in\mathbb{R}\) is a Lagrange multiplier. By homogeneity of the nonlinearity, this multiplier could be scaled out to obtain a solution of \eqref{gkp}.
	
	The same approach can be applied to prove the existence of minimizers for the same variational problem involving more general nonlinearities that satisfy suitable growth conditions. However, for non-homogeneous nonlinearities, the Lagrange multiplier cannot be scaled out. In such cases, the dependence of the Lagrange multiplier \(\theta\) on the constraint parameter \(\lambda\) can become quite intricate, especially when \(\mu_1 \mu_2 < 0\). Instead, we consider an alternative constrained minimization problem. This is driven by the observation that if \(u\) is a critical point of \(S\), then \(P(u) = 0\), where
	\[
	P(u)=\langle S'(u),u\rangle=2I(u)-N(u)
	\]
	and
	\[
	N(u)=\int_{\mathbb{R}^2}\left(\mu_1 |u|^{p_1+1}+\mu_2 |u|^{p_2+1}\right)\mathrm{d}x\mathrm{d}y.
	\]
	We now seek solutions of \eqref{gkp} that minimize the action \(S\) among nontrivial solutions by solving the minimization problem
	\begin{equation}\label{minimum}
		m=\inf_{u\in N_0}S(u),
	\end{equation}
	where
	\[
	N_0=\{u\in X_\alpha,\;u\not\equiv0,\;P(u)=0\}.
	\]
	These solutions satisfy \(S'(u) =\theta P'(u)\) for some \(\theta\), which, due to the homogeneity of the constraint \(P(u) = 0\), can be demonstrated to be zero. Consequently, \(S'(u) = 0\), and $u$ is a weak solution of \eqref{gkp}. A ground state is a solution of \eqref{gkp} that minimizes the action \(S\) among all nontrivial solutions of \eqref{gkp}. The main result of this section is the following. 
	
	\begin{theorem}\label{theorem-exist}
		Let $\mo>0$ and $1<\poo<\po<2^\ast$.	Suppose that $\{u_n\}$ is a minimizing sequence  of \eqref{minimum}. Then there exists a subsequence, renamed by the same, a sequence $\{z_n\}\subset\rt$, and $u\in X_\al$ such that $u_n(\cdot-z_n)\to u$ strongly in $X_\al$, $P(u)=0$ and $S(u)=m$.
	\end{theorem}
	
	\begin{remark}
		By a slight modification, one can extend Theorem \ref{theorem-exist} to the cases $f(u)=\mu_1f_1(u)+\mu_2f_2(u)$ where $f_1$ and $f_2$ are homogeneous of degree $p_1$ and $p_2$, respectively, $p_1>p_2$, and there exist $u\in X_\al$ such that $\mu_1K_1(u)>0$. This includes for example even nonlinearities $f(u)=\mu_1|u|^{p_1}+\mu_2|u|^{p_2}$ and mixed-parity nonlinearities such as $f(u)=\mu_1|u|^{p_1}+\mu_2|u|^{p_2-1}u$ and $f(u)=\mu_1|u|^{p_1-1}u+\mu_2|u|^{p_2}$ with $\mu_1>0$. See \cite{esfahani-levan-dpde}.   
	\end{remark}
	\begin{figure}[ht]
		\begin{center}
			\scalebox{0.5}{\includegraphics{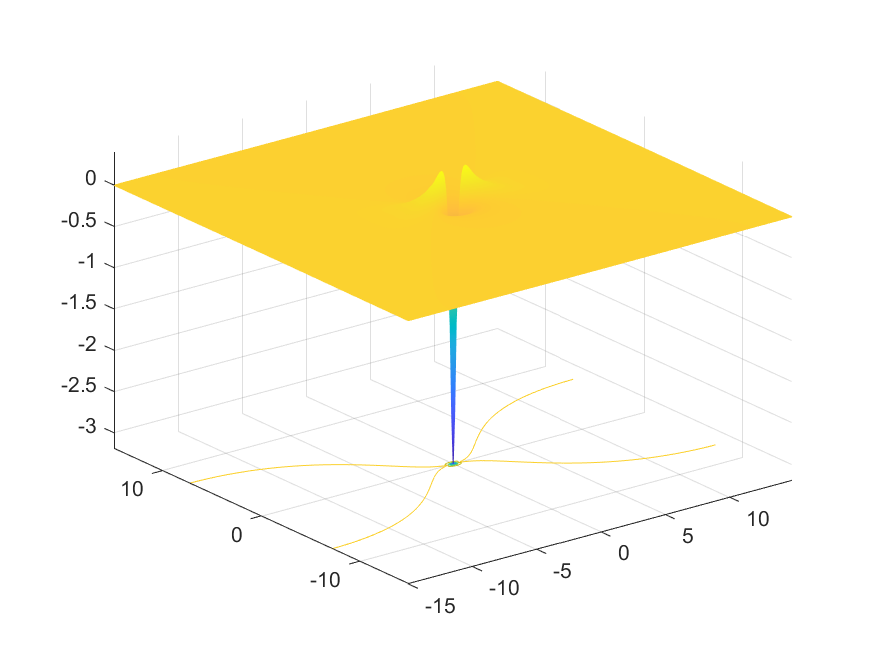}}
			\caption{Numerical solitary wave of  \eqref{main-fifthkp} with $\al=1$, $f(u)=|u|^3-2u^2$. }\label{nume-slt-1}
	\end{center}\end{figure}

	To prove Theorem \ref{theorem-exist} we first recall the well-known result from \cite{B:besov,liuwang} that provides an embedding for any \(q\in[2, 2^*]\):
	\begin{equation}\label{embed}
		X_\alpha\hookrightarrow L^q(\mathbb{R}^2),
	\end{equation}
	where
	\[
	2^*=\begin{cases}
		\infty^-&\alpha\geq2,\\
		\frac{3\alpha+2}{2-\alpha}&\alpha<2.
	\end{cases}
	\]
	Moreover the embedding 	$X_\al\hookrightarrow L^q_{\rm loc}(\rt) $ is compact if $q\in(2,2^	\ast)$.
	
	Related to \eqref{embed}, we recall  the following anisotropic Sobolev-type inequality  from the results of  Besov,   Il'in, and Nikolskii \cite{B:besov} (see also \cite[Proposition 2.2]{Borluk-Bruell-Nilsson} and \cite[Lemma 1.1]{lps})), together with an interpolation:
	\begin{equation}\label{bin}
		\|u\|_{L^{p+1}(\rr^2)}^{p+1}\leq \rho_p 
		\|u\|_\lt^{2c_p}
		\|\dx u\|_\lt^{\frac{p-1}{\al}}
		\|\nd u_y\|_\lt^{\frac{p-1}{2}},\quad u\in X_\al,
	\end{equation}
	provided
	\[
	c_p=\frac{3\al+2+p(\al-2)}{4\al}>0.
	\] 
	See also \cite[Lemma 2.1]{blt}.   
	
	Next we set
	\[
	N_\sigma=\{u\in X_\al,\;u\not\equiv0,\; P(u)=\sigma\}
	\]
	and 
	define
	\begin{equation}
		m_\sigma=\inf_{u\in N_\sigma} S_0(u),
	\end{equation}
	where\[
	S_0=S-\frac{1}{\gamma+1}P,
	\]
	$\gamma=\po$ if $\mu_2<0$, and $\gamma=\poo$ if $\moo>0$.
	Hereafter in this section, we consider the case $\moo<0$ in the proof of Theorem \ref{theorem-exist}. The argument also works for the case $\moo>0$ and is simpler. 
	\begin{lemma}\label{lem23}
		For any $\sigma\in\rr$, the set $N_\sigma$ is nonempty. Moreover, $N_\sigma$ is bounded away from zero if $\sigma\leq0$.
	\end{lemma}
	\begin{proof}
		Let $\sigma\in\rr$ and fix any nonzero $u\in X_\al$. For $C\in\rr$, we define
		\[
		u_C(x,y)=u(x,Cy).
		\]
		Then it is easy to see that  
		\[
		\lim_{C\to+\infty}	P(u_C)=+\infty.
		\]
		Hence, there exists $C>0$ such that $P(u_C)>\sigma$. On the other hand, since $K_1(u_C)>0$ and $p_1>p_2>1$,
		\[
		\lim_{A\to\infty}P(Au_C)=-\infty.
		\]
		Thus, there exists $A>0$ such that $P(Au_C)=\sigma$. This proves $N_\sigma$ is not empty.
		
		Now, let $\sigma\leq0$ and $u\in N_\sigma$. Then, we have from \eqref{embed} that
		\[
		0\geq P(u)\geq C_1\|u\|_{X_\al}^2-C_2(\|u\|_{X_\al}^{\po+1}+\|u\|_{X_\al}^{\poo+1}),
		\]
		for some constants $C_1,C_2>0$. Therefore, $\|u\|_{X_\al}\geq C>0$, where $C$ is a constant depending only on $\po$, $\poo$, $\mo$ and $\moo$.
	\end{proof}
	
	In the following lemma, we show that $m_\sigma$ is decreasing on $\sigma\leq0$ which shows the strict subadditivity condition of $m_\sigma$.
	\begin{lemma}
		We have $m_\sigma\geq0$ for all $\sigma\in\rr$. Moreover,   $m_\sigma$ positive and strictly decreasing on $(-\infty,0]$.
	\end{lemma}
	\begin{proof}
		Since $\frac{1}{\po+1}N(u)\geq K(u)$, then $S_0(u)\geq0$ for all $u\in X_\al$. This means that $m_\sigma\geq0$ for any $\sigma\in\rr$. If $\sigma\leq 0$, then by Lemma \ref{lem23}, there exits $C$ such that $\|u\|_{X_\al}\geq C$. Then 
		\[
		S_0(u)\geq C_{p}C^2
		\]
		for any $u\in N_\sigma$. Hence, $m_\sigma>0$.
		
		Next, for $\sigma<\sigma_2\leq0$, we can consider $u\in N_\sigma$ such that $S_0(u)<2m_{\sigma_2}$. We will get the desired inequality $m_{\sigma_1}>m_{\sigma_2}$, if no such $u$  exists. First note that $N(u)>0$ so if we define $g(r)=P(ru)=r^2(2I(u)-r^{-2}N(ru))$ then since $\left<N'(v),v\right>\geq(\gamma+1)N(v)$ for all $v\in X_\al$ it follows that $\frac{d}{dr}(r^{-2}N(ru))>0$ and thus there exists at most one $r>0$ such that $g(r)=\sigma_2$. Hence by the reasoning in the proof of Lemma \ref{lem23}, there exists a unique $r_u<1$ such that $P(r_uu)=\sigma_2$. We show that there exists $r_0<1$, independent of $u$, such that $r_u\leq r_0$. By using 
		\[
		S_0(u)\geq\left(1-\frac2{\po+1}\right)I(u)\geq
		\left(1-\frac2{\po+1}\right)\|u\|_{X_\al}^2,
		\]
		we obtain that there is $C>0$ such that $\|u\|_{X_\al}\leq C$ for all $u$ such that $S_0(u)<2m_{\sigma_2}$. Hence, it follows immediately that $g'(r)\geq- C_0$ for all $r<1$, where $g(r)=P(ru)$. Then we have by integrating from $r_u$ to 1 that 
		\[
		\sigma_1-\sigma_2=g(1)-g(r_u)\geq -C_0(1-r_u).
		\]
		and thereby $r_u\leq1-\frac{\sigma_2-\sigma_1}{C_0}=:r_0$.
		Now as $P(ru)\leq0$ for $r_u\leq r\leq1$, then similar to the proof of Lemma \ref{lem23}, it holds that $\|ru\|_{X_\al}\geq C$. Consequently, $h'(r)\geq2(1-\frac2{\po+1}) C^2$ for all $r_u\leq r\leq 1$, where $h(r)=S_0(ru)$. It follows that $S_0(u)-S_0(ru)\geq 2(1-\frac2{\po+1}) \frac{C^2}{C_0}(\sigma_2-\sigma_1)$. The fact $m_{\sigma_2}\leq S_0(r_uu)$ reveals that $$S_0(u)\geq m_{\sigma_2}2\left(1-\frac2{\po+1}\right) \frac{C^2}{C_0}(\sigma_2-\sigma_1)$$
		for all $u\in N_\sigma$ such that $S_0(u)\leq 2m_{\sigma_2}$. Therefore, $m_{\sigma_1}>m_{\sigma_2}$, and the proof is complete.
	\end{proof}

	\begin{figure}[ht]
		\begin{center}
			\scalebox{0.18}{\includegraphics{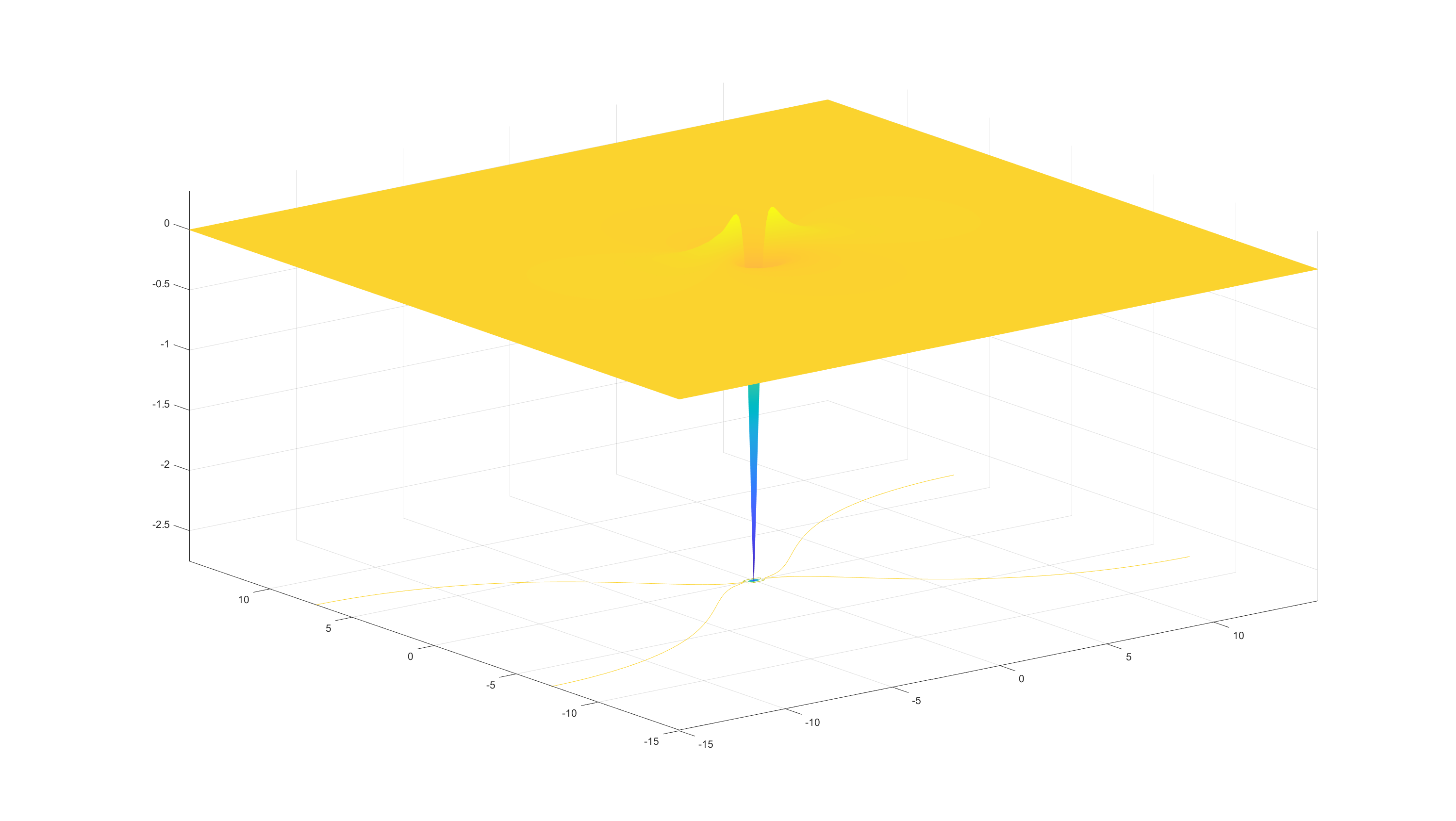}   	\ \includegraphics{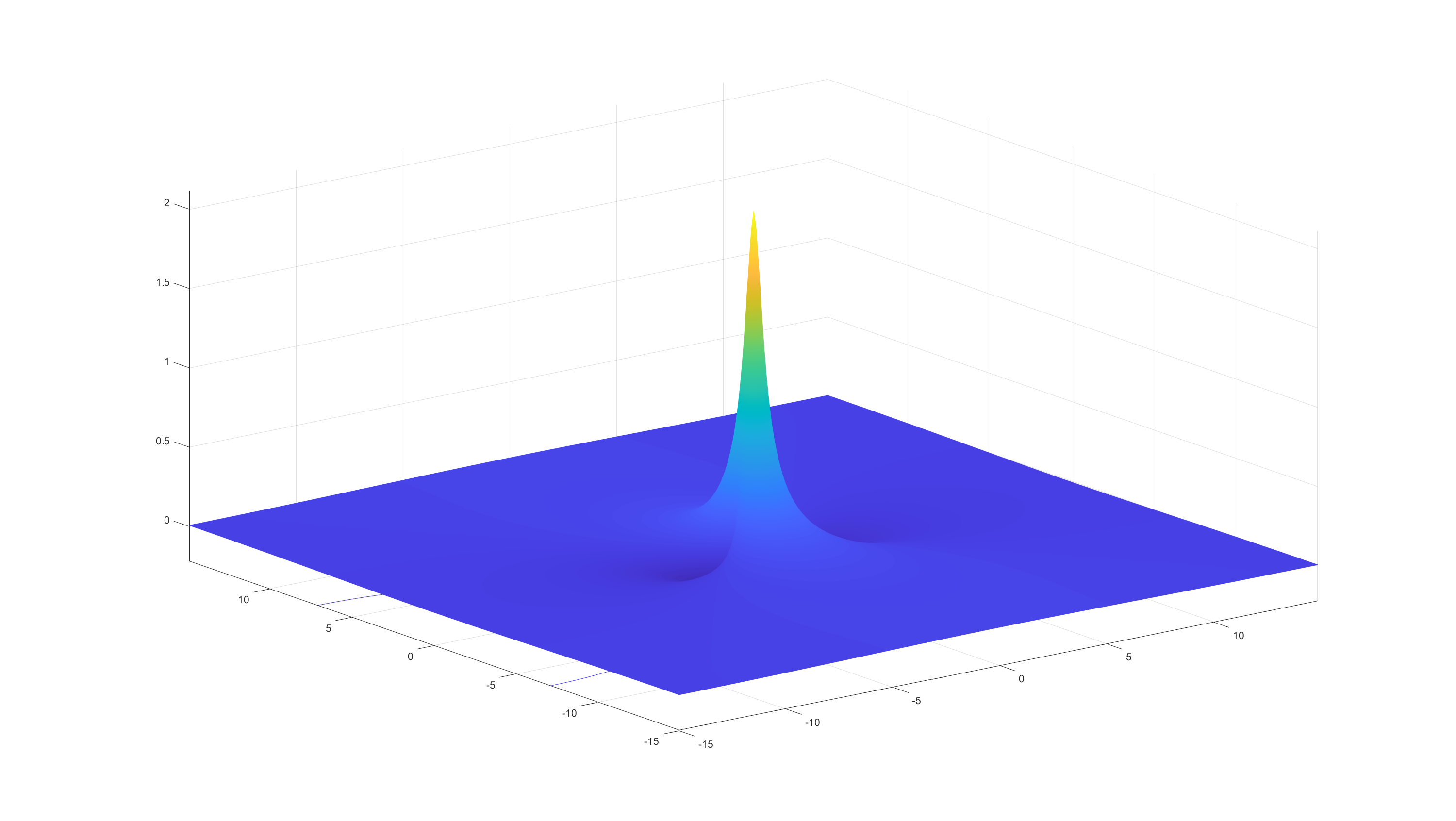}}
			\caption{Numerical solitary waves of  \eqref{main-fifthkp} with $\al=1$, $f(u)=u^4-2u^2$ (left) and $f(u)=u^4+2u^2$ (right). } \label{nume-slt-2}
	\end{center}\end{figure}
	
		\begin{figure}[ht]
		\begin{center}
			\scalebox{0.18}{\includegraphics{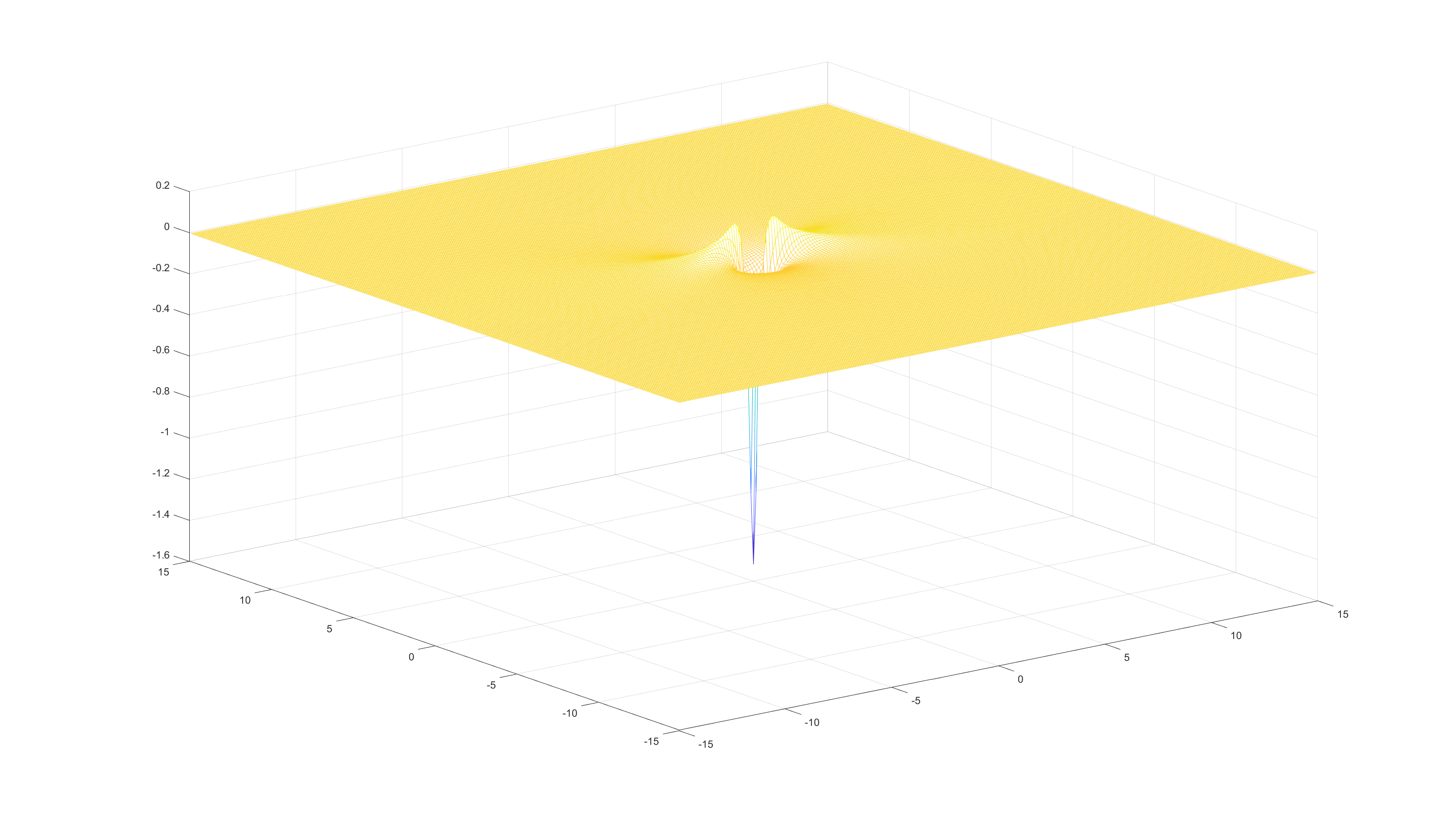} \ \includegraphics{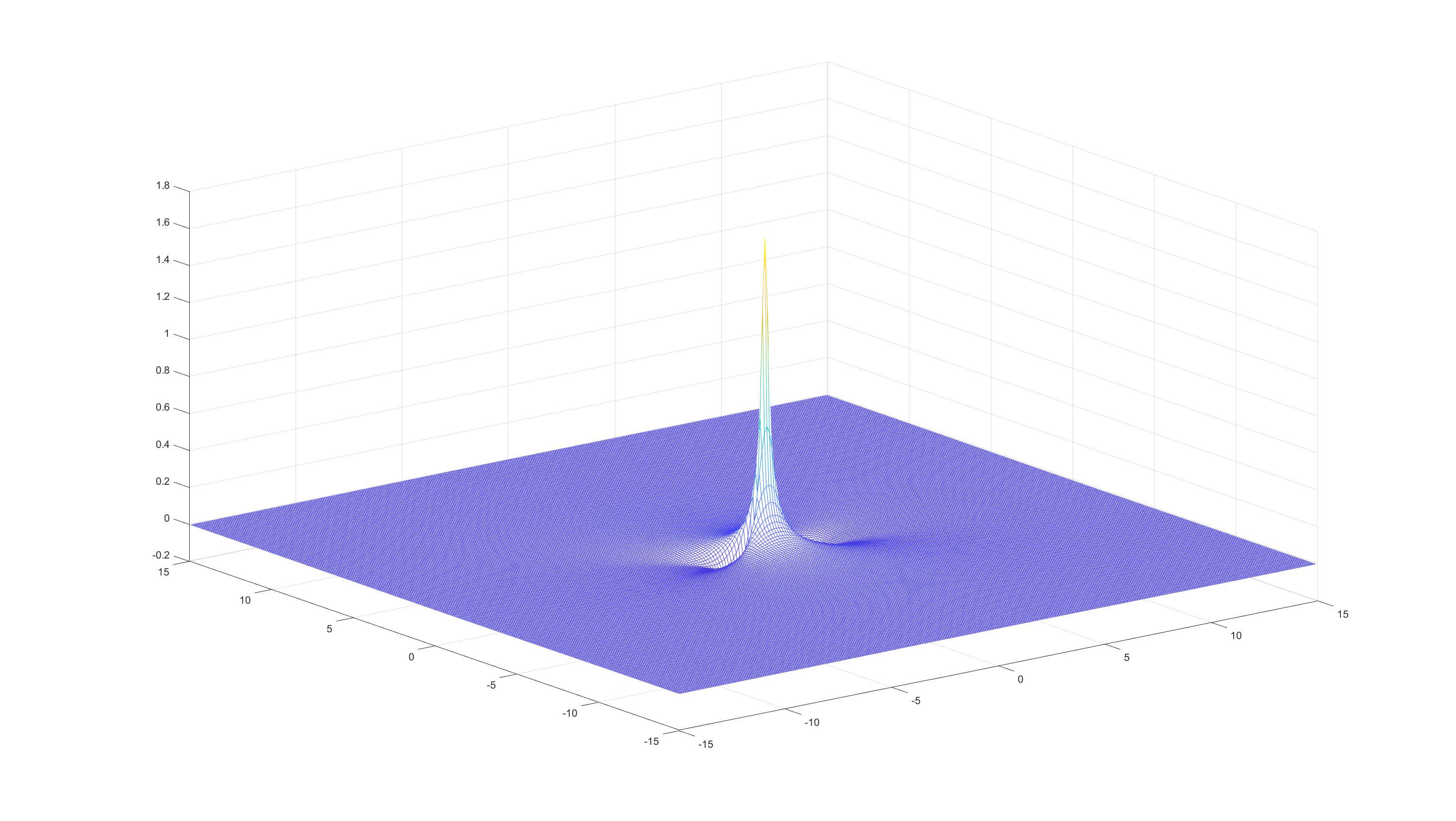}}  				\caption{Numerical solitary waves of  \eqref{main-fifthkp} with $\al=1$, $f(u)=-2|u|^3u+|u|u$ (left) $f(u)=-\frac32|u|^3u+|u|u$ (right). } \label{nume-slt-212}
	\end{center}\end{figure}
	
	\begin{proof}[Proof of Theorem \ref{theorem-exist}]
		Let $\{u_n\}$ be a minimizing sequence. Then it is bounded. Indeed, since $S(u_n)$ and $P(u_n)$ are bounded, $S_0(u_n)$ is bounded, and because of\[
		S_0(u_n)\gtrsim\left(1-\frac{2}{\po+1}\right)I(u_n),
		\]
		the sequence $\{u_n\}$ is bounded in $\x$. Now, since $m>0$, the sequence $\{S(u_n)\}$ is uniformly bounded below for large enough $n$, and then $\{u_n\}$ is uniformly  bounded below in $X_\al$. To get the minimizer function $u\in X_\al$, we apply the concentration-compactness principle (see \cite{AMC,lps-0}) for the sequence $\chi_n=|u_n|^2+|\dx u_n|^2+|\nd(u_n)_y|^2$, and $\lim_{n\to\infty}\int_\rt\chi_n\dz=L>0$, up to a subsequence (see also the proof of Theorem \ref{criticalcase-normalized}). The sequence $\chi_n$ can be normalized such that $\int_\rt\chi_n\dz=L$. We rule out the vanishing and dichotomy cases. If the vanishing case occurs, then by an argument similar to \cite{dbs-0,dbs-1}, $u_n\to0$ for any $2<q<2^\ast$ as $n\to\infty$. This implies that $K(u_n)\to0$ as $n\to\infty$. Hence, $I(u_n)\to0$ and $S(u_n)\to0$ as $n\to\infty$, because $P(u_n)\to0$. This contradicts $S(u_n)\to m>0$. Suppose that the dichotomy case occurs. Then, there are the bounded sequences   $\{v_n\},\{w_n\}\subset X_\al$ such that
		\[
		\begin{split}
			&		\lim_{n\to\infty}\|u_n-v_n-w_n\|_{\x}=0,\\
			&\lim_{n\to\infty} K(u_n)-K(v_n)-K(w_n)=0 
		\end{split}
		\]
		and $N(u_n)-N(v_n)-N(w_n)\to0$ as $n\to\infty$. These imply that $P(u_n)-P(v_n)-P(w_n)\to0$ and $$S_0(u_n)-S_0(v_n)-S_0(w_n)\to0$$ as $n\to\infty$. Suppose (by extracting subsequences if necessary) that $\sigma_1=\lim_{n\to\infty}  P(v_n)$ and $\sigma_2=\lim_{n\to\infty} P(w_n)$. Then $\sigma_1+\sigma_2=0$. If $\sigma_1>0$, then there is $n_0\in\N$ such that $\sigma_{2,n}=P(w_n)<\sigma_2/2$ for all $n\geq n_0$. Since $m_\sigma$ is strictly decreasing in $\sigma$, so 
		$S_0(w_n)\geq m_{\sigma_{2,n}}>m_{\sigma_2/2}$ for all $n\geq n_0$. As $S_0(v_n)\geq0$  for all $n$, then 
		\[
		S_0(v_n)+S_0 w_n\geq\sigma_{\frac12\sigma_2}
		\]
		for all $n\geq n_0$. It is concluded to the contradiction
		\[
		m=\lim_{n\to\infty} S(u_n)=\lim_{n\to\infty}S_0(u_n)=\lim_{n\to\infty}\left(S_0(v_n)+S_0(w_n)\right)\geq\sigma_{\frac12\sigma_2}>m.
		\]
		A similar contradiction holds for the case $\sigma_1<0$. Next, we consider  $\sigma_0=\sigma_2=0$. In this case, we have from the coercivity of $I$ that $I_1,I_2>0$, where $I_1=\lim_{n\to\infty}I(v_n)$ and $I_2=\lim_{n\to\infty}I(w_n)$. Then, for any $\epsilon>0$  there is $n_0\in\N$ such that $I(v_n)<2N(v_n)(1+\epsilon)^{\po-1}$
		and $I(w_n)<2N(w_n)(1+\epsilon)^{\po-1}$ for all $n\geq n_0$. Since $\{v_n\}$ is bounded in $X_\al$, then $K(v_n)-K(\theta v_n)\leq C(\theta-1)$ for all $n$ and $\theta>1$, and $N(v_n)-N(\theta v_n)\leq C(\theta-1)$ for some $C>0$. If $P(v_n)>0$, then $P(\theta v_n)=0$ for some $\theta<\left(\frac{2I(v_n)}{N(v_n)}\right)^{\frac1{\po-1}}$, because we have for $\theta>1$ that $P(\theta v_n)\leq2\theta^2I(v_n)-\theta^{\po+1}N(v_n)$. A straightforward computation shows for some $C_1>0$ that
		\[
		m\leq S_0(\theta v_n)
		\leq S_0(v_n)+C(\theta-1)\leq S_0(v_n+C_1\epsilon).
		\]
		Hence, $S_0(v_n)\geq m-C_1\epsilon$. This inequality also trivially holds if $P(v_n)\leq0$, and also for $S_0(w_n)$. Therefore, we obtain that
		\[
		S_0(v_n)+	 S_0(w_n)\geq 2m-2C_1\epsilon
		\]
		for all $n\geq n_0$. This obviously shows that $m\lim_{n\to\infty}S(u_n)\geq 2m-2C_1\epsilon$, and consequently, $m\geq 2m$ which is a contradiction. Thus the dichotomy does not occur. Finally, the compactness case should occur. So, by using Lemma 3.3 in \cite{dbs-1}, there is a sequence $\{z_n\}\subset\rt$ and some $u\in X_\al$ such that $\tilde u_n=u_n(\cdot-z_n)\rightharpoonup u$  in $X_\al$ and  $\tilde u_n(\cdot-z_n)\to u$ in $L^q_{\rm loc}(\rt)$ for any $q\in(2,2^\ast)$, and thereupon the strong convergence in $L^q(\rt)$ is deduced. Hence $K(\tilde u_n)=K(u_n)\to K(u)$  and $N(\tilde  u_n)=N(u_n)\to N(u)$ as $n\to\infty$. The weak lower semicontinuity of $I$ shows that
		\[
		S(u)+K(u)=I(u)\leq \liminf_{n\to\infty}I(u_n)=\liminf_{n\to\infty}(S(u_n)+K(u_n))=m+K(u).
		\]
		By a similar computation, $P(u)\leq0$ and $S(u)\leq m$. But, we have $S_0(u)<m_\sigma$ for all $\sigma<0$, and thereby $P(u)=0$. Indeed if $P(u)=\sigma<0$ then we get the contradiction $m_\sigma\leq S_0(u)<m_\sigma$. hence, $u$ achieves the minimum $m$. Moreover, as $\lim_{n\to\infty}S(\tilde u_n)=m=S(u)$, then $\lim_{n\to\infty} I(\tilde u_n)=\lim_{n\to\infty} (S(\tilde u_n)+K(\tilde u_n))=I(u)$. Consequently, we obtain from $\tilde u_n\rightharpoonup u$ in $X_\al$ that $\tilde u_n\to u$ in $X_\al$.
	\end{proof}
	
	\begin{figure}[ht]
		\begin{center}
			\scalebox{0.16}{\includegraphics{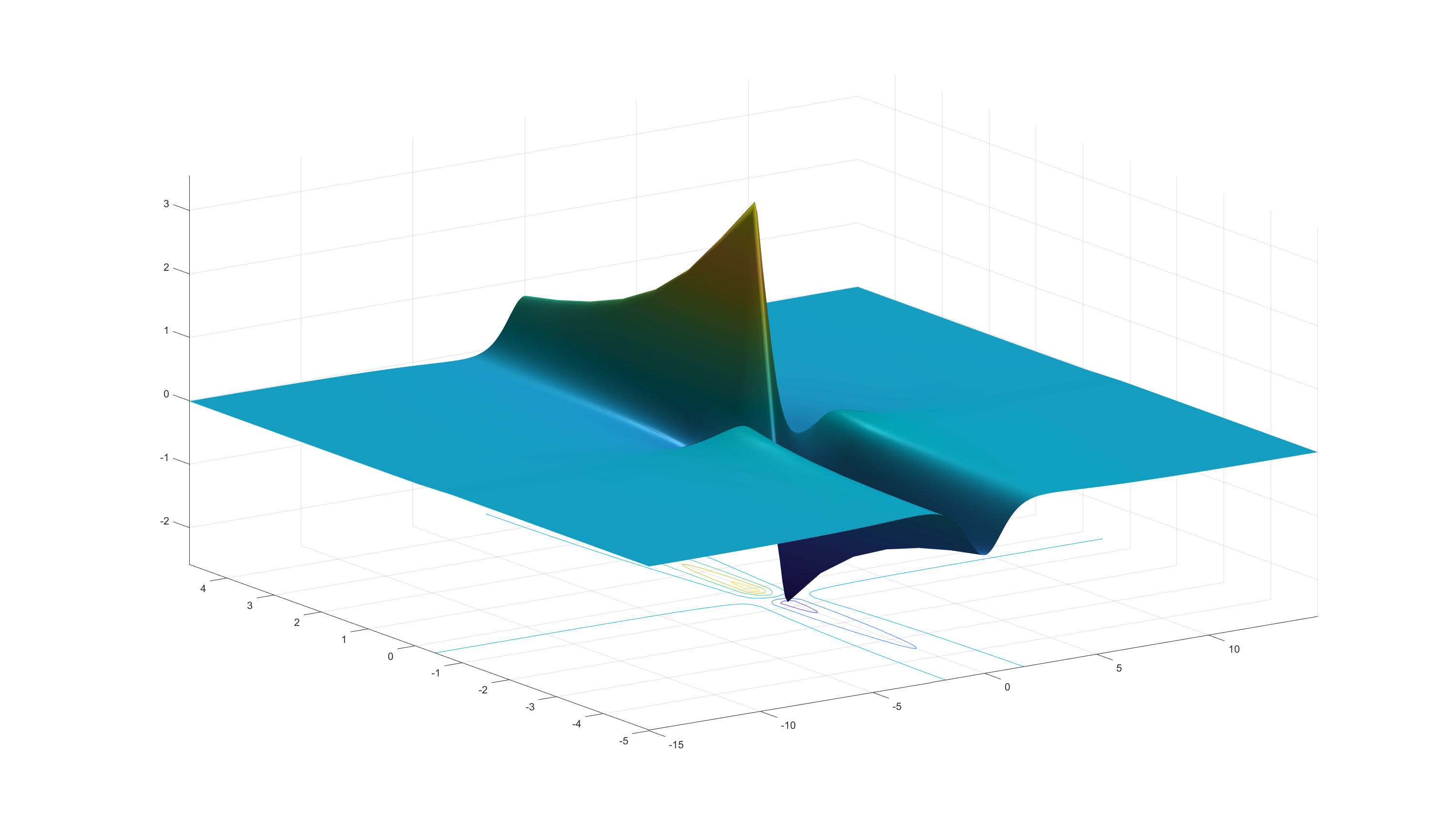}  	 \ \includegraphics{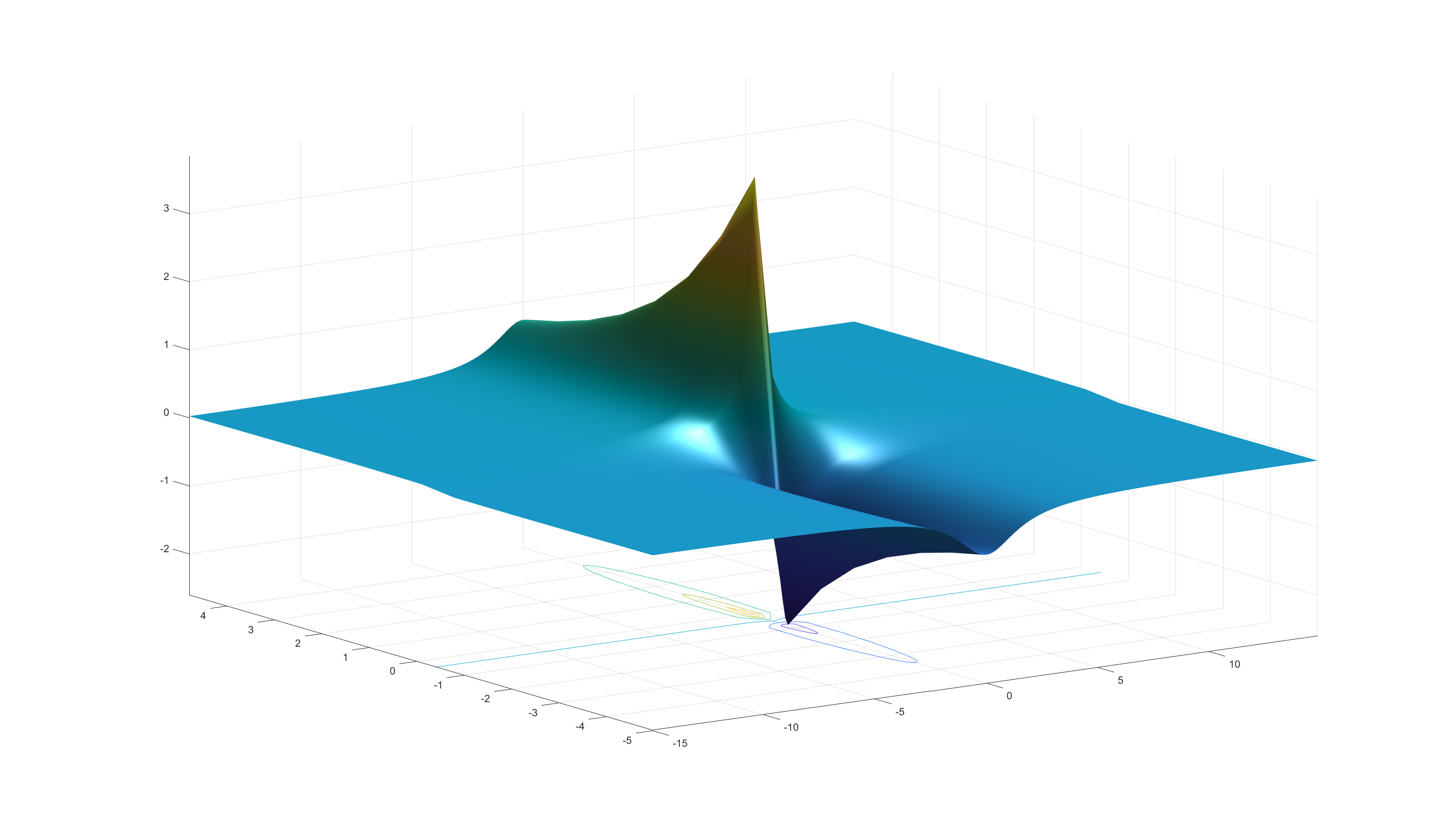}}
			\caption{The numerical surfaces of the projections of solitary waves of  \eqref{main-fifthkp} on the $XZ$-plane (left) and  the $YZ$-plane (right) with $\al=1$ and $f(u)=|u|^3+\mu_2 u^2$ with various values of $\mu_2$. }\label{nume-slt-3} 
	\end{center}\end{figure}

	\begin{remark}
		Numerical results illustrated in Figures \ref{nume-slt-1}-\ref{nume-slt-3} show a similar description of the behavior and the polarity change of the solitary waves with the different sign of $\mu_1$ and $\mu_2$ as it was reported for the Gradner equation \eqref{gardner} (see   Figure \ref{gard}).  
	\end{remark}
	\begin{theorem}\label{exist-th}
		Let $u\in X_\al$ satisfy $P(u)=0$ and $S(u)=m$. Then $u$ is a solution of \eqref{gkp}. Moreover, $m=m'$, and $u$ achieves $m$ is and only if $u$ achieves the minimum $m'$, where
		\[
		m'=\inf_{u\in N'_0}S_0(u)
		\]
		and $N_0'=\{u\in X_\al,\;u\not\equiv0,\; P(u)\leq0\}$.
	\end{theorem}
	\begin{proof}
		By the definition of $m$, there is the Lagrange multiplier $\theta\in\rr$ such that $S'(u)=\theta P'(u)$. Thus,
		\[
		\theta\langle P'(u),u\rangle=\langle S'(u),u\rangle=P(u)=0.
		\]
		By this assumption, we obtain
		\[
		\theta\langle P'(u),u\rangle
		=4I(u)-	\theta\langle N'(u),u\rangle=2\left(1-\frac{1}{\po+1}\right)I(u)<0.
		\]
		Thus, $\theta=0$ and $S'(u)=0$.
		
		Next, suppose that $S(u)=m$ and $P(u)=0$. Then clearly, $m'\leq S(u)=m$. Since $m<m_{\sigma}$ for all $\sigma<0$, we have $m<S_0(u)$ for all $u$ such that $P(u)<0$ and thereby  $m\leq m'$. This means that $m=m'$. Now if $u$ achieves the minimum $m'$, then $P(u)\leq0$ and $S_0(u)=m'=m$. Suppose that $P(u)<0$. Then $m_\sigma\leq S_0(u)=m$ which contradicts the fact $m<m_\sigma$. Hence, $P(u)=0$ and $u$ achieves  $m$.
	\end{proof}

	To have a complete picture of the existence of ground state, we extend our results to the case   $\mu_1<0<\mu_2$ with $\po>\poo$. Assuming
\begin{equation}\label{exist-cond}
		\moo(\po-1)+\mo(\poo-1)>0,
\end{equation}
	we define
	\[
	c_\ast=\frac{2(\po-1)\mu_2+\mu_1(\poo-1)}{(\po-1)(\poo+1)}
	\paar{\frac{\moo(\poo-1)(\po+1)}{-\mo(\po-1)(\poo+1)}}^\frac{\poo-1}{\po-\poo}.
	\]

	\begin{theorem}\label{ex-theo-pq}
Let $\moo(\po-1)+\mo(\poo-1)>0$.		For any $c\in(0,c_\ast),$  there exists a ground state of \eqref{gkp}.
	\end{theorem}
	
		\begin{remark}
			In view of Berestycki-Lions seminal work \cite{berlions}, we observe that
	\begin{equation}
		c_\ast=\sup\sett{c>0,\;\inf_{r\geq0}\paar{\frac{c}{2}r^2-F(r)}<0}.
	\end{equation}
			\end{remark}
	
	\begin{proof}[Proof of Theorem \ref{ex-theo-pq}]
		Consider a new variational problem
		\begin{equation}\label{new-min-highlow}
			\fd(c)=\inf\sett{S(u),\;u\in X_\al,\;u\not\equiv0,\;\fp(u)=0},
		\end{equation}
		where
		\begin{equation}
			\fp(u)=\frac{2-\al}{2(1+\al)}\|u\|_{\dot{X}_\al}^2+\frac12\int_\rt   \paar{cu^2-2F(u)}\dz.
		\end{equation}
		It is clear to see that
		\begin{equation}\label{expl-dc}
			\fd(c)=\infty
		\end{equation}
		as $c\to c_\ast$. It is straightforward to see from Lemma \eqref{Pohozaev} that any ground state of \eqref{gkp} is also a minimizer of \eqref{new-min-highlow}. We show that
		\begin{equation}\label{new-min-highlow-t}
			\fd(c)=\tilde\fd(c):=\inf\sett{\frac\al{\al+2}\|u\|_{\dot{X}_\al}^2,\;u\in X_\al,\;u\not\equiv0,\;\fp(u)\leq0}.
		\end{equation}
		Let $u$ be in $X_\al\setminus\{0\}$ such that $\fp(u)\leq0$. Using the fact $S(u)=\frac\al{\al+2}\|u\|_{\dot{X}_\al}^2$ and defining $u_\lam(x,y)=u(\lam x,\lam^2y)$, we obtain that there exists $\tau_0\geq1$ such that 
		\[
		\fd(c)\leq S(u_{\lam_0})=
		\frac\al{\al+2}\norm{u_{\lam_0}}_{\dot{X}_\al}^2
		\leq\frac\al{\al+2}\|u\|_{\dot{X}_\al}^2.
		\]
		This means that $\fd(c)\leq\tilde\fd(c)$. On the other hand, since $S(u)-\fp(u)=\frac\al{\al+2}\|u\|_{\dot{X}_\al}^2$, then $\fd(c)\geq\tilde\fd(c)$, so that $\fd(c)=\tilde\fd(c)$. Moreover, we have $\tilde\fd(c)>0$. Indeed, it yields for any $u\in X_\al\setminus\{0\}$ with $\fp(u)\leq0$ from \eqref{bin} that
		\[
		\frac{2-\al}{2(\al+2)}\|u\|_{\dot{X}_\al}^2+cM(u)\leq K(u)
		\leq \frac c2M(u)+C_c\|u\|_{\dot{X}_\al}^{2^\ast}
		\]
		for some $C_c$ depending only on $c$. This reveals that $\|u\|_{\dot{X}_\al}^{2^\ast-2}\geq C>0$. Hence, $\tilde\fd(c)>0$.
		
		Next, we prove that any minimizer $\ff$ of $\fd(c)$ is a ground state of \eqref{gkp}. To do so, it is enough to show that $\ff$ satisfies \eqref{gkp}.  By contradiction, we assume that there exists a minimizer of $\ff$ of $\fd(c)$ such that $S'(\ff)\neq0$. Without loss os generality, we can assume that $\scal{S'(\ff),\psi}=-1$ for some $\psi\in X_\al$. Since $\al=1$, $\fp(\ff)=0$ and $$\frac{\dd}{\dd\lam}S(\ff_\lam)=-\frac{\lam^{-4}}2(1-\lam^2)\|\ff\|_{\dot{X}_\al}^2.$$
		Hence, $S(\ff_\lam)$ with $\lam>0$ attains its maximum at $\lam=1$. 
		On the other hand, since $S$ is $C^1$, there exists $1\gg\delta_0>0$ sufficiently small such that $\scal{S'(\ff_\lam),\psi}$ is close to $-1$ if $|\lam-1|<\delta_0$. Define $\tilde\ff_\lam=\ff_\lam+\delta\chi\paar{\frac{\lam-1}{\delta}}\psi$, where $0<\delta\ll\delta_0$ and $\chi(\lam)=|1-\lam|\mathbb{1}_{[-1,1]}(\lam)$. It is clear that
		\[
		\fp(\tilde\ff_{1-\delta})<0<\fp(\tilde\ff_{1+\delta}).
		\]
		So, there exists $\delta^0\in(-\delta,\delta)$ such that $\fp(\tilde\ff_{\delta^0})=0$. Thusly, it follows from the definition $\fd(c)$  that
		\[
		\begin{split}
			\fd(c)&\leq S(\tilde\ff_{\delta^0})
			=
			S( \ff_{\delta^0})+
			\scal{S'(\ff_{\delta^0}),\delta\chi\paar{\frac{\delta^0-1}{\delta}}\psi}
			+o\paar{\norm{\tilde\ff_{\delta^0}-\ff_{\delta^0}}_{X_\al}}\\
			&\leq S( \ff )-\delta\chi\paar{\frac{\delta^0-1}{\delta}}\\
			&<S(\ff)=\fd(c)
		\end{split}
		\]
		for sufficiently small $\delta$. This contradiction reveals that $\ff$ is aground state.
		Now, assume that $\{u_n\}\subset X_\al$ is a minimizing sequence. So, $S(u_n)\to \fd(c)$ and $\fp(u_n)\to0$ as $n\to\infty$. This implies that $\fd(c)\sim\|u_n\|_{\dot{X}_\al}^2$ for any $n\gg1$. Another application of \eqref{bin} gives for any $n\gg1$ that
		\[
		\fp(u_n)\gtrsim M(u_n)-\paar{M(u_n)}^{c_p} \|u_n\|_{\dot{X}_\al}^{\frac{(p-1)(\al+2)}{2\al}},
		\]
		where $c_p$ is defined in \eqref{bin}.
		This yields that $M(u_n)\lesssim1$ for any $n\gg1$. Moreover, $\|u_n\|_{L^{p+1}}^{p+1}\gtrsim \fd(c)$ for any $n\gg1$. By applying $\mathbbm{pqr}$-lemma (see \cite{fll}) with $\mathbbm{p}=2<\mathbbm{q}=p<\mathbbm{r}=6$, there exists $C>0$
		and $r>0$ such that
		\[
		\abso{\sett{(x,y)\in\rt,\; |u_n(x,y)|>r}}\geq C
		\]
		for any $n\gg1$.
		Thus, it follows from the boundedness of $\{u_n\}\subset X_\al$ and \cite[Lemma 4]{liuwang} that there exist $C_1=C_1(C,\eta,r)$ and  a subsequence of $\{u_n\}$, still denoted by the same symbol, a non-trivial function $\ff$ and a sequence $\sett{z_n}\subset\rt$ such that 
		\[
		\abso{B\cap\sett{(x,y)\in\rt,\; |u_n((x,y)+z_n)|>r}}\geq C_1
		\]
		and
		$
		u_n(\cdot+z_n)\rightharpoonup\ff
		$ in $X_\al$, where $B$ is the unit box in $\rt$. Moreover, since $\al=1$, the weak convergence of $u_n$, the Brezis-Lieb lemma and \cite{liuwang} show  that
		
		\begin{equation}\label{bl-id}
			\begin{split}
				&\fp(u_n)-\fp(u_n-\ff)-\fp(\ff)\to0,\\
				&\norm{u_n}_{\dot{X}_\al}^2-\norm{u_n-\ff}_{\dot{X}_\al}^2-\norm{\ff}_{\dot{X}_\al}^2\to0
			\end{split}
		\end{equation}
		as $n\to\infty$. It is clear that $\fp(\ff)\leq0$. Indeed, if $\fp(\ff)>0$, then \eqref{bl-id} reveals that 
		\[
		\fp(u_n-\ff)=-\fp(\ff)<0,
		\]
		so that we obtain from $\fd(c)=\tilde\fd(c)$ that $\fd(c)\lesssim \|u_n-\ff\|_{\dot{X}_\al}^2$ for any $n\gg1$. Hence, we have from \eqref{bl-id} that
		\[
		\|\ff\|_{\dot{X}_\al}^2=\lim_{n\to\infty}\paar{\|\ff\|_{\dot{X}_\al}^2-\|u_n-\ff\|_{\dot{X}_\al}^2}
		\lesssim \fd(c)-\fd(c)=0.
		\]
		This contradiction shows the non-positivity of $\fp(\ff)$. Therefore, by utilizing again the fact $\fd(c)=\tilde\fd(c)$ together with the Fatou  lemma, we obtain that
		\begin{equation}\label{limit-fd}
			\lim_{n\to\infty}\frac\al{\al+2}\|u_n\|_{\dot{X}_\al}^2=\fd(c)=\frac\al{\al+2}\|\ff\|_{\dot{X}_\al}^2.
		\end{equation}
		Thus, $\{u_n(\cdot+z_n)\}$  strongly converges to $\ff$ in $\dot{X}_\al$. Moreover, 
		\begin{equation}\label{st-conv}
			u_n(\cdot+z_n)\to\ff
		\end{equation} in $L^s(\rt)$ for any $2<s\leq2^\ast$. Moreover, there exists $\lam_1\geq1$ such that $\fp(\ff_{\lam_1})=0$, where $\ff_{\lam_1}(x,y)=\ff(\lam_1x,\lam_1^2y )$, whence
		\[
		\fd(c)\leq S(\ff_{\lam_1})=\frac\al{\al+2}\norm{\ff_{\lam_1}}_{\dot{X}_\al}^2=\lam_1^{-1}\fd(c).
		\]
		So, $\lam_1=1$ which results to $\fp(\ff)=0$. Another application of \eqref{bl-id} together with $\lim_{n\to\infty}\fp(u_n)=0$ shows that $\lim_{n\to\infty}\fp(u_n-\ff)=0$. Consequently,    a combination of \eqref{st-conv} and \eqref{limit-fd} completes the proof.
	\end{proof}

Finally, we conclude this section by deriving the following Pohozaev identities, which will be needed in our instability analysis.

	\begin{lemma}\label{Pohozaev} Suppose $\ff\in X_\al$ is a solution of \eqref{gkp}. Then for integers $\alpha\geq1$,
		\begin{align*}
			\int_{\mathbb R^2}(D_x^\alpha\ff)^2+c\ff^2+(\partial_x^{-1}\ff_y)^2\dd x\dd y&=\int_{\mathbb R^2}\ff f(\ff)\dd x\dd y\\
			\int_{\mathbb R^2}\left(\alpha-\frac12\right)(D_x^\alpha\ff)^2-\frac12c\ff^2-\frac32(\partial_x^{-1}\ff_y)^2\dd x\dd y&=-\int_{\mathbb R^2}F(\ff)\dd x\dd y\\
			\int_{\mathbb R^2}-\frac12(D_x^\alpha\ff)^2-\frac12c\ff^2+\frac12(\partial_x^{-1}\ff_y)^2\dd x\dd y&=-\int_{\mathbb R^2}F(\ff)\dd x\dd y
		\end{align*}
	\end{lemma}
	\vskip 10pt
	\begin{proof} If $\ff$ is a solution of \eqref{gkp} then 
		\[
		D_x^{2\alpha}\ff+c\ff+\partial_x^{-2}\ff_{yy}=f(\ff)
		\]
		so multiplying by $\ff$, $x\ff_x$ and $y\ff_y$ and integrating yields the identities in the lemma.
	\end{proof}
 
	\begin{corollary}\label{poho-c}
		Summing the 3 identities above yields
	\[
	\alpha\int_{\mathbb R^2}(D_x^\alpha\ff)^2\dd x\dd y=\int_{\mathbb R^2}\ff f(\ff)-2F(\ff)\dd x\dd y,
	\]
	and subtracting the third one from the second one gives
	\[
	2\int_{\mathbb R^2}(\partial_x^{-1}\ff_y)^2\dd x\dd y=\alpha\int_{\mathbb R^2}(D_x^\alpha\ff)^2\dd x\dd y
	\]
	So using the first equation we obtain the relations
	\begin{align*}
		\int_{\mathbb R^2}(D_x^\alpha\ff)^2\dd x\dd y&=\frac1{\alpha}\int_{\mathbb R^2}\ff f(\ff)-2F(\ff)\dd x\dd y\\
		\int_{\mathbb R^2}(\partial_x^{-1}\ff_y)^2\dd x\dd y&=\frac1{2}\int_{\mathbb R^2}\ff f(\ff)-2F(\ff)\dd x\dd y\\
		\int_{\mathbb R^2}c\ff^2\dd x\dd y&=\frac1{2\alpha}\int_{\mathbb R^2}(\alpha-2)\ff f(\ff)+(2\alpha+4)F(\ff)\dd x\dd y
	\end{align*}
	\end{corollary}
	
	\section{Stability and Instability of Ground States}\label{sect-stability}
	
	In this section we investigate both analytically and numerically the stability of ground state solitary waves. We first recall that a subset $D$ of $X_\al$ is said to be stable with respect to \eqref{main-fifthkp} if for any
	$\epsilon>0$ there exists some $\delta > 0$ such that, for any $u_0\in B_\delta (D)$ ($\delta$-neighborhood of $D$), the solution $u$ of \eqref{main-fifthkp} with $u(0) = u_0$ satisfies
	$u(t)\in B_\epsilon(D) $ for all $t > 0$. Otherwise, we say $D$ is unstable. We will consider both the stability of the set of ground states with speed $c$, namely
	\[
	G_c=\{\ff\in X_\al\setminus\{0\},\;S(\ff)=m,\;P(\ff)=0\},
	\]
	as well that of the orbit of a particular ground state $\ff\in G_c$ defined by
	\[
	\mathcal{O}(\ff)=\{\ff(\cdot-\tau,\cdot):\tau\in\mathbb R\}.
	\]
	The following stability and instability results follow from the arguments in \cite{esfahani-levan-dpde}, and we omit their proofs.
	\subsection{Stability}
	 We first state state our result in the case $\mo>0$.
	\begin{theorem}\label{stability}
	Let $\mo>0$ and $\po>\poo$.
		 Define $d(c)=S(\ff)$, where $\ff\in G_c$. If $d''(c)>0$ then $G_c$ is stable.
	\end{theorem}
 The proof of Theorem \ref{stability} does not directly apply to the case where $\mu_1 < 0$, and addressing this scenario requires more intricate modifications.
 	\begin{theorem}\label{T:stability}
 	Let	$\mo<0<\moo$, $\po>\poo$, and $c\in(0,c_\ast)$ such that \eqref{exist-cond} holds. If $d''(c)>0$, then $G_c$ is stable.
 \end{theorem}
 
However, we also establish the stability of ground states by relaxing the convexity condition when the wave speed is close to $c_\ast$.
 
		\begin{theorem}\label{stab-theopq}
Let	$\mo<0<\moo$ and $\po>\poo$ such that \eqref{exist-cond} holds.	There exists $\{c_n\}\subset(0,c_\ast)$ such that $c_n\to c_\ast$ as $n\to\infty$ and the set $G_{c_n}$ of all ground states is stable for all $n$.
	\end{theorem}
	
	From now on, we   present $\fp$ and $S$ by $\fp_c$ and $S_c$, receptively, to insist their dependence on $c$.
	\begin{lemma}
		There exists $C>0$ such that for any ground state  $\ff_{c_1}$ and $\ff_{c_2}$ of \eqref{new-min-highlow-t} corresponding to $c_1<c_2<c_\ast$,
		\begin{equation}\label{left}
			\begin{split}
				\fd(c_1)&\leq\fd(c_2)-M(\ff_{c_2})(c_2-c_1)+C\frac{\paar{M(\ff_{c_2})}^2}{\fd(c_2)}(c_2-c_1)^2.
			\end{split}
		\end{equation}
		Moreover,
		\begin{equation}\label{right}
			\fd(c_2) \leq\fd(c_1)+M(\ff_{c_1})(c_2-c_1)+C\frac{\paar{M(\ff_{c_1})}^2}{\fd(c_1)}(c_2-c_1)^2 
		\end{equation}
		provided $c_2-c_1$ is sufficiently close to zero.
	\end{lemma}
	\begin{proof}
		For any $u\in X_\al$ and $\lam>0$, define $u_\lam(x,y)=u(\sqrt{\lam}x,\lam y)$. First we note from $\fp_{c_2}(\ff_{c_2})=0$ that
		\[
		\fp_{c_1}(\ff_{c_2,\lam})=\frac{\lam^{-\frac32}}{6}\paar{(\lam-1)\norm{\ff_{c_2}}_{\dot{X}_\al}^2-3(c_2-c_1)\norm{\ff_{c_2}}_\lt^2
		}.
		\]
		Then, $\fp_{c_1}(\ff_{c_2,\lam_0})=0$, where
		\[
		\lam_0=1+\frac{6(c_2-c_1)M(\ff_{c_2})}{\norm{\ff_{c_2}}_{\dot{X}_\al}^2},
		\]
		thereby  
		\begin{equation}\label{left-1}
			\fd(c_1)=\tilde\fd(c_1)\leq\frac\al{\al+2}\norm{\ff_{c_2,\lam_0}}_{\dot{X}_\al}^2=\lam_0^{-\frac12}\fd(c_2).
		\end{equation}	
		Now, we have from the Taylor expansion of $\lam_0^{-\frac12}$ about $1$ for some $\tau\in(0,1)$ that
		\[
		\begin{split}
			\lam_0^{-\frac12}&=1-\frac{(c_2-c_1)M(\ff_{c_2})}{\fd(c_2)}
			+\frac38\paar{\frac{6(c_2-c_1)M(\ff_{c_2})}{\norm{\ff_{c_2}}_{\dot{X}_\al}^2}}^2\paar{1+\frac{6\tau(c_2-c_1)M(\ff_{c_2})}{\norm{\ff_{c_2}}_{\dot{X}_\al}^2}}^{-\frac52}\\&
			\leq 
			1-\frac{(c_2-c_1)M(\ff_{c_2})}{\fd(c_2)}
			+C\paar{\frac{(c_2-c_1)M(\ff_{c_2})}{\fd(c_2)}}^2. 
		\end{split} 
		\]
		By combining the above inequality with \eqref{left-1}, we  arrive at \eqref{left}.
		
		Similarly, we have
		\[
		\fp_{c_2}(\ff_{c_1,\lam})=
		\frac{\lam^{-\frac32}}{6}
		\paar{(\lam-1)\norm{\ff_{c_1}}_{\dot{X}_\al}^2+3(c_2-c_1)\norm{\ff_{c_1}}_\lt^2
		}. 
		\]
		This shows that $\fp_{c_2}(\ff_{c_1,\lam_1})=0$ provided $c_2-c_1$ is small enough, where
		\[
		\lam_1=1-\frac{6(c_2-c_1)M(\ff_{c_1})}{\norm{\ff_{c_1}}_{\dot{X}_\al}^2}\in(0,1).
		\]
		Hence, it holds that
		\begin{equation}\label{right-1}
			\fd(c_2) \leq\frac\al{\al+2}\norm{\ff_{c_1,\lam_1}}_{\dot{X}_\al}^2=\lam_1^{-\frac12}\fd(c_1).
		\end{equation}
		Another application of the Taylor expansion for $\lam_0^{-\frac12}$ about $1$ reveals for some $\tau\in(0,1)$ that
		\[
		\begin{split}
			\lam_1^{-\frac12}&=1+\frac{(c_2-c_1)M(\ff_{c_1})}{\fd(c_1)}
			+\frac38\paar{\frac{6(c_2-c_1)M(\ff_{c_1})}{\norm{\ff_{c_1}}_{\dot{X}_\al}^2}}^2\paar{1-\frac{6\tau(c_2-c_1)M(\ff_{c_1})}{\norm{\ff_{c_1}}_{\dot{X}_\al}^2}}^{-\frac52}\\&
			\leq 
			1-\frac{(c_2-c_1)M(\ff_{c_1})}{\fd(c_1)}
			+C\paar{\frac{(c_2-c_1)M(\ff_{c_1})}{\fd(c_1)}}^2, 
		\end{split} 
		\]
		provided that $c_2-c_1$ is close to zero. Inserting the above inequality into \eqref{right-1} yields \eqref{right}.
	\end{proof}
	\begin{corollary}\label{cont-insc-d}
		The function $\fd(c)$ is strictly increasing  and continuous on $(0,c_\ast)$.
	\end{corollary}
	
	\begin{lemma}\label{lemma-st-3}
		Let $c_0\in(0,c_\ast)$ and $\ff_{c_0}$ be a ground state of \eqref{gkp}. If 
		\begin{equation}\label{assump-lem-1}
			\fd(c)-\fd(c_0)> M(\ff_{c_0})(c-c_0)
		\end{equation} 
		for all $c\in(0,c_\ast)$ with $c\neq c_0$, then for any $\varepsilon\in(0,c_0)$,  there exists $\delta>0$ such that if $\psi\in X$ satisfies 
		\begin{equation}
			\norm{\psi-\ff_{c_0}}_{X_\al}\leq\delta,
		\end{equation}
		the solution $u\in C([0,T);X_\al)$ of \eqref{main-fifthkp} with $u(0)=\psi$ satisfies
		\begin{equation}\label{claim-lem-1}
			\fd(c_0-\varepsilon)<\frac\al{\al+2}\|u(t)\|_{\dot{X}_\al}^2<\fd(c_0+\varepsilon)
		\end{equation}
		for all $0<t<T$.
	\end{lemma}
	We postpone the proof of Lemma \ref{lemma-st-3} and first complete the proof of Theorem \ref{stab-theopq}.
	
	\begin{proof}[Proof of Theorem \ref{stab-theopq}.]
		The proof is divided into two Steps.

		\textbf{Step 1.} We first show that there exists $\{c_n\}\subset(0,c_\ast)$ and $\{Z_n\}\subset(0,\infty)$ such that $Z_n\to\infty$, $c_n\to c_\ast$ as $n\to\infty$ and
		\begin{equation}\label{claim-s1}
			\fd(c)-\fd(c_n)>Z_n(c-c_n)
		\end{equation}
		for all $n$ and all $c\in(0,c_\ast)$ with $c\neq c_n$.
		
		Choose a sequence $\{a_n\}\subset(1,\infty)$ such that $a_n\to\infty$, and define $b_n=c_\ast^{-1}a_n\fd((1-a_n^{-1})c_\ast)$. Then, is is seen from \eqref{expl-dc} that $b_n\to\infty$ as $n\to\infty$. Taking $\ell\geq2$, it follows for any $c\in(0,(1-\ell a_n^{-1}))$ that
		\[
		\begin{split}
			\fh_n(c):=\fd(c)-b_nc-c^2&>-b_nc -c^2>-b_n(1-\ell a_n^{-1})c_\ast-\paar{(1-\ell a_n^{-1})c_\ast}^2\\
			&>
			b_n(a_n^{-1}(\ell-1))c_\ast-b_nc_\ast(1-a_n^{-1})-\paar{(1-a_n^{-1})c_\ast}^2\\&\geq \fh_n((1-a_n^{-1})c_\ast).
		\end{split}
		\]
		This shows from $\lim_n\fh_n(c)=\infty$ that there exists a minimum point $c_n\geq (1-a_n^{-1})c_\ast$ to $\fh_n(c)$ such that $c_n\to c_\ast$ as $n\to\infty$. Next we show for any $n$ and any $c\in(0,c_\ast)$ with $c\neq c_n$ that 
		\[
		\fd(c)>\fd(c_n)+(b_n+2c_n)(c-c_n).
		\]
		Since $c_n$ is a minimum point of $\fh_n$, then
		\[
		\begin{split}
			\fd(c)&=\fh_n(c)+b_nc+c^2
			>
			\fh_n(c_n)+b_nc_n+c^2_n
			+(b_n+2c_n)(c-c_n)\\&
			=\fd(c_n)+(b_n+2c_n)(c-c_n)
		\end{split}.
		\]
		By define $Z_n=2c_n+b_n$, we get \eqref{claim-s1}. Moreover, for any ground state $\ff_n\in G_{c_n}$, we have from \eqref{left} that
		\[
		M(\ff_n)(c_n-c)\leq \fd(c_n-\fd(c))+o(c_n-c)\leq Z_n(c_n-c)+o(c_n-c),
		\]
		provided $c_n-c$ is close to zero. Whence we obtain $M(\ff_n)\leq Z_n$ by taking $c\to c_n$. By using \eqref{right}, it is similarly deduced that $M(\ff_n)\geq Z_n$. Thus,
		\[
		M_n(\ff_n)=Z_n.
		\]

		\textbf{Step 2.} Completion of the proof.
		
		Let $\{c_n\}$ be the same sequence obtained in Step 1. We show for any $c_n$,  the set  $G_{c_n}$ of all ground states of \eqref{gkp} corresponding with $c_n$ is stable. By contradiction, assume that $G_{c_N}$ is unstable for some $N$. Thus, there exists $\epsilon_0>0$ such that for any $m\in\N$, there exists $\psi_m\in X_\al$ and $t_m\in[0,T_m)$ such that
		\begin{equation}\label{contrad-st}
			\begin{split}
				\inf_{\ff\in G_{c_N}}\norm{\psi_m-\ff}_{X_\al}\leq\frac1m,\quad\text{and}\quad
				\inf_{\ff\in G_{c_N}}\norm{u_m(t_m)-\ff}_{X_\al}\geq\epsilon_0,
			\end{split}
		\end{equation}
		where $u_m\in C([0,T_m);X_\al)$ is the unique solution of \eqref{main-fifthkp} with $u_m(0)=\psi_m$.
		The conservations of energy and momentum, Step 1, Lemma \ref{lemma-st-3} and Corollary \ref{cont-insc-d} show that
		\[
		\begin{split}
			&	S_{c_N}(u_m(t_m)) =S_{c_N}(\psi_m )\to \fd(c_N),\qquad
			\frac\al{\al+2}\norm{u_m(t_m)}_{\dot{X}_\al}^2\to\fd(c_N)
		\end{split}
		\]
		as $m\to\infty$. Hence, we obtain that
		\[
		\fp_{c_N}(u_m(t_m))=	S_{c_N}(u_m(t_m))-
		\frac\al{\al+2}\norm{u_m(t_m)}_{\dot{X}_\al}^2\to0
		\]
		as $m\to\infty$. By repeating the proof of Theorem \ref{ex-theo-pq}, we can find a subsequence of $\{u_m(t_m)\}$, denoted by the same symbol, a sequence $\{z_m\}\subset\rt$, and  a ground state $\ff\in X_\al$ such that $u_m(t_m,\cdot+z_m)$ strongly converges to $\ff$ in $X_\al$. This means that
		\[
		\lim_{m\to\infty}\norm{u_m(t_m)-\ff(\cdot-z_m)}_{X_\al}=0.
		\]
		This is a contradiction to \eqref{contrad-st}, and the proof of Theorem \ref{stab-theopq} is complete.
	\end{proof}
	
	\begin{proof}[Proof of Lemma \ref{lemma-st-3}.]
		Let $\delta>0$ and $\varepsilon\in(0,c_0)$, and $\psi\in X_\al$ such that
		\[
		\norm{\psi-\ff_{c_0}}_{X_\al}\leq\delta.
		\]
		The facts $\fp_{c_0}(\ff_{c_0})=0$ and $S_{c_0}(\ff_{c_0})=\frac\al{\al+2}\norm{\ff_{c_0}}_{\dot{X}_\al}^2$ imply that
		\begin{equation}
			\fd(c_0)=
			\frac\al{\al+2}\norm{\psi}_{\dot{X}_\al}^2+O(\delta).
		\end{equation}
		Hence, by choosing $\delta=\delta(\varepsilon)$ sufficiently small, we derive from Corollary \ref{cont-insc-d} and \eqref{assump-lem-1} that
		\begin{equation} \label{claim-lem-1-0}
			\fd(c_0-\varepsilon)<\frac\al{\al+2}\|\psi\|_{\dot{X}_\al}^2<\fd(c_0+\varepsilon).
		\end{equation}
		Define the sets
		\begin{equation}
			\begin{split}
				&\Gamma_c^+=\sett{u\in X_\al,\;S_c(u)<\fd(c),\; 3\fd(c)>\norm{u}_{\dot{X}_\al}^2 },\\
				&\Gamma_c^-=\sett{u\in X_\al,\;S_c(u)<\fd(c),\; 3\fd(c)<\norm{u}_{\dot{X}_\al}^2 }.
			\end{split}
		\end{equation}
		It is straightforward to see that the submanifolds $\Gamma_c^\pm$ are invariant under the flow of the Cauchy problem associated with \ref{main-fifthkp}. Indeed, if $\psi\in \Gamma_c^+$ and $u(t)$ is a unique solution of \eqref{main-fifthkp} with $u(0)=\psi$. Then, the invariants $E$ and $M$ show that $S_c(u(t))<\fd(c)$ for all $t\in[0,T)$. By using
		\[
		\frac\al{\al+2}\norm{u(t)}_{\dot{X}_\al}^2=S_c(u(t))-\fp_c(u(t))<\fd(c)-\fp_c(u(t)),
		\]
		it suffices to show $\fp_c(u(t))\geq0$ for all $t$. Suppose by contradiction that there exists $t_0\in(0,T)$ such that $\fp(u(t_0))<0$. Since $\psi\in\Gamma_c^+$, then $\fp_c(\psi)>0$, so that there exists $t_1\in(0,t_0)$ such that $\fp_c(u(t_1))=0$. This contradicts the fact $S_c(u(t_1))\geq\fd(c)$. The invariance of $\Gamma_c^-$ is proved similarly.
		
		Hence, to complete the proof, it is enough to prove that
		\begin{equation}\label{assu-lem-2}
			S_{c_0\pm\varepsilon}(\psi)<\fd(c_0\pm\varepsilon).
		\end{equation}
		We only prove \eqref{assu-lem-2} for $c_0+\varepsilon$, and the case $c_0-\varepsilon$ can be analogously proved. By using \eqref{assump-lem-1} and the Taylor expansion about $\ff_{c_0}$, we obtain that
		\[ \begin{split}
			S_{c_0+\varepsilon}(\psi)
			&=S_{c_0+\varepsilon}(\ff_{c_0})+
			\scal{ S'_{c_0+\varepsilon}(\ff_{c_0}),\psi-\ff_{c_0}}+o(\delta)\\
			&=S_{c_0}(\ff_{c_0})+\varepsilon M(\ff_{c_0})+
			\scal{ S'_{c_0+\varepsilon}(\ff_{c_0}),\psi-\ff_{c_0}}+o(\delta)\\
			&\leq \fd(c_0)+\varepsilon M(\ff_{c_0})+2\epsilon\delta M(\ff_{c_0})+o(\delta). 
		\end{split}
		\]
		This implies that
		\[\begin{split}
			S_{c_0+\varepsilon}(\psi)-\fd(c_0+\varepsilon)&< \fd(c_0)-\fd(c_0+\varepsilon)+\varepsilon M(\ff_{c_0}) +2\varepsilon\delta M(\ff_{c_0})+o(\delta)\\&=:C_{c_0}(\varepsilon)+2\varepsilon\delta M(\ff_{c_0})+o(\delta).
		\end{split}
		\]
		Since $C_{c_0}(\varepsilon)<0$, by taking $\delta=\delta(\varepsilon,c_0)>0$ sufficiently small we arrive at \eqref{assu-lem-2}.
	\end{proof}

We now shift our focus to the proof of Theorem \ref{T:stability}, which incorporates the essential convexity condition for establishing stability. The proof is obtained through an argument similar to one in \cite{esfahani-levan-dpde} and is complemented by the following crucial lemmas; therefore, we omit the details.

	\begin{lemma}\label{L:c_map} 
		There exists $\epsilon>0$ and a continuous map $c:U_\epsilon(G_c)\to\mathbb R$ such that $d(c(u))=\frac\al{\al+2}\|u\|_{\dot{X}_\al}^2$ for each $u\in U_\epsilon(G_c)$.
	\end{lemma}
	
	\begin{proof} Since $\fd(c)$ is continuous and strictly increasing, it follows that $d^{-1}$ is continuous and strictly increasing. Since
		\[
		d(c)=\fd(c)=\tilde\fd(c)=S_c(\ff)
		\]
		for any $\ff\in G_c$, it follows from the continuity of $S_c$ and $\fp_c$  that there exists $\epsilon>0$ such that $S_c(u)$ is in the range of $d$ for any $u\in U_\epsilon(G_c)$. We may therefore define
		\begin{equation}\label{c_map_definition}
			c(u)=d^{-1} \paar{\frac\al{\al+2}\|u\|_{\dot{X}_\al}^2} 
		\end{equation}
		for any such $u$. Continuity of this map follows from the continuity of $d^{-1}$ and $X_\al$-norm.
	\end{proof}
	
	The following lemma is the key step in the stability proof.
	
	\begin{lemma}\label{L:stability_bound} Suppose $d''(c)>0$. Then there exists some $\epsilon>0$ such that  we have for any $\ff\in G_c$ and any $u\in U_\epsilon(G_c)$
		$$E(u)- E(\ff) +c(u)(M(u)-M(\ff))\geq \dfrac{1}{4} d''(c) (c(u)-c)^2.$$
	\end{lemma}

	\begin{proof}Following the \cite[Lemma 4.3]{esfahani-levan-dpde}, we have
		\[
		d(c_1) \geq d(c)+M(\ff)(c_1-c)+ \frac{1}{4}d''(c) (c_1- c)^2
		\]
		for $c_1$ sufficiently close to $c$. It then follows that
		\begin{align*}
			d(c(u))&\geq d(c)+M(\ff)(c(u)-c)+\frac{1}{4}d''(c) (c(u)- c)^2\\
			&=E(\ff)+cM(\ff)+M(\ff)(c(u)-c)+\frac{1}{4}d''(c) (c(u)- c)^2\\
			&=E(\ff)+c(u)M(\ff)+\frac{1}{4}d''(c) (c(u)- c)^2
		\end{align*}
		for all $u\in U_\epsilon(G_c)$ provided $\epsilon>0$ is sufficiently small. Now we consider two cases.
		
		\noindent{\bf Case 1:} $\fp_{c(u)}(u)>0$. In this case, we have
		\[
		d(c(u))= S_{c(u)}(u)- \fp_{c(u)}(u)< S_{c(u)}(u)
		=E(u)+c(u)M(u)
		\]
		and combining this with the previous inequality gives
		\[
		E(u)+c(u)M(u)> E(\ff)+c(u)M(\ff)+\frac{1}{4}d''(c) (c(u)- c)^2,
		\]
		and thus proves the claim.
		
		\noindent{\bf Case 2:} $\fp_{c(u)}(u)\leq0$. In this case, let $\ff\in G_{c(u)}$. Then, since $\ff$ minimizes
		$S_{c(u)}(v)$ over all $v$ with $\fp_{c(u)}(v)\leq0$,  we have  
		\begin{align*}
			d(c(u))&=S_{c(u)}(\ff)\leq S_{c(u)}(u)\\
			&=S_{c(u)}(u)- \fp_{c(u)}(u)\\
			&=d(c(u)).
		\end{align*}
		Consequently, all of the aforementioned quantities are equivalent, indicating that $u$ attains the same minimum value as $\varphi$. As a result, $u$ also belongs to $G_{c(u)}$. This leads to the equation:
		\[
		d(c(u)) = S_{c(u)}(u) = E(u) + c(u)M(u),
		\]
		which, in turn, establishes the validity of the claim.	
	\end{proof}
	\subsection{Instability}
	\begin{theorem}\label{instability} Let $\ff\in G_c$. If there exists $\phi\in L^2$ such that $\phi_x\in X_s$ for some $s>3/2$, $\phi_{xx}\in X_\al$, $\left<\phi_x,\ff\right>=0$ and $\left<S''(\ff)\phi_x,\phi_x\right><0$, then $\mathcal{O}(\ff)$ is unstable.
	\end{theorem}
	
	\vskip 10pt
	Assuming there exists a choice of $\ff(c)\in G_c$ that is $C^1$ as a mapping from $\mathbb R^+$ to $X_\al$, the function $\phi_x=\ff-\frac{2d'(c)}{d''(c)}\frac{d\ff(c)}{dc}$ satisfies the hypotheses of Theorem \ref{instability} and as a consequence, we have the following converse of Theorem \ref{instability}. 
	\vskip 10pt
	\begin{corollary}\label{d''_instability} If $d''(c)<0$ then $\mathcal{O}(\ff)$ is unstable.
	\end{corollary}
	\vskip 10pt
	Due to the inhomogeneity of the nonlinear term, we do not have explicit formulas for $d(c)$. The numerical results presented at the end of this section provide approximate intervals of stability and instability for various nonlinear terms. We first consider the following alternate instability criteria.

	\begin{corollary}\label{instability_criterion} Let $\ff\in G_c$ and define 
		\[
		\mathcal{K}_f(\ff)=\int_{\mathbb R^2}\ff f(\ff)-\ff^2f'(\ff)\dd x\dd y
		+2\left((\alpha+2)a^2-4a+3\right)\int_{\mathbb R^2}\ff f(\ff)-2F(\ff)\dd x\dd y.
		\]
		If $\mathcal{K}_f(\ff)<0$ for some $a\in\mathbb R$ then $\mathcal{O}(\ff)$ is unstable. 
	\end{corollary}
	
	\vskip 10pt
	\begin{proof} Let $\phi_x=\ff+ax\ff_x+by\ff_y$ where $a+b=2$. A straightforward integration by parts gives
		\[
		\left<\phi_x,\ff\right>=\left(1-\frac12(a+b)\right)\int_{\mathbb R^2}\ff^2\dd x\dd y
		\]
		so $\left<\phi_x,\ff\right>=0$ when $a+b=2$. Next set $\mathcal{L}=S''(\ff)$. Then
		\[
		\mathcal{L}=D_x^{2\alpha}+c+\partial_x^{-2}\partial_y^2-f'(\ff)
		\]
		we have $\mathcal{L}\ff=f(\ff)-\ff f'(\ff)$, so 
		\[
		\left<\mathcal{L}\ff,\ff\right>=\int_{\mathbb R^2}\ff f(\ff)-\ff^2 f'(\ff)\dd x\dd y=
		\mu_1(1-p_1^2)K_1(\ff)+\mu_2(1-p_2^2)K_2(\ff)
		\]
		and
		\begin{align*}
			\left<\mathcal{L}\ff,x\ff_x\right>&=\left<\mathcal{L}\ff,y\ff_y\right>
			=\int_{\mathbb R^2}\ff f(\ff)-2F(\ff)\dd x\dd y\\
			&=\mu_1(p_1-1)K_1(\ff)+\mu_2(p_2-1)K_2(\ff).
		\end{align*}
		Next, since
		\begin{align*}
			\mathcal{L}(x\ff_x)&=D_x^{2\alpha}(x\ff_x)+cx\ff_x+\partial_x^{-2}\partial_y^2(x\ff_x)-f'(\ff)x\ff_x\\
			&=xD_x^{2\alpha}\ff_x+2\alpha D_x^{2\alpha}\ff+cx\ff_x+\partial_y^2(x\partial_x^{-1}\ff-2\partial_x^{-2}\ff)-f'(\ff)x\ff_x\\
			&=x\left(D_x^{2\alpha}\ff_x+c\ff_x+\partial_y^2\partial_x^{-2}\ff_x-f'(\ff)\ff_x\right)+2\alpha D_x^{2\alpha}\ff
			-2\partial_y^2\partial_x^{-2}\ff\\
			&=2\alpha D_x^{2\alpha}\ff-2\partial_y^2\partial_x^{-2}\ff\\
			&=2\alpha f(\ff)-2c\alpha\ff-(2\alpha+2)\partial_y^2\partial_x^{-2}\ff
		\end{align*}
		we have
		\begin{align*}
			\left<\mathcal{L}(x\ff_x),x\ff_x\right>&=\int_{\mathbb R^2}\left(2\alpha f(\ff)-2c\alpha\ff-(2\alpha+2)\partial_y^2\partial_x^{-2}\ff\right)x\ff_x\dd x\dd y\\
			&=-2\alpha\int_{\mathbb R^2}F(\ff)\,dx\,dy+c\alpha\int_{\mathbb R^2}\ff^2\,dx\,dy+(3\alpha+3)\int_{\mathbb R^2}(\partial_x^{-1}\ff_y)^2\dd x\dd y
		\end{align*}
		and
		\begin{align*}
			\left<\mathcal{L}(x\ff_x),y\ff_y\right>&=\int_{\mathbb R^2}\left(2\alpha f(\ff)-2c\alpha\ff-(2\alpha+2)\partial_y^2\partial_x^{-2}\ff\right)y\ff_y\dd x\dd y\\
			&=-2\alpha\int_{\mathbb R^2}F(\ff)\dd x\dd y+c\alpha\int_{\mathbb R^2}\ff^2\dd x\dd y-(\alpha+1)\int_{\mathbb R^2}(\partial_x^{-1}\ff_y)^2\dd x\dd y
		\end{align*}
		Using the Pohozaev identities (see Corollary \ref{poho-c}) these become
		\[
		\left<\mathcal{L}(x\ff_x),x\ff_x\right>=\frac{4\alpha+1}{2}\int_{\mathbb R^2}\ff f(\ff)-2F(\ff)\dd x\dd y
		\]
		and
		\[
		\left<\mathcal{L}(x\ff_x),y\ff_y\right>=-\frac32\int_{\mathbb R^2}\ff f(\ff)-2F(\ff)\dd x\dd y.
		\]
		Finally, since $\mathcal{L}(y\ff_y)=2\partial_x^{-2}\ff_{yy}$, 
		\[
		\left<\mathcal{L}(y\ff_y),y\ff_y\right>
		=\int_{\mathbb R^2}(\partial_x^{-1}\ff_y)^2\dd x\dd y
		=\frac1{2}\int_{\mathbb R^2}\ff f(\ff)-2F(\ff)\dd x\dd y.
		\]
		Combining these gives
		\begin{align*}
			\left<\mathcal{L}\phi_x,\phi_x\right>&=\left<\mathcal{L}\ff,\ff\right>+a^2\left<\mathcal{L}(x\ff_x),x\ff_x\right>
			+b^2\left<\mathcal{L}y\ff_y),y\ff_y\right>\\
			&\quad+2a\left<\mathcal{L}\ff,x\ff_x\right>+2b\left<\mathcal{L}\ff,y\ff_y\right>
			+2ab\left<\mathcal{L}(x\ff_x),y\ff_y\right>\\
			&=\int_{\mathbb R^2}\ff f(\ff)-\ff^2f'(\ff)\dd x\dd y
			+2\left((\alpha+2)a^2-4a+3\right)\int_{\mathbb R^2}\ff f(\ff)-2F(\ff)\dd x\dd y
		\end{align*}
		which shows $\phi_x$ satisfies the hypotheses of Theorem \ref{instability}. 
	\end{proof}
	In our application of Corollary \ref{instability_criterion}, the following estimate will be needed. 
	\begin{lemma}\label{Lp_comparison} Let $\ff$ be a solution of \eqref{gkp} with $f(u)=\mu_1|u|^{p_1-1}u+\mu_2|u|^{p_2-1}u$, where $\mu_1>0$. Then 
		\[
		\int_{\mathbb R^2}|u|^{p_1+1}\dd x\dd y\geq\left(\left(\frac{c}{\mu_1}\right)^\theta-\frac{\mu_2^+}{\mu_1}\right)\int_{\mathbb R^2}|u|^{p_2+1}\dd x\dd y
		\]
		for all $c>0$, where $\theta=\frac{p_1-p_2}{p_1-1}$ and $\mu_2^+=\max\{0,\mu_2\}$.  
	\end{lemma}
	\vskip 10pt
	\begin{proof} Suppose instead that 
		\[
		\int_{\mathbb R^2}|u|^{p_1+1}\dd x\dd y<\left(\left(\frac{c}{\mu_1}\right)^\theta-\frac{\mu_2^+}{\mu_1}\right)\int_{\mathbb R^2}|u|^{p_2+1}\dd x\dd y
		\]
		for some $c>0$. Since $N(\ff)=2I(\ff)\geq c\|\ff\|_{L^2(\mathbb R^2)}^2$, applying H\"older's inequality gives
		\begin{align*}
			\int_{\mathbb R^2}|\ff|^{p_2+1}\dd x\dd y&\leq\left(\int_{\mathbb R^2}u^2\dd x\dd y\right)^\theta
			\left(\int_{\mathbb R^2}u^{p_1+1}\dd x\dd y\right)^{1-\theta}\\
			&\leq \left(\frac{\mu_1}{c}\right)^\theta\left(\int_{\mathbb R^2}|u|^{p_1+1}+\frac{\mu_2}{\mu_1}|u|^{p_2+1}\dd x\dd y\right)^\theta
			\left(\int_{\mathbb R^2}|u|^{p_1+1}\dd x\dd y\right)^{1-\theta}\\
			&<\left(\frac{\mu_1}{c}\right)^\theta\left(\left(\frac{c}{\mu_1}\right)^\theta
			\int_{\mathbb R^2}|u|^{p_2+1}\dd x\dd y\right)^\theta
			\left(\left(\frac{c}{\mu_1}\right)^\theta\int_{\mathbb R^2}|u|^{p_2+1}\dd x\dd y\right)^{1-\theta}\\
			&= \int_{\mathbb R^2}|\ff|^{p_2+1}\dd x\dd y,
		\end{align*}
		a contradiction. 
	\end{proof}
	\vskip 10pt
	\begin{theorem} Let $f(u)=\mu_1|u|^{p_1-1}u+\mu_2|u|^{p_2-1}u$ and let $\ff$ be a ground state with speed $c$. 
		\begin{enumerate}[(a)]
			\item If $\mu_1>0$, $\mu_2\geq0$ and $p_1>p_2\geq \frac{5\alpha+2}{\alpha+2}$, then $\mathcal{O}(\ff)$ is unstable.
			\item If $\mu_1>0$, $\mu_2\geq0$ and $p_1>\frac{5\alpha+2}{\alpha+2}>p_2$, then there exists $c(p_1,p_2,\mu_1,\mu_2)$ such that $\mathcal{O}(\ff)$ is unstable for $c>c(p_1,p_2,\mu_1,\mu_2)$.
			\item If $\mu_1>0$, $\mu_2<0$ and $p_1>\frac{5\alpha+2}{\alpha+2}>p_2$, then $\mathcal{O}(\ff)$ is unstable.
			\item If $\mu_1>0$, $\mu_2<0$ and $p_1>p_2\geq\frac{5\alpha+2}{\alpha+2}$, then there exists $c(p_1,p_2,\mu_1,\mu_2)$ such that $\mathcal{O}(\ff)$ is unstable for $c>c(p_1,p_2,\mu_1,\mu_2)$.
		\end{enumerate}
	\end{theorem}
	
	\vskip 10pt

	\begin{proof} 
		First assume $\mu_1>0$ and $\mu_2\geq0$. 
		The quadratic $(\alpha+2)a^2-4a+3$ is minimized when $a=\frac{2}{\alpha+2}$ and for this choice of $a$ we have 
		\[
		\mathcal{K}_f(\ff)=
		\mu_1\frac{p_1-1}{p_1+1}\left(\frac{5\alpha+2}{\alpha+2}-p_1\right)\|\ff\|_{L^{p_1+1}(\mathbb R^2)}^{p_1+1}+\mu_2\frac{p_2-1}{p_2+1}\left(\frac{5\alpha+2}{\alpha+2}-p_2\right)\|\ff\|_{L^{p_2+1}(\mathbb R^2)}^{p_2+1}
		\]
		This is clearly negative when  $p_1>p_2\geq \frac{5\alpha+2}{\alpha+2}$, which proves (a). If $p_1>\frac{5\alpha+2}{\alpha+2}>p_2$ the first term is negative and the second term is positive. However, by Lemma \ref{Lp_comparison} the first term dominates the second term for large $c$, which proves (b). Next assume $\mu_1>0$ and  $\mu_2\leq0$. If $p_1>\frac{5\alpha+2}{\alpha+2}\geq p_2$, then the first term is negative and the second non-positive, so again $\mathcal{K}_f(\ff)$ is negative. This proves (c). Finally, if $p_1>p_2\geq\frac{5\alpha+2}{\alpha+2}$ the first term is negative and by Lemma \ref{Lp_comparison} dominates the second term for sufficiently large $c$, and thus (d) holds. 
	\end{proof}

	\vskip 10pt
	We conclude this section by presenting results of numerical approximations of $d''(c)$ for the odd nonlinearities $f(u)=|u|^{p_1-1}u\pm |u|^{p_2-1}u$ for various pairs $(p_1,p_2)$ with $p_1>p_2$. The results for even and mixed parity nonlinearities are similar. We use the algorithm presented in \cite{esfahani-levan-kdv5} to numerically approximate ground states and then use these approximations to compute $d''(c)$. The method is inspired by the fact that ground states minimize $S$ subject to $P=0$ and consist of a gradient descent combined with a rescaling to maintain the constraint $P=0$. Figures \ref{F:alpha=1_stable}, \ref{F:alpha=1_mixed} and \ref{F:alpha=1_unstable} show the plots of $d''(c)$ for sums of two powers with $\alpha=1$. They illustrate that when $p_2\geq\frac73$ we have stability for all $c>0$ (within the range of computation performed), when $p_1\leq\frac73$ we have instability for all $c>0$, and when $p_1>\frac73>p_2$ we have stability for small $c$ and instability for large $c$. Figure \ref{F:alpha=2_sums} illustrates the same behavior for $\alpha=2$, where the critical exponent is $s_c=3$. For differences of two powers, the behavior appears to depend on $p_1$ and $p_1+p_2$. For $\alpha=1$ and $p_1<7/3$, Figures \ref{F:alpha=1_diff_stable} and \ref{F:alpha=1_diff_mixed} show that we may either have stability for all $c$ or a transition from instability for small $c$ to stability for large $c$, while for $p_1>7/3$ we appear to have instability for all $c$ (Figure \ref{F:alpha=1_diff_unstable}). Figures \ref{F:alpha=2_diff_stable_mixed} and \ref{F:alpha=2_diff_unstable} show results for differences of powers with $\alpha=2$. We note that these results resemble those obtained by Ohta in \cite{ohta} for the Schr\"odinger equation with double power nonlinearity. 
	
	\begin{figure}[ht!]
		\begin{center}
			\scalebox{.5}{\includegraphics{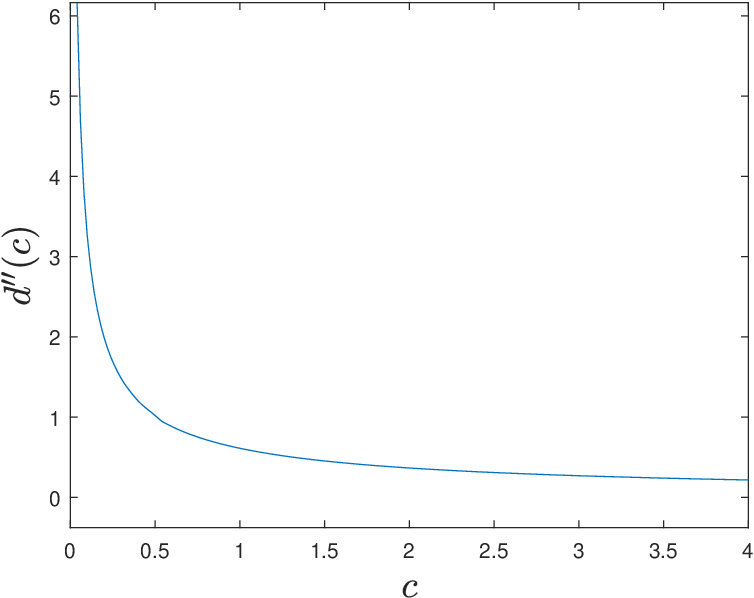}}\qquad
			\scalebox{.5}{\includegraphics{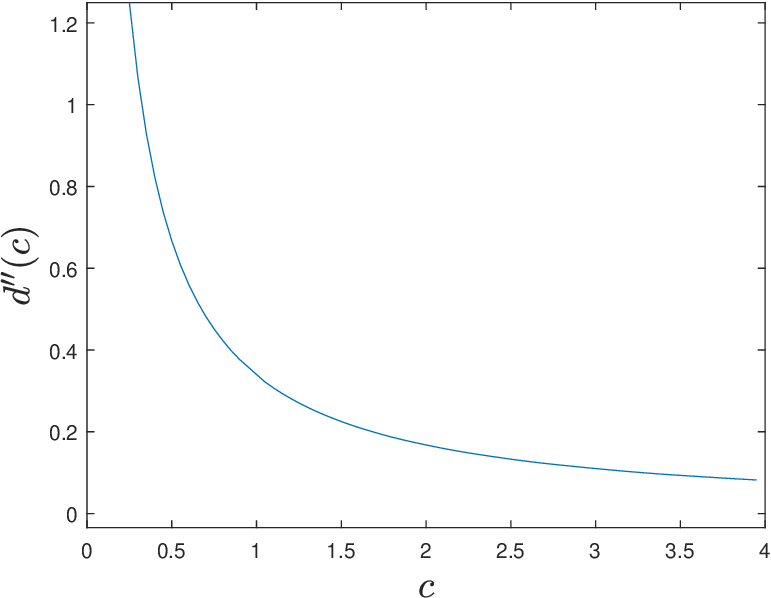}}
		\end{center}
		\caption{Plots of $d''(c)$ for $\alpha=1$ and $f(u)=|u|^{p_1-1}u+|u|^{p_2-1}u$ with $(p_1,p_2)=(2.2,2)$ (left) and $(p_1,p_2)=(7/3,2)$ (right). In both cases we have $d''(c)>0$ for all $c$ in the domain considered.}\label{F:alpha=1_stable}
	\end{figure}

	\begin{figure}[ht!]
		\begin{center}
			\scalebox{.5}{\includegraphics{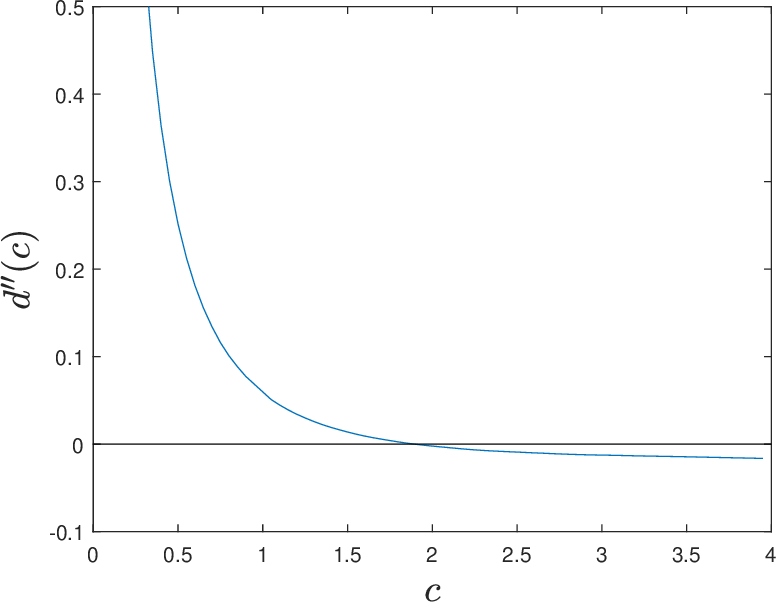}}\qquad
			\scalebox{.5}{\includegraphics{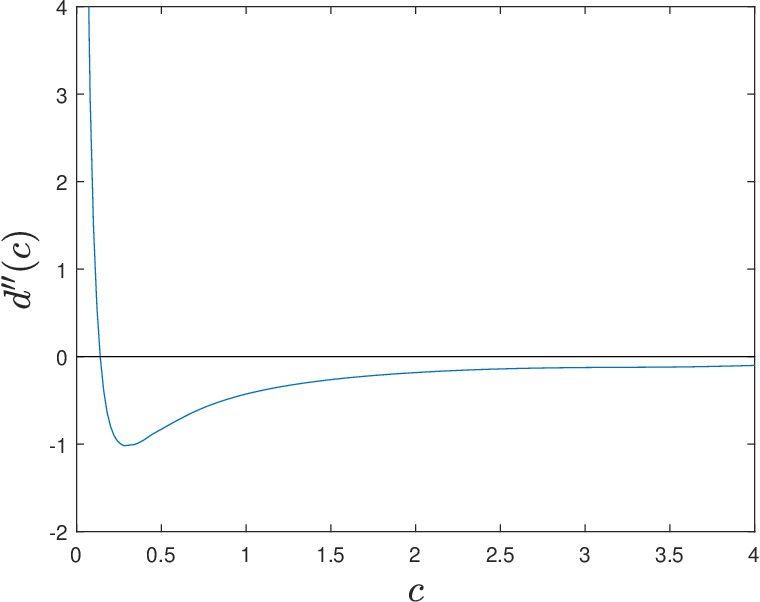}}
		\end{center}
		\caption{Plots of $d''(c)$ for $\alpha=1$ and $f(u)=|u|^{p_1-1}u+|u|^{p_2-1}u$ with $(p_1,p_2)=(2.5,2)$ (left) and $(p_1,p_2)=(3,2)$ (right). In these cases, we have $d''(c)>0$ for small $c$ and $d''(c)<0$ for large $c$.}\label{F:alpha=1_mixed}
	\end{figure}
	
	\begin{figure}[ht!]
		\begin{center}
			\scalebox{.5}{\includegraphics{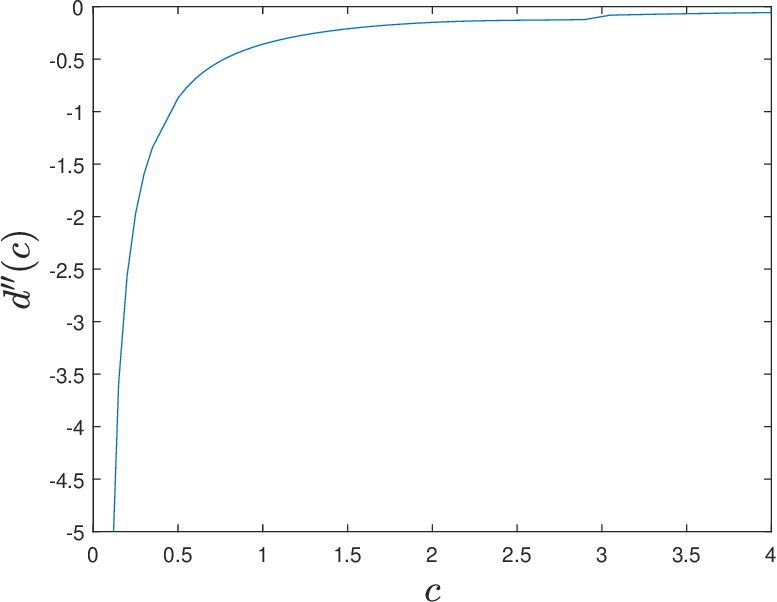}}
			\scalebox{.5}{\includegraphics{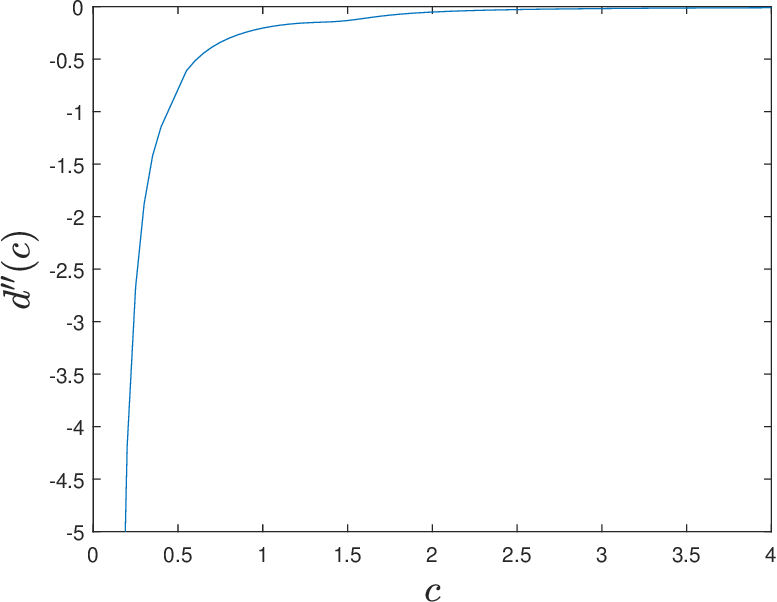}}
		\end{center}
		\caption{Plots of $d''(c)$ for $\alpha=1$ and $f(u)=|u|^{p_1-1}u+|u|^{p_2-1}u$ with $(p_1,p_2)=(3,7/3)$ (left) and $(p_1,p_2)=(4,3)$ (right). In both cases we have $d''(c)<0$ for all $c$ in the domain considered.}\label{F:alpha=1_unstable}
	\end{figure}
	
	\begin{figure}[ht]
		\begin{center}
			\scalebox{.4}{\includegraphics{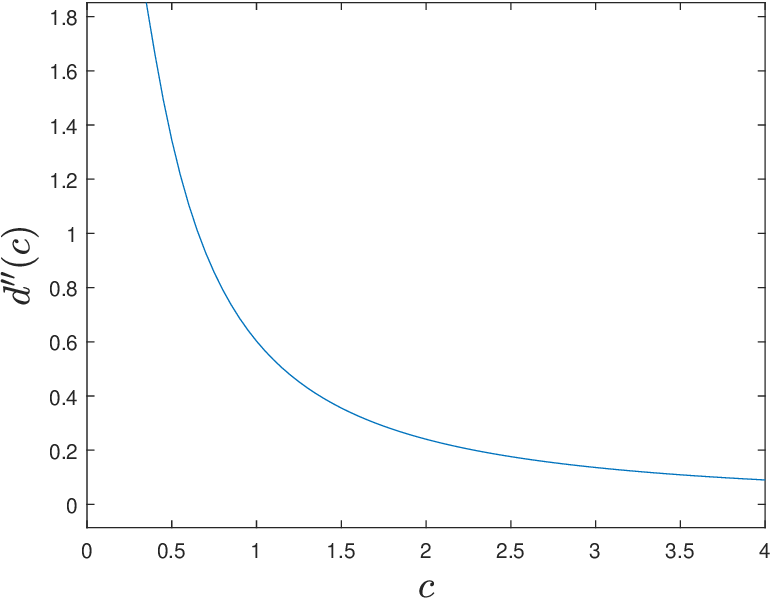}}
			\scalebox{.4}{\includegraphics{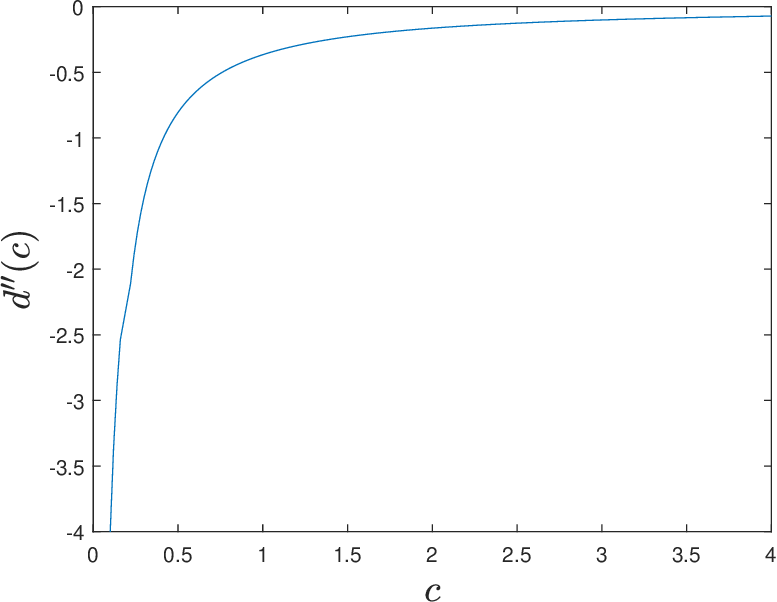}}
			\scalebox{.4}{\includegraphics{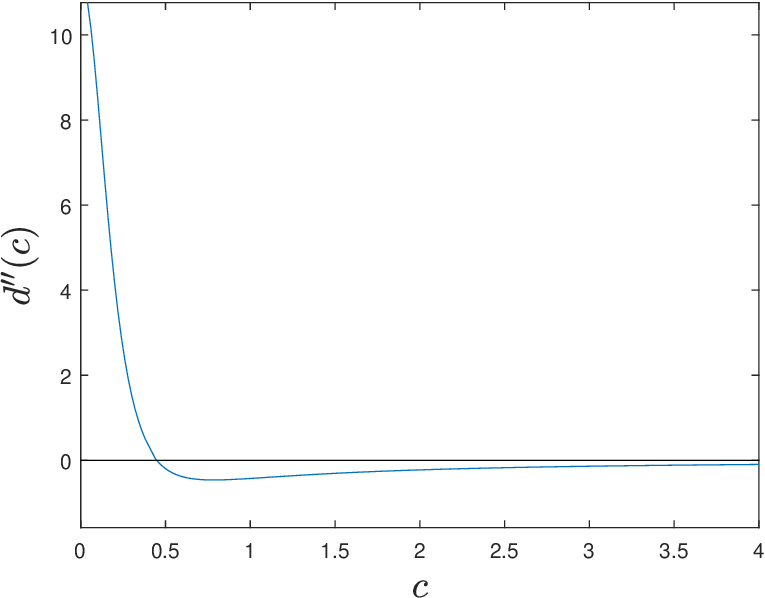}}
		\end{center}
		\caption{Plots of $d''(c)$ for $\alpha=2$ and $f(u)=|u|^{p_1-1}u+|u|^{p_2-1}u$ with $(p_1,p_2)=(3,2)$ (left), $(p_1,p_2)=(4,3)$ (middle) and $(p_1,p_2)=(4,2)$ (right). }\label{F:alpha=2_sums}
	\end{figure}
	
	\begin{figure}[ht!]
		\begin{center}
			\scalebox{.5}{\includegraphics{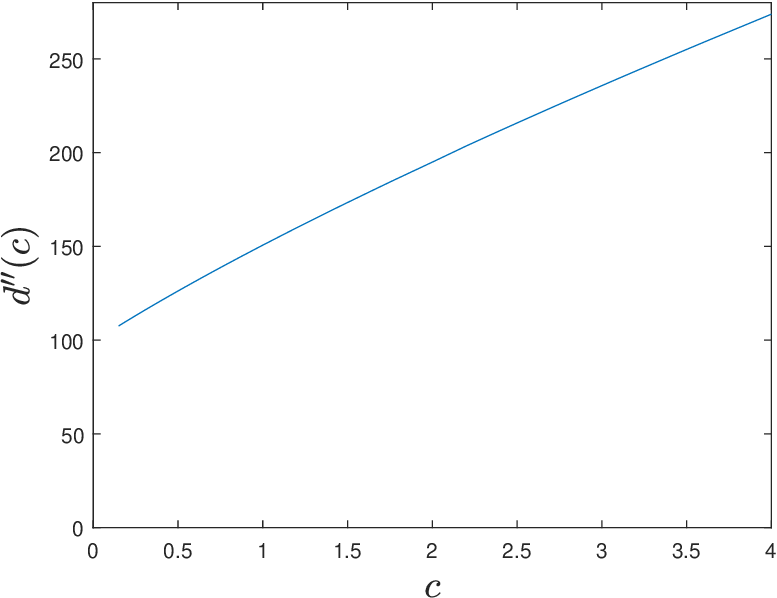}}\qquad
			\scalebox{.5}{\includegraphics{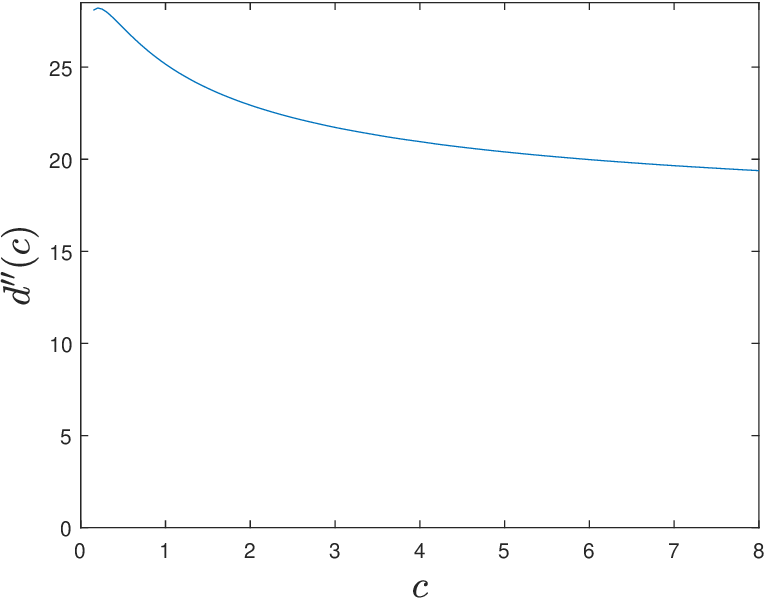}}
		\end{center}
		\caption{Plots of $d''(c)$ for $\alpha=1$ and $f(u)=|u|^{p_1-1}u-|u|^{p_2-1}u$ with $(p_1,p_2)=(1.6,1.2)$ (left) and $(p_1,p_2)=(1.8,1.4)$ (right). }\label{F:alpha=1_diff_stable}
	\end{figure}

	\begin{figure}[ht!]
		\begin{center}
			\scalebox{.5}{\includegraphics{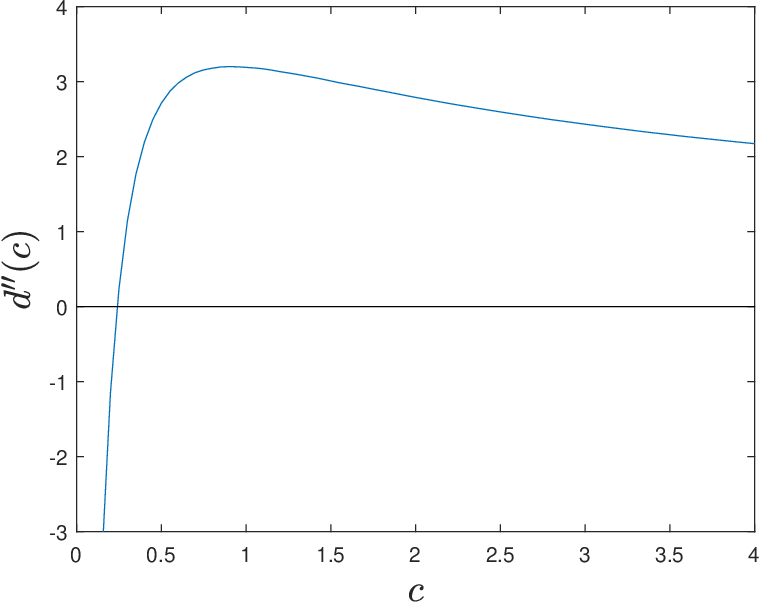}}\quad
			\scalebox{.5}{\includegraphics{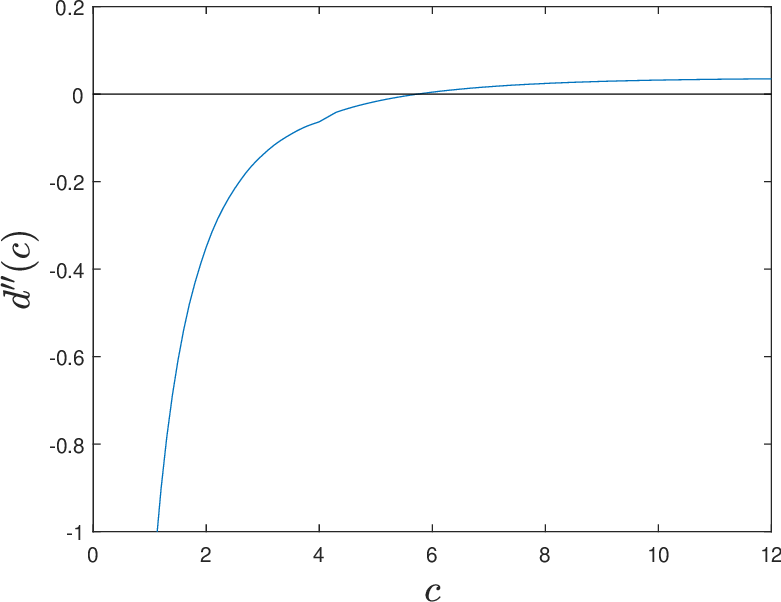}}
		\end{center}
		\caption{Plots of $d''(c)$ for $\alpha=1$ and $f(u)=|u|^{p_1-1}u-|u|^{p_2-1}u$ with $(p_1,p_2)=(2,1.6)$ (left) and $(p_1,p_2)=(2.2,1.8)$ (right). }\label{F:alpha=1_diff_mixed}
	\end{figure}

	\begin{figure}[ht!]
		\begin{center}
			\scalebox{.5}{\includegraphics{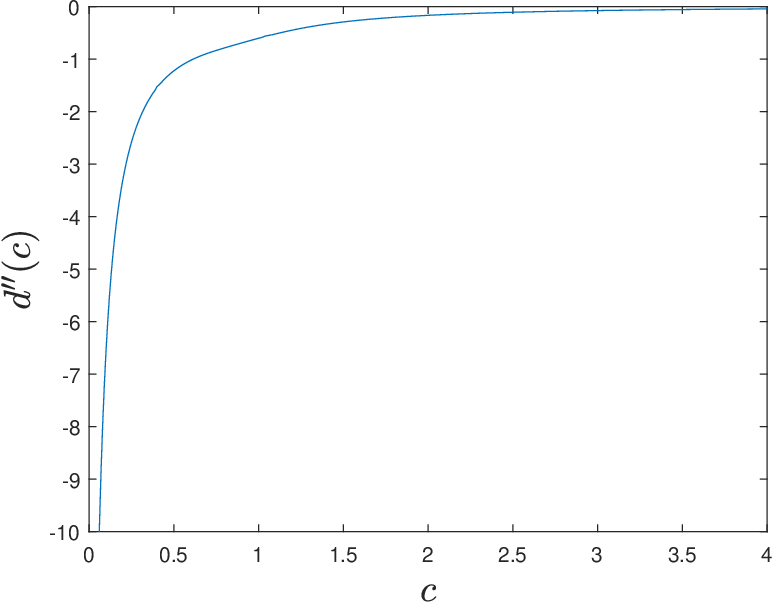}}
			\scalebox{.5}{\includegraphics{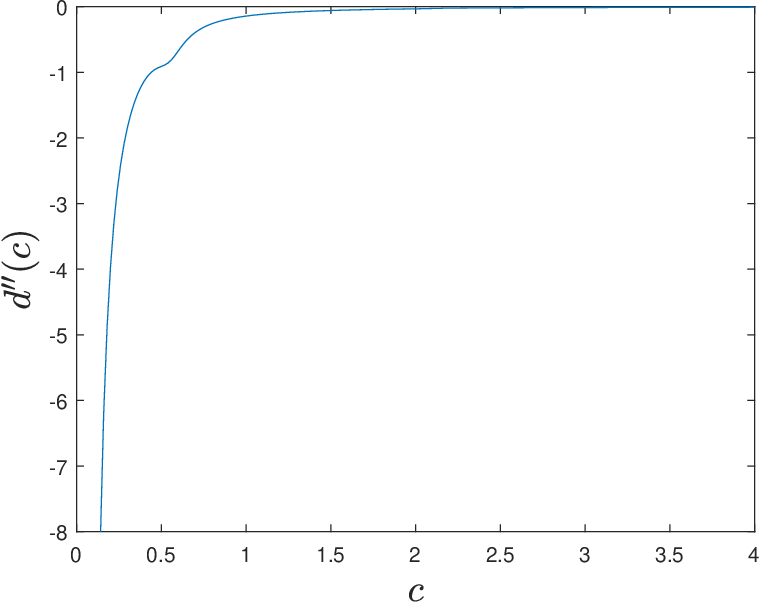}}
		\end{center}
		\caption{Plots of $d''(c)$ for $\alpha=1$ and $f(u)=|u|^{p_1-1}u-|u|^{p_2-1}u$ with $(p_1,p_2)=(3,2)$ (left) and $(p_1,p_2)=(4,3)$ (right). }\label{F:alpha=1_diff_unstable}
	\end{figure}

	\begin{figure}[ht]
		\begin{center}
			\scalebox{.5}{\includegraphics{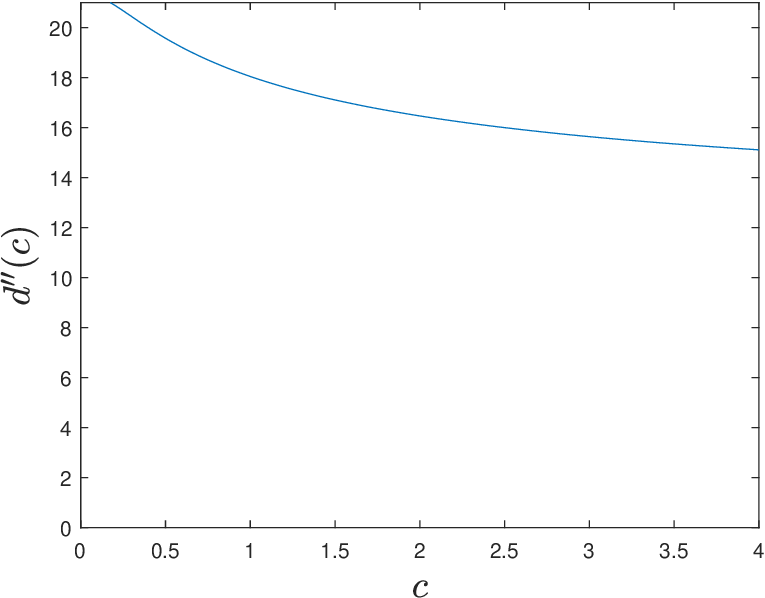}}
			\scalebox{.5}{\includegraphics{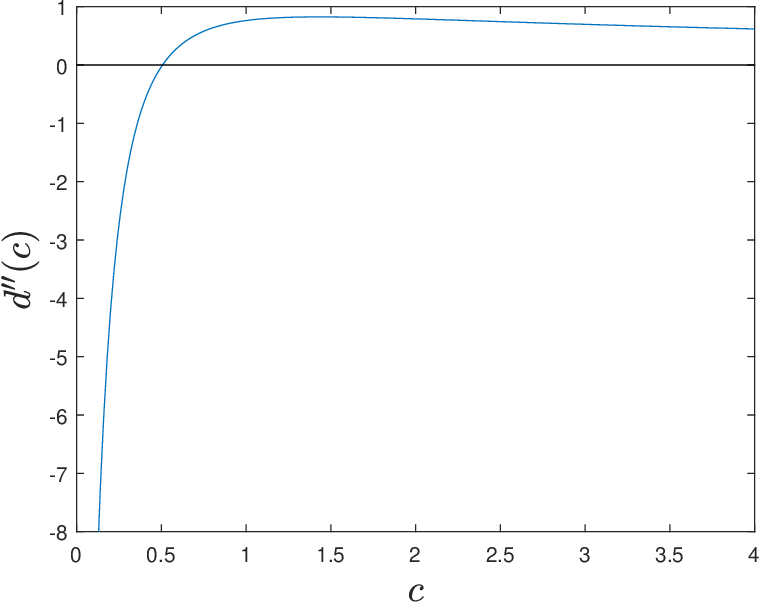}}
		\end{center}
		\caption{Plots of $d''(c)$ for $\alpha=2$ and $f(u)=|u|^{p_1-1}u-|u|^{p_2-1}u$ with $(p_1,p_2)=(2.5,2)$ and $(p_1,p_2)=(2,1.5)$. }
	\end{figure}\label{F:alpha=2_diff_stable_mixed}
	
	\begin{figure}[ht]
		\begin{center}
			\scalebox{.5}{\includegraphics{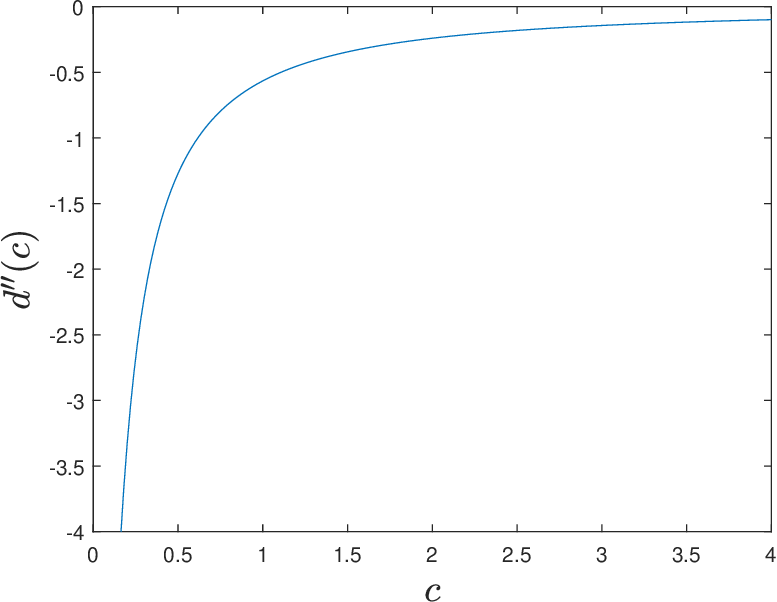}}
			\scalebox{.5}{\includegraphics{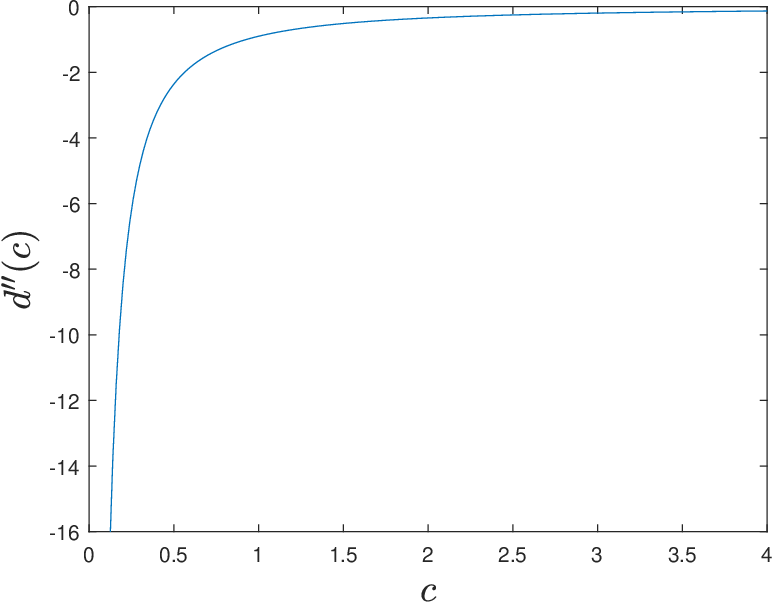}}
		\end{center}
		\caption{Plots of $d''(c)$ for $\alpha=2$ and $f(u)=|u|^{p_1-1}u-|u|^{p_2-1}u$ with $(p_1,p_2)=(4,2)$ and $(p_1,p_2)=(4,3)$. }
	\end{figure}\label{F:alpha=2_diff_unstable}
	
	\section{Blow-up and Uniform Boundedness}\label{sect-bound-blow}
	
	In this section,  we derive a
	sharp threshold between blow-up and global existence (uniformly boundedness in the energy space).
	First, we state the following local
	well-posedness for the initial-value problem associated with \eqref{main-fifthkp}. This local existence can be obtained along the lines of arguments of \cite{saut}. See also \cite{tom}. 
	
	\begin{theorem}\label{cauchy-problem}
		Let  $u_0\in X^s$, $s\geq\al+1$. Then there exists $T>0$ such that \eqref{main-fifthkp} has a unique solution $u(t)$ with $u(0)=u_0$  satisfying
		\[
		u\in C([0,T);X^s)\cap C^1\left([0,T);H^{s-\al-1}(\rr^2)\right),
		\qquad \partial_x^{-1}u_y\in C\left([0,T);H^{s-1}(\rr^2)\right);
		\]
		and if $\partial_x^{-2}(u_0)_{yy}\in L^2(\rr^2)$, one has
		\[
		u_t\in L^\infty\left((0,T);X^0\right),\qquad \partial_x^{-1}u_{yt}\in L^\infty\left([0,T);H^{-1}(\rr^2)\right).
		\]
		Furthermore we have $M(u(t))=M(u_0)$ and $E(u(t))=E(u_0)$.
	\end{theorem}
	
	Next, we obtain the conditions under which the local solutions are uniformly bounded in the energy space. We start by showing that the best constant $\rho_p$ in \eqref{bin} may be expressed in terms of the ground state solutions of \eqref{gkp} with homogeneous nonlinearity,
	\begin{equation}\label{zeroalpha}
		D^{2\al}_x\ff+\partial_x^{-2}\ff_{yy}+\ff=|\ff|^{p-1}\ff.
	\end{equation}

	\begin{theorem}\label{sharp-constant}
		Let $p>1$. Then the optimal constant $\rho_p$  in \eqref{bin} is such that
		\begin{equation}\label{best-c}
			\rho_p^{-1}=
			\al^{-1}\left(\frac{p-1}{p+1}\right)k_p^{c'_p}\left(\frac{\al}{2}\right)^{\frac{p-1}{4}}\|\ff\|_\lt^{p-1}
			=\al^{-1}\left(\frac{p-1}{p+1}\right)k_p^{c_p}\left(\frac{2}{\al}\right)^{\frac{p-1}{4}} 
			m^{\frac{p-1}{2}},
		\end{equation}
		where  $c_p'=c_p-\frac{p-1}{2}$,
		\[
		k_p=\frac{3\al+2+p(\al-2)}{2(p-1)}
		\] 
		and $\ff$ is a ground state of \eqref{zeroalpha}.
	\end{theorem}
	\begin{proof}
		The proof is based on the ideas of \cite[Theorem 1.2]{zamp}, so we omit the details.
	\end{proof}
	
	We next obtain some Pohozaev-type identities related to the solutions of \eqref{zeroalpha}.
	\begin{lemma}\label{poho-homo}
		Let $\ff$ be a ground state of \eqref{zeroalpha}, then
		\begin{eqnarray}
			2\|\nd \ff_y\|_\lt^2=\al \|\dx\ff\|_\lt^2=	\left(\frac{p-1}{p+1}\right)\|\ff\|_{L^{p+1}(\rr^2)}^{p+1},\\
			\|\dx \ff\|_\lt^2	=k_p^{-1} \|\ff\|_\lt^2.
		\end{eqnarray}
	\end{lemma} 
	\begin{proof}
		The proof follows from the ideas of \cite{dbs-1}. 
	\end{proof}
	
		\begin{theorem}
				Let $s\geq\al+1$, $\mu_1<0<\mu_2$ and $\po>\poo$. Let $u\in C([0,T);H^s(\rt))$ with $\nd u_y\in C([0,T);H^{s-1}(\rt))$ be the solution obtained in Theorem \ref{cauchy-problem}, corresponding to the initial data $u_0$.
		Then the solution $u$ satisfies
		\[
		\sup_{t\in[0,T)}\|u(t)\|_{\dot{X}_\al}^2\lesssim 
		E(u_0)+C\|u_0\|_\lt^2,
		\]
		where $C=C(\po,\poo,\mu_1,\mu_2)>0$.
	\end{theorem}
	\begin{proof}
		The proof follows easily from the conservation laws of  the energy and the momentum. Indeed, one can see that
		\[
		\begin{split}
			\frac12\|u(t)\|_{\dot{X}_\al}^2
			&= E(u(t))
			+\int_\rt F(u(t))\dz\\
			&
			\leq E(u_0)+\frac{\mu_2(\po-\poo)}{(\po-1)(\poo+1)}
			\paar{\frac{\mu_2(\poo-1)(\po+1)}{-\mu_1(\po-1)(\poo+1)}}
			^\frac{\poo-1}{\po-\poo}
			\norm{u(t)}_\lt^2
			\\
			&=
			E(u_0)+\frac{\mu_2(\po-\poo)}{(\po-1)(\poo+1)}
			\paar{\frac{\mu_2(\poo-1)(\po+1)}{-\mu_1(\po-1)(\poo+1)}}
			^\frac{\poo-1}{\po-\poo}\norm{u_0}_\lt^2.
		\end{split}
		\]
		Here, we have used the fact that $u^{-2}F(u)$ has a unique maximum point.
	\end{proof}

	\begin{theorem}\label{subcritical-th}
		Let $s\geq\al+1$, $\mo,\moo>0$ and $\po>\poo$. Let $u\in C([0,T);H^s(\rt))$ with $\nd u_y\in C([0,T);H^{s-1}(\rt))$ be the solution obtained in Theorem \ref{cauchy-problem}, corresponding to the initial data $u_0$.
		\begin{enumerate}[(i)]
			\item If $\po,\poo<s_c =1+\frac{4\al}{2+\al}$, then $u(t)$ is uniformly bounded in $\x$ 	for $t\in [0, T )$.
			\item If $p_1=s_c$  and 
			\begin{equation}\label{crit-in}
				1-\frac{2\mu_1\rho_{p_1}}{\po+1}\|u_0\|_\lt^{2c_{p_1}}>0,
			\end{equation}
			then $u(t)$ is uniformly bounded in $\x$  	for $t\in [0, T )$. Particularly,  $u(t)$ is uniformly bounded in $\x$ if
			\[
			\|u_0\|_\lt^{2c_{p_1}}<\frac{1}{ \mu_1}
			\frac{2}{2+3\al}  \left(\frac{\al}{2}\right)^{\frac{\al}{\al+2}}\|\ff_{s_c}\|_\lt^{\frac{4\al}{\al+2}},
			\]
			where $\ff_{s_c}$ is the  ground state of
			\begin{equation}\label{criti-gs}
				\left(u+D_x^{2\al} u-u^{s_c} \right)_{xx}  + u_{yy}=0. 
			\end{equation}
		\end{enumerate} 
	\end{theorem}

	\begin{proof}
		Let $u\in  C([0, T ); X^s)$ be the solution of \eqref{main-fifthkp} with the
		initial data $u(0)=u_0$. Then by using the invariants $E$ and $M$, we have  that
		\begin{equation}\label{energy-estimate-1}
			\begin{split}
				E(u_0)&\geq\frac{1}{2}\|u\|_{\xx}^2-\frac{\mo}{\po+1}\|u\|_{L^{\po+1}(\rt)}^{\po+1}-\frac{\moo}{\poo+1}\|u\|_{L^{\poo+1}(\rt)}^{\poo+1}\\
				&\geq \frac{1}{2}\|u\|_{\xx}^2-\frac{\mo \rho_{p_1}}{\po+1}\|u\|_\lt^{2c_{p_1}}\|\dx u\|_\lt^{(\po-1)/\al}\|\nd u_y\|_\lt^{(\po-1)/2}\\
				&\quad-\frac{\moo \rho_{p_2}}{\poo+1}\|u\|_\lt^{2c_{p_2}}\|\dx u\|_\lt^{(\poo-1)/\al}\|\nd u_y\|_\lt^{(\poo-1)/2}\\
				& \geq \frac{1}{2}\|u\|_{\xx}^2-
				\frac{\mo \rho_{p_1}}{\po+1}\|u\|_\lt^{2c_{p_1}}\|u\|_{\xx}^{\frac{(\al+2)(\po-1)}{2\al}}-\frac{\moo \rho_{p_2}}{\poo+1}\|u\|_\lt^{2c_{p_2}}\|u\|_{\xx}^{\frac{(\al+2)(\poo-1)}{2\al}}\\
				&=h(\|u\|_{\xx}),
			\end{split}
		\end{equation}
		where
		\[
		h(z)=\frac{1}{2}z^2-
		\frac{ \mo \rho_{p_1}}{\po+1}\|u_0\|_\lt^{2c_{p_1}}z^{\frac{(\al+2)(\po-1)}{2\al}}-\frac{ \moo \rho_{p_2}}{\poo+1}\|u_0\|_\lt^{2c_{p_2}}z^{\frac{(\al+2)(\poo-1)}{2\al}}
		\]
		and $\rho_{p_1},\rho_{p_2}>0$ are the same as in Theorem \ref{sharp-constant}.
		This   immediately implies for  $p_1,p_2<s_c$ that $\|u\|_\x$ is uniformly bounded for all $t\in[0, T )$. 
		The uniform bound still holds for $p_1=s_c>p_2$ provided	
		\[
		1-\frac{2\mu_1\rho_{p_1}}{\po+1}\|u_0\|_\lt^{2c_{p_1}}>0.
		\]
		On the other hand, \eqref{best-c} gives the condition of boundedness of $u$ in terms of the ground state of \eqref{zeroalpha}.
	\end{proof}

	For the supercritical nonlinearities or in the case of the combined supercritical and critical nonlinearities, we need to impose more additional conditions in terms of the best constant of \eqref{bin}.

	\begin{theorem}\label{supercritical-th}
		Let $\mo,\moo>0$, $\poo>\po\geq s_c$. Suppose that $u$ is the solution of \eqref{main-fifthkp} as in Theorem \ref{subcritical-th}, corresponding to the initial data $u_0$.
		\begin{enumerate}[(i)]
			\item If $\po=s_c$,   $\|u_0\|_\xx<z_0$, $E(u_0)<h(z_0)$ and \eqref{crit-in} holds, then $u(t)$ is uniformly bounded in $\x$ 	for $t\in [0, T )$, 
			where $h$ is as in the proof of Theorem \ref{subcritical-th},
			\[
			z_0=\left(\frac{A}{B}\right)^{\frac{s_c-1}{2p_2-s_c+1}},
			\]
			and
			\[
			A=1-\frac{2\rho_{p_1}\mo}{\po+1}\|u_0\|_\lt^{2c_{\po}},\qquad
			B=\frac{\rho_{p_2}\moo}{\poo+1}(\poo-1)\left(\frac{\al+2}{2\al}\right)\|u_0\|_\lt^{2c_{\poo}}.
			\]
			\item  If $\po>s_c$,
			there exists 
			\[
			z_1=z_1\left(\al,p_1,p_2,\rho_{p_1},\rho_{p_2},\mu_1,\mu_2,\|u_0\|_\lt\right)>0
			\]
			such that if $$E(u_0)< \frac{(p_1-1)(\al+2)-4\al}{2(p_1-1)(\al+2)}z_1^2,$$ and $\|u_0\|_\xx<z_1$, then the solution $u(t)$ is uniformly bounded in the energy space.
		\end{enumerate} 
	\end{theorem}
	
	\begin{proof}
		(i) 	Let $h$ be as the same in the proof of  \ref{subcritical-th}. Then,
		\[
		E(u(t))\geq  h(\|u(t)\|_\xx).
		\]
		The function $h$ is continuous on $[0,\infty)$ and
		\[
		h'(z)=Az-Bz^{(p_2-1)(\frac{\al+2}{2\al})-1}.
		\]
		Inequality \eqref{crit-in} shows that $h'(z)=0$ has only a positive root $z_0$. 	
		Hence, $h$ is increasing on the interval $[0,z_0)$, decreasing on $[z_0,+\infty)$ and
		\[
		h_{\max}=h(z_0)=A\frac{(p_2-1)(\al+2)-4\al}{2(p_2-1)(\al+2)}z_0^2.
		\]
		The invariant $E(u(t))=E(u_0)$ and $E(u_0)<h(z_0)$ imply that
		\begin{equation}\label{super-eq-1}
			h(\|u(t)\|_\xx)\leq E(u(t))=E(u_0)<h(z_0)
		\end{equation}
		for all $[0,T)$. We show that $\|u(t)\|_\xx<z_0$ for all $[0,T)$, if  $\|u_0\|_\xx<z_0$. This means that $u(t)$ is uniformly bounded in the energy space. Suppose by contradiction that $\|u(t)\|_\xx\geq z_0$ for some $t_1<T$. Then, by continuity of $\|u(t)\|_\xx$, there exists $t_0\in(0,T)$ such that $\|u(t_0)\|_\xx= z_0$. Thus,
		\[
		h(\|u(t_0)\|_\xx)=h(z_0)=h_{\max}.
		\]
		Now, \eqref{super-eq-1} yields with $t=t_0$ that
		\[
		h(\|u(t_0)\|_\xx)=h(z_0)=h_{\max}\leq E(u(t_0))=E(u_0)<h_{\max}.
		\]
		This contradiction gives the desired result. We note by a similar argument that if $\|u(t)\|_\xx>z_0$ for all $[0,T)$, if $\|u_0\|_\xx>z_0$. 
		
		(ii) By an argument similar to (i), we see that 
		\[
		h'(z)=z-B_1z^{e_1-1}-B_2z^{e_2-1},
		\]
		where
		\[
		e_j=(p_j-1)\left( \frac{\al+2}{2\al} \right)
		\]
		and
		\[
		B_j=\frac{\rho_{p_j}\mu_j}{p_j+1}e_j\|u_0\|_\lt^{2c_{p_j}},\quad j=1,2.
		\]
		Set $\tilde{h}(z)=\frac{h'(z)}{z}$. Then,
		\[
		\tilde{h}'(z)=-B_1\left(e_1-1\right)z^{e_1-2}
		-B_2\left(e_2-1\right)z^{e_2-2}.
		\]
		The assumption $p_2>p_1>s_c$ shows for $z>0$ that $\tilde{h}'(z)<0$. So, $\tilde{h}$ is decreasing on $[0,+\infty)$. By the fact $\tilde{h}(0)=1$, there exists $z_1>0$ such that $\tilde{h}(z_1)=0$, and then
		\[
		h(z_1)=\left(\frac{1}{2}-\frac{1}{e_1}\right)z_1^2+\frac{\rho_{p_2}\mu_2}{e_1(p_2+1)}\left( \frac{\al+2}{2\al} \right)(p_2-p_1)z_1^{e_2}.
		\]
		We now obtain from the conservation of energy  together with $E(u_0)<\frac{e_1-2}{2e_1}z_1^2$ that
		\[
		h 	(\|u(t)\|_\xx)\leq E(u(t))=E(u_0)
		\leq
		\left(\frac{1}{2}-\frac{1}{e_1}\right)z_1^2+\frac{\rho_{p_2}\mu_2}{e_1(p_2+1)}\left( \frac{\al+2}{2\al} \right)(p_2-p_1)z_1^{e_2}
		=h(z_1).
		\]
		By the same argument as in (i), we conclude that $\|u(t)\|_\xx<z_1$ for all $t\in[0,T)$, if $\|u_0\|_\xx<z_1$. This gives the uniform boundedness of $u$.
	\end{proof}

	
	\begin{proposition}\label{vir-cd}
		Let $\phi\in C^1(\rr)$ be a nonegative measurable function  satisfying $|\phi'(y)|\lesssim \phi(y) $ and  $u$ be the solution of Theorem \ref{cauchy-problem}. Then	 $\phi^{1/2}(y)u\in L^\infty((0,T); L^2(\rr^2))$, if $\phi^{1/2}(y)|u_0|\in L^2(\rr^2)$. 
	\end{proposition}		 
	
	The proof of this proposition is similar to one of Theorem 3.3 in \cite{saut}.
	
	Now by Proposition \ref{vir-cd}, the quantity 
	\[
	\nii(u)=\int_\rt \phi(y)u^2(t)\dz
	\]
	is well defined as soon as $u_0\in X^s$ and $\phi^{1/2}(y)u_0\in L^2(\rt)$.
	
	\begin{theorem}\label{virial-theorem}
		Let  $\phi\in C^4(\rr)$ as in Proposition \ref{vir-cd}.  Suppose that $u$  is a solution, obtained from Theorem \ref{cauchy-problem}, such that $\phi^{1/2}(y)u_0\in\lt$.
		Then $u(t)$ satisfies
		\begin{equation}\label{virial-id}
			\begin{split}
				\frac{1}{2}	\frac{\dd^2}{\dd t^2}\nii(u)&=-\frac{1}{2}\int_\rt\phi^{(4)} (\nd u)^2\dz\\
				&\quad+  2\int_\rt\phi''(y)\left( (\nd u_y)^2- \ps_1\mu_1F_1(u)-\ps_2\mu_2F_2(u)\right)\dz,
			\end{split}
		\end{equation}
	\end{theorem}
	where
	$\ps_j=(p_j-1)/2$. 
	\begin{remark}
		If $\int_\rt|x|u_0^2\dz$ is finite, then $t\mapsto\int_\rt xu^2(t)\dz$ is a $C^1$ function of time and
		\begin{equation}
			\begin{split}
				\frac{\dd}{\dd t}\int_\rt xu^2(t)\dz&=
				(2\al-1)\|\dx u\|_\lt^2  -\|\nd u_y\|^2_\lt 
				+2\int_\rt(F(u)-uf(u))\dz.
		\end{split}\end{equation}
	\end{remark}
	
	We use the   Young inequality for the next result. Recall from the Young inequality  for any $a>0$ and $\poo>\po>1$ that for any $\tau>0$ there exists $C_\tau>0$ such that
	\begin{equation}\label{young-in}
		a^{\po+1}\leq
		\tau a^2+C_\tau a^{\poo+1},\qquad C_\tau=\tau^{\frac{p_2-p_1}{1-p_1}}.
	\end{equation}

	Now we state our blow-up result in terms of the energy of initial data and the sign of $\mu_1$ and $\mu_2$.
	\begin{theorem}\label{blow-up-theorem-1}
		Let  $u\in C([0,T);H^s(\rt))$ with $\nd u_y\in C([0,T);H^{s-1}(\rt))$ be the solution obtained in Theorem \ref{cauchy-problem}, corresponding to the initial data $u_0$. If $yu_0\in\lt$, then the solution $u(t)$ blows up in finite time  	in the sense that $T<+ \infty$ must hold in each of the following three cases:
		\begin{enumerate}[(i)]
			\item $\moo>0$, $E(u_0)\leq0$ and $p_2\geq p_1\geq5$;
			\item $\mo<0$, $E(u_0)\leq0$  and $p_2\geq \max\{5, p_1\}$;
			\item $\moo<0<\mo$, $p_2\geq\max\{\frac{4}{\theta}+1,\po\}$ with  $\theta\in(0,1)$, and there is $\tau>0$ such that
			\begin{equation}\label{blow-up-cond}
				\frac{\poo-1}{2}\theta E(u_0) 
				+A_{\po}\tau M(u_0)\leq0 
			\end{equation}
			and 
			\[
			A_{\po}C_\tau- A_\poo\leq0,
			\]
			where $C_\tau$ is as in \eqref{young-in} and
			\[
			A_{\po}=\frac{\mo}{2(\po+1)} |\theta(\poo-1)-\po+1|,\quad A_{\poo}=\frac{\moo\ps_2}{\poo+1}(1-\theta).
			\]
			
			\item $\moo>0$, $p_2\geq\max\{\frac{4}{\theta}+1,\po\}$ with  $\theta\in(0,1)$, and there is $\tau>0$ such that
			\begin{equation}\label{blow-up-cond-1}
				\frac{\poo-1}{2}\theta E(u_0) 
				-A'_{\poo}\tau M(u_0)\leq0 
			\end{equation}
			and 
			\[
			A'_{\po}+ A'_\poo\leq0,
			\]
			where $C_\tau$ is as in \eqref{young-in} and
			\[
			A'_{\poo}=\frac{\moo (p_2-1)(1-\theta)}{2C_\tau(\poo+1)} ,\quad 
			A'_{\po }=\frac{\mo}{\po +1}( \theta(p_2-1)-p_1+1).
			\]
		\end{enumerate}
	\end{theorem} 
	\begin{proof}
		In what follows, we will show for $\phi(y)=y^2$ in Theorem \ref{virial-theorem} that the second derivative of $\nii(u)$  is negative for positive
		times $t$. More precisely, in each of the described three cases,  it follows that $\nii(u(t_0))$ for some $t_0 <T$, and the blow-up result can be
		deduced from the conserved momentum $M$ and the classical Weyl-Heisenberg inequality.
		
		(i) and (ii) We obtain  from \eqref{virial-theorem} and  the conservation of energy $E$ that
		\[
		\begin{split}
			\frac{1}{8}\frac{\dd^2}{\dd t^2}\nii (u)&=\|\nd u_y\|_\lt^2-\tilde{K}(u)\\
			&=\frac{5-\po}{4}\|\nd u_y\|_\lt^2-
			\frac{\poo-\po}{2}K_2(u) 
			-\frac{\po-1}{4}\|\dx u\|_\lt^2
			+\ps_1E(u_0)
			\\
			&=
			\frac{5-\poo}{4}\|\nd u_y\|_\lt^2-
			\frac{\po -\poo}{2 } K_1(u)
			-\frac{\poo-1}{4}\|\dx u\|_\lt^2
			+\ps_2E(u_0),
		\end{split}
		\]
		where
		\begin{equation}\label{ktilde}
			\tilde{K}(u)=\ps_1K_1(u)+ \ps_2K_2(u).
		\end{equation} 
		
		(iii)
		In this case, we have from \eqref{young-in} and the mass conservation that 
		\[
		\begin{split}
			\frac{1}{8}\frac{\dd^2}{\dd t^2}\nii (u)&
			=\|\nd u_y\|_\lt^2
			- (\theta+(1-\theta)) \ps_2 K_2(u)- \ps_1 K_1(u)\\
			&=\left(1-\frac{\poo-1}{4}\theta\right)\|\nd u_y\|_\lt^2 -\frac{\ps_2\theta}{2}\|\dx u\|_\lt^2
			+\ps_2\theta E(u_0)\\
			&\qquad
			-\ps_2(1-\theta)K_2(u)
			-\frac{1}{2}(\po-1-\theta(\poo-1))K_1(u)\\
			&\leq 
			\frac{4-(\poo-1)\theta}{4}\|\nd u_y\|_\lt^2
			-\frac{\ps_2\theta}{2}\|\dx u\|_\lt^2
			+\ps_2\theta E(u_0)\\
			&\qquad +A_{\po}\tau M(u_0)+(A_{\po}C_\tau+ A_\poo )\int_\rt u^{\poo+1}\dz.
		\end{split}
		\]
		Now, if  $\tau>0$ is in such a way that $A_{\po}C_\tau+ A_\poo\leq0$, then $\frac{\dd^2}{\dd t^2}\nii (u)<0$ provided \eqref{blow-up-cond} holds.
		
		Case (iv) is obtained similarly.
	\end{proof}

	We have seen in Theorem \ref{subcritical-th} in the critical case $\po=s_c$ that the solutions of \eqref{main-fifthkp} are uniformly bounded if the initial data is almost less than the $L^2$-norm of the ground state of \eqref{zeroalpha}. In the following, we study the conditions under which  $L^2$-prescribed solutions of \eqref{gkp} exist in the critical case.
	
	For $\varrho>0$, we consider the minimizing problem
	\begin{equation}\label{minim-normal}
		d_\varrho=\inf\{E(u);\; u\in X_\al,\;M(u)=\varrho\}
	\end{equation}
	and the set
	\[
	\Sigma_\varrho=\{u\in X_\al,\; E(u)=d_\varrho,\;M(u)=\varrho\}.
	\]
	\begin{theorem}\label{criticalcase-normalized}
		Let $p_2=s_c$, $\moo=1$ and $Q$ be a ground state of \eqref{zeroalpha} with $p=s_c$.
		If $\mo<0$ and $\po>s_c$, then for any $\varrho\in(M(Q),+\infty)$, the set $\Sigma_\varrho$ is nonempty.
		If $\mo>0$ and $\po<s_c$, then for any $\varrho\in(0,M(Q))$, the set $\Sigma_\varrho$ is not empty.
	\end{theorem}
	\begin{proof}
		First, we consider the case $\mo<0$ and $p_1>s_c$. The minimization value $d_\varrho$ is bounded from below. Indeed, we have from \eqref{young-in} for any $u\in X_\al$ with $M(u)=\varrho$ that
		\begin{equation}\label{bound-below}
			\begin{split}
				E(u)&\geq\frac{1}{2}\|u\|_{\xx}^2
				+\left(\frac{\mo}{\po+1}-\frac{C_\tau}{\poo+1}\right)\int_\rt u^{\po+1}\dz-\tau\varrho\\
				& \geq\frac{1}{2}\|u\|_{\xx}^2
				+\left(\frac{\mo}{\po+1}-\frac{C_\tau}{\poo+1}\right)\varrho^{c_{p_1}}\rho_{p_1}\|u\|_{\xx}^{\frac{2(p_1-1)}{s_c-1}}-\tau\varrho
		\end{split}	\end{equation}
		If we choose $\tau>0$ such that $\frac{\mo}{\po+1}-\frac{C_\tau}{\poo+1}>0$, then $d_\varrho>-\infty$. We show that if $\varrho\in(0,M(Q))$, then $d_\varrho\geq0$ while $d_\varrho<0$ for all $\varrho\in(M(Q),+\infty)$. Let $u\in X_\al$ and $M(u)=\varrho\leq M(Q)$. Then, we have that
		\[
		\begin{split}
			E(u)
			&\geq\frac12\|u\|_{X_\al}^2-\frac{  \rho_{p_2}\varrho^{c_{p_2}}}{p_2+1}\|u\|_{\xx}^2
			=\left(\frac12-\frac{  \rho_{p_2}\varrho^{c_{p_2}}}{p_2+1}\right)\|u\|_{\xx}^2.
		\end{split}	\]
		We note that
		\[
		\rho_{p_2}^{-1}=\al^{-1}\left(\frac\al2\right)^{\frac{\al}{\al+2}}\frac{2\al}{\al+2}\|Q\|_\lt^{\frac{4\al}{2+\al}}.
		\]
		Then, $E(u)\geq0$ if
		\[
		\|u\|_\lt\leq\ell_\al \|Q\|_\lt,
		\]
		where
		\[
		\ell_\al = \al^{\frac14}(\al+2)^{-(s_c-1)}2^{\frac1{2\al}}.
		\]
		Thus, it follows that $d_\varrho\geq0$ for all $\varrho\in(0,\ell_\al^2 M(Q)]$. Next, for $\epsilon>0$ we set  
		\[
		u_\epsilon(x,y)=\frac{\sqrt{\varrho}}{\|Q\|_\lt}\epsilon^{\frac{\al+2}{2}}Q(\epsilon x,\epsilon^{\al+1}y).
		\] 
		Then $\|u_\epsilon\|_\lt^2=\varrho$. Moreover, it is straightforward to check that
		\begin{equation}\label{lowerb}
			\begin{split}
				E(u_\epsilon)&=\epsilon^{2\al}\left(
				\frac{\varrho}{2M(Q)}\|Q\|_\xx^2- \left(\frac{\varrho}{M(Q)}\right)^{\frac{3\al+2}{\al+2}}K_2(Q)\right)
				+
				\epsilon^{(\al+2)\frac{p-1}{2}}
				\left(\frac{\varrho}{M(Q)}\right)^{\frac{p-1}{2}(\al+2)}K_1(Q)\\
				&\leq 
				\frac{\epsilon^{2\al}}{2}\|Q\|_\xx^2\left(
				\frac{\varrho}{M(Q)}- \left(\frac{\varrho}{M(Q)}\right)^{\frac{3\al+2}{\al+2}}\right)
				+
				\epsilon^{(\al+2)\frac{p-1}{2}}
				\left(\frac{\varrho}{M(Q)}\right)^{\frac{p-1}{2}(\al+2)}K_1(Q)<0
			\end{split}
		\end{equation}
		provided $\varrho>M(Q)$ as $\epsilon\ll1$,
		where in the above we use the following Pohozaev identity
		\[
		\|Q\|_\xx^2=2K_2(Q).
		\]
		
		Next, we show that every minimizing sequence for $d_\varrho$ is bound in $X_\al$ and bound from below in $L^{\po+1}(\rt)$. Let $\{u_n\}$ be a minimizing sequence. Since $M(u_n)=\varrho$, then we obtain from \eqref{bound-below} that  $\{u_n\}$ is bounded in $X_\al$.   Furthermore, since $d_\varrho$, we have $E(u_n)\leq d_\varrho/2$ for $n$
		large enough. The definition of $E$ and \eqref{young-in} show  that
		\begin{equation}
			\|u_n\|_{L^{\po+1}(\rt)}^{\po+1}\geq-E(u_n)\geq -\frac{ d_\varrho}{2}.
		\end{equation}
		
		Now suppose that $\{u_n\}$ is a minimizing sequence for $d_\varrho$, i.e., $\|u_n\|_\lt^2=\varrho$ and $E(u_n)\to d_\varrho$ as $n\to\infty$. We first prove
		\begin{equation}
			d_{\theta\varrho}<\theta d_\varrho
		\end{equation}
		for all $\varrho>M(Q)$ and $\theta>1$.  Since $\varrho>M(Q)$, there exists $\delta>0$ such that $\|u_n\|_\xx^2\geq\delta$ for sufficiency large $n$. If this is not true, then $\|u_n\|_\xx\to0$ as $n\to\infty$. By \eqref{bin}, this implies that $E(u_n)$ tends to zero as $n\to\infty$ which contradicts $d_\varrho<0$. Therefore, the minimization problem \eqref{minim-normal} can be rewritten as
		\begin{equation}\label{min-upb}
			d_\varrho=\inf\{E(u);\; u\in X_\al,\; M(u)=\varrho,\;\|u\|_\xx^2\geq\delta\}.
		\end{equation}
		If we set $\tilde{u}(x,y)=u(\theta^{-\frac{1}{\al+2}}x,\theta^{-\frac{\al+1}{\al+2}}y)$ with $\theta>1$, then we can see that $M(\tilde{u})=\theta M(u)$ and
		\begin{equation}
			d_{\theta\varrho}\leq\inf\{E(\tilde{u});\; u\in X_\al,\;M(u)=\varrho,\;\|u\|_\xx^2\geq\delta\},
		\end{equation}
		where
		\[
		\begin{split}
			E(\tilde{u})&=
			\frac{\theta^{\frac{2-\al}{2+\al}}}{2}\|u\|_\xx^2-\theta K(u)\\
			&=\theta E(u)-\frac{1}{2}(\theta-\theta^{\frac{2-\al}{2+\al}})\|u\|_\xx^2\\
			&\leq\theta E(u)-\frac{\delta}{2}(\theta-\theta^{\frac{2-\al}{2+\al}}).
		\end{split}
		\]
		We have by taking the infimum that
		\[
		d_{\theta\varrho}\leq \theta d_\varrho-\frac{\delta}{2}(\theta-\theta^{\frac{2-\al}{2+\al}})<\theta d_\varrho
		\]
		for all $\varrho>M(Q)$ and $\theta>1$. This implies that
		\begin{equation}\label{subadd}
			d_\varrho<d_\beta +d_{\varrho-\beta}
		\end{equation}
		for all $\varrho>M(Q)$ and $\beta\in(0,\varrho)$. More precisely, if $\beta>M(Q)$ and $0<\varrho-\beta\leq M(Q)$, then we  obtain that
		\[
		d_\varrho=d_{\beta\varrho/\beta}<\frac\varrho\beta d_\beta=d_\beta+\frac{\varrho-\beta}{\beta}d_\beta\leq d_\beta+d_{\varrho-\beta}.
		\]
		Now we apply the concentration-compactness principle (see \cite{Li1}) in $X_\al$ for $\{u_n\}$, similar to \cite{dbs-0,dbs-1,AMC,lps-0}. The vanishing case cannot occur by \eqref{bound-below}. The dichotomy case is ruled out, similar to \cite{dbs-1} by using the subadditivity property \eqref{subadd} with the help of the fractional Leibniz rule together with the fractional commutator estimate in \cite[Lemma 2.5]{gh} and \cite[Lemma 2.12]{zeng} (see also \cite{li} for the estimates of higher order fractional derivatives), and the fractional Poincar\'{e} inequality (see \cite[Proposition 2.1]{dd}, \cite[Lemma 2.2]{hv} and \cite{bucur,lnp}). 
		Finally,   we can conclude by ruling out vanishing and dichotomy that indeed compactness occurs. Thereby, we infer 
		that there exist a subsequence $\{u_{n_k}\}$, $\{z_k\}\subset\rt$ and some $u\in X_\al$ such that
		\[
		u_{n_k}(\cdot-z_k)\to u
		\]
		in $L^q(\rt)$ for all $q<2^\ast$. Thus, we deduce from this fact combined with the weak lower semicontinuity of the $X^\al$-norm that
		\[
		E(u)\leq\lim_{k\to\infty}E(u_{n_k})=d_\varrho.
		\]
		We have from the definition $d_\varrho$ of that  $E(u) =d_\varrho$. Particularly, $E(u_{n_k})\to E(u)$, and it indicate that $\|u_{n_k}\|_\xx\to\|u\|_\xx$, which implies that $u_{n_k}(\cdot-z_k)\to u$
		strongly in $X_\al$.

		Next, we consider the case $\mo>0$ and $p_1<s_c$. We first note for $u\in X_\al$ with $u\not\equiv0$ and
		\begin{equation}\label{scal-1}
			u_\epsilon(x,y)=\epsilon^{\frac{\al+2}{2}}u(\epsilon x,\epsilon^{\al+1}y) 
		\end{equation}
		that
		\[
		d_\varrho\leq E(u_\epsilon)
		=\frac{\epsilon^{2\al}}{2}\|u\|_\xx^2-\epsilon^{\frac{(\al+2)(p_1-1)}{2}}K_1(u)-\epsilon^{\frac{(\al+2)(p_2-1)}{2}}K_2(u)<0
		\]
		for sufficiently small $\epsilon>0$. Moreover, any minimizing sequence is bounded in $X_\al$ by \eqref{bin}. Furthermore, it follows from  $\vr<M(Q)$ that there exists $\delta>0$ such that $K_1(u_n),K_2(u_n)\geq \delta$ for all sufficiently large $n$. We also observe for $\varrho>M(Q)$ and $V=\frac{\sqrt\varrho}{\|Q\|_\lt^2}$ that $M(U)=\varrho$ and
		\[
		\frac12\|U\|_\xx^2-K_2(Q)=\frac12\frac{\|Q\|_\xx ^2}{M(Q)}\varrho -\frac{\varrho^{\frac{\poo}{2}+1}}{(M(Q))^{\frac{\poo+2}{2}}}K_2(Q)<0.
		\]
		This together with $p_2<p_2$ with \eqref{scal-1} reveals that
		\[
		\begin{split}
			\lim_{\epsilon\to+\infty}	E(u_\epsilon)=-\infty.
		\end{split}
		\]
		Second,  let 
		\[
		u_\theta(x,y)=\theta^{\frac{2\al}{\po-1}}u(\theta x,\theta^{\al+1}y),
		\]
		so that
		\[
		\begin{split}
			{\frac{\dd}{\dd\theta}}\Big|_{\theta=1}E(u_\theta)&=
			\theta_1\left(\frac12\|u\|_\xx^2-K_1(u)\right)-\theta_2K_2(u)\\
			&=\theta_1 E(u)-\theta_3K_2(u),
		\end{split}
		\]
		where
		\[
		\theta_1=\frac{4\al}{\po-1}+\al-2,\quad\theta_2=\frac{4\al}{\poo-1}+\al-2,
		\quad
		\theta_3=\frac{2\al}{\po-1}(\poo-\po).
		\]
		Together with $\|u_\theta\|_\lt^2=\theta^{\frac{4\al}{\po-1}-\al-2}\|u\|_\lt^2$, which implies that 
		\[
		{\frac{\dd}{\dd\theta}}\Big|_{\theta=1}\|u_\theta\|_\lt>0,
		\]
		and this shows that $d_\varrho$ is decreasing in $\varrho$. As a consequence, we have a strict subadditivity condition. But we alternatively show this directly.
		Let $\{U_n\}$ and $\{V_n\}$ be sequences of functions in $X_\al$ such that $M(U_n)\to\vr'$, $E(U_n)\to d_{\vr'}$, $M(V_n)\to\vr''$  and $E(V_n)\to d_{\vr''}$ as $n\to\infty$. Let $\nu'=\frac{d_{\vr'}}{\vr'}$  and $\nu''=\frac{d_{\vr''}}{\vr''}$. If $\nu'<\nu''$, then by defining $\tilde U_n=\sqrt{\tilde\vr} U_n$ with $\tilde\vr=(\vr'+\vr'')/\vr'$ we obtain that $\tilde U_n\in X_\al$ and $E(\tilde U_n)\to\vr'+\vr''$ and consequently $d_{\vr'+\vr''}\leq\lim_{n\to\infty}E(\tilde U_n)$. A straightforward calculation for $\tilde\vr>1$ gives that $E(\tilde U_n)\leq\tilde\vr E(U_n)$. This then implies that $d_{\vr'+\vr''}\leq \tilde\vr'\nu'.$ Put $\delta=\vr''(\nu''-\nu')>0$. Then
		\[
		d_{\vr'+\vr''}\leq\vr'\nu'+\vr''\nu''-\delta.
		\]
		This means that $d_{\vr'+\vr''}<d_{\vr'}+d_{\vr''}$. This inequality holds similarly when $\nu'>\nu''$. If $\nu'=\nu''$, then there exists $\delta>0$ such that $E(\tilde U_n)\leq \tilde\vr E(U_n)-\delta$ for sufficiently large $n$. This in turn implies that
		\[
		E(\tilde U_n)\leq\tilde\vr E(U_n)-\delta
		\]
		for sufficiently large $n$. Thus, we obtain that $$d_{\vr'+\vr''}\leq\tilde\vr\lim_{n\to\infty}E(U_n)-\delta=\tilde\vr\nu'\vr'-\delta=\nu'\vr'+\nu'\vr''-\delta.$$
		Therefore,  similar to the case $\mu_1<0$, the boundedness of the minimizing sequence $\{u_n\}$ now implies that there exists $u\in X_\al$ such that $u_n\rightharpoonup u$ in,  $X_\al$, up to a subsequence. The limit function is nontrivial. If not, then $K(u_n)\to0$ which contradicts to $E(u_n)=\frac12\|u_n\|_\xx^2\to d_\varrho<0$.
		Next, we prove $M(u)=\varrho$. If $M(u)<\varrho$, we get
		\[
		\kappa_0=\frac{\sqrt\varrho}{\|u\|_\lt}>1,\quad\kappa_n=\frac{\sqrt\varrho}{\|u_n-u\|_\lt}>1.
		\]
		Note that 
		\[
		E(\theta u)
		=\theta^2 E(u)+\theta^2(1-\theta^{p_1})K_1(u)+\theta^2(1-\theta^\poo)K_2(u),
		\]
		which implies for any $\theta>0$ that
		\[
		E(u)
		=\theta^{-2} E(\theta u)+ (\theta^{p_1}-1)K_1(u)+ (\theta^\poo-1)K_2(u).
		\]
		Thus, we deduce from the Brezis-Lieb lemma that
		\[
		\begin{split}
			E(u_n)&=E(u)+E(u_n-u)+o_n(1),\\
			&=\kappa_0^{-2}E(\kappa_0 u)
			+ (\kappa_0^{p_1}-1)K_1(\kappa_0u)+ (\kappa_0^\poo-1)K_2(\kappa_0u)\\
			&\qquad+\kappa_0^{-2}E(\kappa_n (u_n-u))
			+ (\kappa_n^{p_1}-1)K_1(\kappa_0(u_n-u))+ (\kappa_n^\poo-1)K_2(\kappa_0(u_n-u))\\
			&\geq\frac{M(u)}{\varrho}d_\varrho
			+\frac{M(u_n-u)}{\varrho}d_\varrho+o_n(1).
		\end{split}
		\]
		Taking the limit $n\to\infty$, we get
		\[
		d_\varrho\geq\frac{M(u)}{\varrho}d_\varrho,
		\]
		which together with $d_\varrho<0$, implies that $M(u)\geq\varrho$. Hence, $M(u)=\varrho$ and then $u_n\to u$ in $\lt$. We can obtain from the \eqref{bin} that $u_n\to u$ in $L^\po(\rt)\cap L^\poo(\rt)$. Consequently, we obtain 
		\[
		E(u)\leq\liminf_{n\to\infty}E(u_n)=d_\varrho.
		\] 
		This means that $E(u)=d_\varrho$ and $u\in \Sigma_\varrho$.
	\end{proof} 

	\begin{proposition}  
		\label{lem31}
		Let $p_2=s_c$, $\moo=1$ and $Q$ be a ground state of \eqref{zeroalpha} with $p=s_c$, where $\vr_\ast=M(\ff_{s_c})$. If $p_1<s_c$ and $\phi_\vr$ is a minimizer of $d_\vr$, then
		\[
		\frac{d_\vr}{\vr}\approxeq-\varTheta^{-\delta},\qquad\frac{K_1(\phi_\vr)}{\vr}\approxeq-\varTheta^{-\delta}
		\]
		as $\vr\to\vr_\ast$, where $\varTheta=1-\frac{\al+2}{2}(2\al^{-1})^{\frac{\al}{\al+2}}\left(\frac{\varrho}{\varrho_\ast}\right)^\frac{2\al}{\al+2}$ and $\delta=\frac{(p_1-1)(\al+2)}{5\al+2-(\al+2)p_1}$.
	\end{proposition}
	\begin{proof}
		Let $\phi$ be in $\x$ such that $M(\phi)=\vr$. Then by \eqref{bound-below} and the H\"{o}lder inequality we have for any $\epsilon=\frac14\varTheta$ that
		\[
		\begin{split}
			2E(\phi)&\geq
			\|\phi\|_{\xx}^2(1-\varTheta-2\epsilon)-R_\epsilon\vr^{\frac{3\al+2+p_1(\al-2)}{2(5\al+2-p_1(\al+2))}}\\
			&\gtrsim\frac12\|\phi\|_{\xx}^2\varTheta- \vr^{\frac{3\al+2+p_1(\al-2)}{2(5\al+2-p_1(\al+2))}}\varTheta^{-\delta}.
		\end{split}	\]
		where $R_\epsilon=C\epsilon^{-\delta}$. By taking the infimum over all $\phi\in\x$ with $M(\phi)=\vr$, we obtain that
		\[
		\frac{d_\vr}{\vr}\gtrsim-\vr^{\frac{3\al+2+p_1(\al-2)}{2(5\al+2-p_1(\al+2))}}\varTheta^{-\delta}.
		\]
		Define for $\tau>0$ the trial functions $u_\tau(x,y)=\sqrt{\varrho}\tau^{\frac{\al+2}{2}}u(\tau x,\tau^{\al+1}y)$, where $u=\ff_{s_c}/\vr_\ast$. Then $M(u_\tau)=\varrho$. Hence, we have from Theorem \ref{sharp-constant} and Lemma \ref{poho-homo} that
		\begin{equation}\label{equation-bound}
			\begin{split}
				\frac{d_\varrho}{\vr}&\leq\frac{E(u_\tau)}{\vr}
				=
				\frac12\tau^{2\al}\left(\|u\|_\xx^2-\frac{2}{s_c+1}\vr^{\frac{2\al}{\al+2}}K_2(u)\right)-\tau^{\frac{p_1-1}{2}(\al+2)}\vr^{\frac{p_1-1}{2}}K_1(u)\\
				&=\frac{\tau^{2\al}(\al+2)}{4\al}\left(1-\left(\frac{\varrho}{\varrho_\ast}\right)^\frac{2\al}{\al+2}\right)
				-\tau^{\frac{p_1-1}{2}(\al+2)}\vr^{\frac{p_1-1}{2}}K_1(u).
			\end{split}
		\end{equation}
		To find the upper bound, we use \eqref{equation-bound} with $\tau=\epsilon\varTheta^{\frac{\delta}{2\al}}$ with $\epsilon>0$. Then by choosing $\epsilon$ small enough, we infer when $\vr\to\vr_\ast$ that $\frac{d_\vr}{\vr}\lesssim- \varTheta^{-\delta}.$ The second part is proved analogously.
	\end{proof}
	\section{Strong Instability of Ground States}\label{sect-strong-inst}
	In this section, we establish the instability of ground state solutions of equation \eqref{main-fifthkp} through the mechanism of blow-up. To achieve this, we introduce certain submanifolds within the context of $\mathbf{x}$, which remain unchanged under the evolution governed by the Cauchy problem linked to equation \eqref{main-fifthkp}. By appropriately selecting initial data from these sets, the solution of equation \eqref{main-fifthkp} will experience blow-up within a finite time. The instability of the ground states is then deduced from this phenomenon.

	Let $b,d\geq0$, and define 
	\[
	\rrr(u)=\al b\|\dx u\|_\lt^2 + (d-b)\|\nd u_y\|_\lt^2
	-(b+d)(\ps_1F_1(u)
	+\ps_2 F_2(u)).
	\]
	
	We denote 
	\[
	\Lambda_{b,d}=\{u\in\x\setminus\{0\},\;S(u)<m,\;\rrr(u)=0\},
	\]
	and set
	\[
	m_\rrr=\inf_{u\in\Lambda_{b,d}} S(u).
	\]
	\begin{lemma}\label{equiv-vari-1}
		Let $\mu_1>0$.	If $\moo<0$, $p_1\geq p_2$ and
		\[
		p_1\geq4\frac{\max\{d-b,\al b\}}{b+d}+1.
		\]
		Then 
		\begin{equation}\label{equiv-m}
			m=m_\rrr.
		\end{equation}
		Identity \eqref{equiv-m} holds when $\mu_2>0$ if   
		\[
		p_1\geq p_2\geq4\frac{\max\{d-b,\al b\}}{b+d}+1.
		\]
	\end{lemma} 
	\begin{proof}
		By Theorem \ref{theorem-exist} we show that $\ff$ is ground state of \eqref{gkp} if and only if $\ff\in\Lambda_{b,d}$ and $S(\ff)=m_{\rrr}$. First we note that if $\ff$ is a ground state of \eqref{gkp}, then $\ff\in\Lambda_{b,d}$ by 
		\begin{equation}\label{id}
			\rrr(\ff)=\left\la S'(\ff), \frac{b+d}{2}\ff+bx\ff_x+dy\ff_y\right \ra=0.
		\end{equation}
		
		Next, suppose that $\rrr(u)<0$. Then $\rrr(\lam u)>0$ for some sufficiently small $\lam>0$, and hereby there exists $\lam_0\in(0,1)$ such that $\rrr(\lam_0u)=0$. Hence, $m_\rrr\leq\tilde S(u)$, where  $\tilde S=S-\frac{2}{r(b+d)}\rrr$  with $r=p_1-1$ if $\mu_2<0$ and $r=p_2-1$ when $\mu_2>0$. This means that
		\[
		m_\rrr=\tilde m_\rrr:=\inf_{u\in\tilde\Lambda_{b,d}}\tilde S(u),
		\]
		where \[
		\tilde\Lambda_{b,d}=\{u\in\x\setminus\{0\},\;\rrr(u)\leq0\} .
		\]
		So it is enough to find the ground state $\ff$ such that $\tilde m_\rrr=\tilde S(\ff)$. Notice  the assumptions on $p_1$ and $p_2$ shows that $\tilde S(\ff)>0$. Hence, there exists a minimizing sequence $\{u_n\}\subset\tilde \Lambda_{b,d}$ of $\tilde m_\rrr$ that is bounded in $\x$, and in $L^q(\rt)$ for $q\in[2,2^\ast]$ from \eqref{embed},  and $\lim_{n\to\infty}\tilde S(u_n)=\tilde m_\rrr$. This implies that there exist  a subsequence of $\{u_n\}$, still denoted by the same  $\{u_n\}$, and $u\in\x$ such that $u_n$ converges to $u$ weakly in $\x$. It follows from Lemma 3.3 in \cite{dbs-1} that $u_n\to u$ a.e. in $\rt$. We show that $\tilde S(u)=\tilde m_\rrr$ and $u\in\Lambda_{b,d}$. If we assume that $i_{p_1}:=\inf_n\|u_n\|_{L^{p_1+1}(\rt)}>0$, then Lemma 4 in \cite{liuwang} shows that $u\not\equiv0$ a.e. in $\rt$. Now if  $i_{p_1}=0$, then exists a subsequence of $\{u_n\}$, still denoted by the same, such that $\|u_n\|_{L^{p_1+1}(\rt)}\to0$ as $n\to\infty$, so that $\|u_n\|_{L^{p_2+1}(\rt)}\to0$ from the boundedness of $\{u_n\}$ in $\x$. Since ${u_n}\subset\tilde{\Lambda}_\rrr$, $\|u_n\|_{\dot{X}_\al}\to0$ as $n\to\infty$. Moreover, $\|\dx u\|_\lt^{\frac{4b}{b+d}} \|\nd u_y\|_\lt^{\frac{2(d-b)}{b+d}}\to0$. Consequently, we have from \eqref{bin} that
		\[
		\|\dx u\|_\lt^{\frac{4b}{b+d}} \|\nd u_y\|_\lt^{\frac{2(d-b)}{b+d}}\lesssim \|u_n\|_{L^{p+1}(\rt)}^{p+1}
		\lesssim \|\dx u\|_\lt^{\frac{p-1}{2}} \|\nd u_y\|_\lt^{\frac{p-1}{\al}},
		\]
		for all $p\leq p_1$, and equivalently, 
		$$
		\|\dx u\|_\lt^{\frac{4b}{b+d}} \|\nd u_y\|_\lt^{\frac{2(d-b)}{b+d}}\left(1-C\|\dx u\|_\lt^{\frac{p-1}{2}-\frac{4b}{b+d}} \|\nd u_y\|_\lt^{\frac{p-1}{\al}-\frac{2(d-b)}{b+d}}\right)\leq0.
		$$
		This turns into
		$$
		\|\dx u\|_\lt^{\frac{p-1}{2}-\frac{4b}{b+d}} \|\nd u_y\|_\lt^{\frac{p-1}{\al}-\frac{2(d-b)}{b+d}}\gtrsim 1.
		$$
		This contradicts $\|u_n\|_{\dot{X}_\al}\to0$ by the assumptions on $p_1$. Consequently, $i_{p_1}>0$. Now if $\rrr(u)>0$, then   the Brezis-Lieb lemma  and the fact $\{u_n\}\subset\tilde{\Lambda}_{b,d}$ reveals that $\rrr(u_n-u)\leq0$ as $n\to\infty$, so that $\tilde{S}(u_n-u)\geq \tilde m_\rrr$. Since $\tilde{S}(u_n-u)\to \tilde m_\rrr$ as $n\to\infty$, we have again from the Brezis-Lieb lemma that $\tilde{S}(u)\leq0$ which contradicts $u\not\equiv0$ a.e. Consequently, $u\in\tilde{\Lambda}_{b,d}$. This shows that $\tilde S(u)=\tilde m_\rrr$. It is easy to check that $\rrr(u)=0$. Indeed, since $\rrr(\lam u)>0$ for sufficiently small $\lam>0$, so   $\rrr(u)<0$ implies that $\lam_0u\in\Lambda_{b,d}$, which is a contradiction to the definition of $\tilde m_\rrr$. Finally, we prove that $u$ is a ground state. We have from $m_\rrr=\tilde m_\rrr= S(u)$ that there is $\theta\in\rr$ such that $S'(u)+\theta R'_{b,d}(u)=0$. The fact $u\in\Lambda_{b,d}$ reveals from 
		\[
		\left\la S'(u)+\theta R'_{b,d}(u),\frac{b+d}{2}\ff+bx\ff_x+dy\ff_y\right\ra=0
		\]
		that $\theta=0$. Now if  $S'(w)=0$, then  $w\in \Lambda_{b,d}$ from \eqref{id}. Therefore, it follows from the definition of $m_\rrr$ that $S(u)\leq S(w)$.
	\end{proof}
	
	Consider the submanifolds
	\[
	\Lambda_{b,d}^+=\{u\in\x,\;S(u)<m,\;\rrr(u)\geq0\},
	\]and
	\[
	\Lambda_{b,d}^-=\{u\in\x,\;S(u)<m,\;\rrr(u)<0\}.
	\]
	
	\begin{lemma}\label{invar-set}
		Let $d\geq b$ and $p_j$ satisfy Lemma \ref{equiv-vari-1}. Then the sets $\Lambda_{b,d}^+$ and  $\Lambda_{b,d}^-$ are invariant under the flow generated by the Cauchy problem associated with \eqref{main-fifthkp}.
		
	\end{lemma}
	
	\begin{proof}
		We only prove that $\rrr^+$ is invariant under the flow generated by the Cauchy problem associated with  \eqref{main-fifthkp} since the proofs of the other set are similar. Let $u(t)$ be the solution of \eqref{main-fifthkp} with initial data $u_0\in \rrr^+$. We first note from Theorem \ref{cauchy-problem} that $u(t)<m$. Next, we show that $\rrr(u(t))\geq0$ for $t\in[0,T)$. If it is not true, the continuity of $\rrr$ implies that there exists $t_1\in(0,T)$ such that $\rrr(u(t_1))=0$. This means that $u(t_1)\in\Lambda$. So that $S(u(t_1))\geq m_{\rrr}\geq m$, which contradicts with $S(u(t))<m$ for all $t\in(0,T)$. Therefore, $\rrr(u(t))>0$ for $t\in[0,T)$.
	\end{proof}
	
	The following theorem gives another condition under which the uniform boundedness of solutions in the energy space is guaranteed.
	\begin{theorem}
		Let $\mu_1>0$ and $u_0\in \Lambda_{b,d}^+$.	Suppose that  $p_1\geq p_2$ and
		\[
		p_1>4\frac{\max\{d-b,\al b\}}{b+d}+1,
		\]
		if $\moo<0$ while
		\[
		p_1\geq p_2>4\frac{\max\{d-b,\al b\}}{b+d}+1
		\]
		when $\mu_2>0$.
		Then the solution $u(t)$ in Theorem \ref{cauchy-problem} is uniformly bounded in the energy space.
	\end{theorem}
	
	\begin{proof}
		We first note that $\Lambda_{b,d}^+\neq\emptyset$. Let $u_0\in \Lambda_{b,d}^+$ and $u(t)$ be the corresponding solution of \eqref{main-fifthkp} for $t\in[0,T)$ with the initial data $u_0$. Suppose by contradiction that $T<+\infty$. Then by Theorem \ref{cauchy-problem},
		\begin{equation}\label{global-eq-1}
			\lim_{t\to T^-}\|u\|_\x^2=+\infty.
		\end{equation}
		Hence, by the assumptions on $p_j$ and using the conservation laws $E(u(t))=E(u_0)$ and $M(u(t))=M(u_0)$ for $0\leq t<T$  that
		\[
		\begin{split}
			S(u_0)-\frac{2}{r(b+d)}\rrr(u(t))&=
			S(u(t))-\frac{2}{r(b+d)}Q(u(t))\\
			&\lesssim\tilde S(u(t))
			\approxeq\|u\|_{\x}^2,
		\end{split}\]
		where $\tilde S$ and $r$ are as in the proof of Lemma \ref{equiv-vari-1}.
		Thus,  we deduce from \eqref{global-eq-1} that
		\[
		\lim_{t\to T^-}\rrr(u(t))=-\infty.
		\]
		Now, we infer from the continuity that there is $t_0\in(0,T)$ such that $\rrr(u(t_0))=0$. Lemma \ref{equiv-vari-1} then implies $S(u(t_0))\geq m$, which contradicts the fact $S(u(t))=S(u_0)<m$.
	\end{proof}

	To apply the concavity method, we need to show that $R_{b,d}(u(t))$ is negative which is fundamental in our instability analysis.
	\begin{theorem}\label{thm-improved-blow-up}
		Let $\ff\in N_0$, $b\geq0$, $d\geq b+1$  and $u_0\in \Lambda_{b,d}^-\cap \Lambda_{b,d-1}^+$.	
		Suppose that  $p_1\geq p_2$ and
		\[
		p_1>4\max\left\{\frac{d-b}{b+d},\frac{\al b}{d+b-1}\right\}+1,
		\]
		if $\moo<0$ while
		\[
		p_1\geq p_2>4\max\left\{\frac{d-b}{b+d},\frac{\al b}{d+b-1}\right\}+1,
		\]
		when $\mu_2>0$.
		Then the solution $u(t)$ of \eqref{main-fifthkp},
		corresponding to the initial data $u(0) = u_0$,  satisfies  
		\begin{equation}\label{rho}
			\rrr(u(t))< 	\frac{(1-r)(b+d)}{2}(S(\ff)-S(u_0)) 
		\end{equation}
		for $0\leq t<T$, where $r$ is as the same in the proof of Lemma \ref{equiv-vari-1}.
		Moreover, $u(t)$ blows up in finite time. More precisely, there
		exists $0 < \tau <\infty$ such that
		\begin{equation}\label{eq-blup}
			\lim_{t\to\tau^-}\|u_y(t)\|_{L^2(\rr^2)}=+\infty.
		\end{equation}
		
	\end{theorem}
	
	\begin{proof}
		Note from  the virial identity \eqref{virial-id} that
		\begin{equation}\label{virial-m}
			\frac{1}{8}\frac{\dd^2}{\dd t^2}\nii (u)
			=\rrr(u)-R_{b,d-1}(u).
		\end{equation} 
		Then the proof of the theorem is similar to the one of Theorem \ref{blow-up-theorem-1} by using the above identity and combining Lemmas \ref{equiv-vari-1} and \ref{invar-set}.
	\end{proof}
	
	The following equivalence is useful to work with the critical points of $m$.
	\begin{lemma}\label{euqiv-2}
		Let $s\geq\al+1$. Then $\tilde m=m$, where $\tilde m=\inf_{u\in\tilde N_0}S(u)$ and 
		\[
		\tilde N_0=\{u\in X^s\setminus\{0\},\;P(u)=0 \}.
		\]
	\end{lemma}
	\begin{proof}
		Clearly $\tilde m\geq m$. To see the converse, it suffices to show that for any $\epsilon>0$ and $u^\ast\in N_0$, there holds
		\[
		S(u^\ast)\geq \inf_{u\in\tilde N_0}S(u)-\epsilon.
		\]
		Since $X^s$ is dense in $\x$, we find a sequence $\{u_n\}\subset X^s$ such that $u^\ast=u_n+w_n$ such that $w_n\to0$ in $\x$ as $n\to\infty$. Then we obtain from $u^\ast\not\equiv0$  and $u^\ast\in N_0$  that $\lim_{n\to\infty}P(u_n)=0$ and 
		\[
		\liminf_{n\to\infty}\|u\|_{L^q(\rt)}^q\neq0
		\]
		for any $q\leq 2^\ast$. It is easy to check that there exists a sequence $\{\nu_n\}\subset\rr$ such that $\nu_n u_n\in\tilde N_0$ and $\nu_n\to1$ as $n\to\infty$. Denoting $u^\ast=\nu_nu_n+(1-\nu_n)u_n+w_n$, we have from $1-\nu_n\to0$ and $w_n\to0$ in $\x$ as $n\to\infty$ that
		\[
		\frac12\|u^\ast\|_\x^2\geq\frac12\|\nu_nu_n\|\x^2-\frac{\epsilon}{3}
		\]
		and
		\[
		K(u^\ast)\leq K(\nu_nu_n)+\frac{\epsilon}{3}.
		\]
		On the other hand, the fact $\|w_n\|_{\dot X_\al}\to0$ as $n\to\infty$ implies that $S(u^\ast)\geq S(\nu_nu_n)-\epsilon\geq\tilde m-\epsilon$.
	\end{proof}
	\begin{lemma}\label{initial}
		Let $u,w\in\x$ be fixed and $\|w\|_\x\leq C$. Then there exists positive numbers $C_j$ with $j=1,2,3$, independent of $w$, such that
		\[
		\begin{split}
			&	S(u+w)<S(u)+C_1\|w\|_\x,\\
			&	\rrr(u)-C_3\|w\|_\x<\rrr(u+w)<\rrr(u)+C_2\|w\|_\x.
		\end{split}
		\]
	\end{lemma}
	\begin{proof}
		The proof is similar to Lemma 4.7 in \cite{nonlinearity}, so we omit the details.
	\end{proof}
	
	Now we are in the position to show the strong instability of solitary waves.
	
	\begin{theorem}\label{instability-theorem}
		Let $\mu_1>0$ and $\ff$ be a solitary wave solution.  Suppose that  $p_1\geq p_2$ and $p_1>s_c$	if $\moo<0$ and
		$ 	p_1\geq p_2>s_c	$ 	when $\mu_2>0$.
		Then	  for any $\delta>0$, there exists $u_0\in X_s$ ($s\geq\al+1$) with $\|u_0-\ff\|_{{X}}<\delta$, such that the solution $u(t)$ of \eqref{main-fifthkp} with initial data $u(0)=u_0$ satisfies $$\displaystyle\lim_{t\to T^-}\|u(t)\|_\x=+\infty$$ for some $T>0$.
	\end{theorem}
	
	\begin{proof}
		The proof is based on the application of Theorem \ref{thm-improved-blow-up}. For $B,D>0$ we set $w(x,y)=\sqrt{BD}\ff(Bx,Dy)$, and
		\[
		r_j=4+\frac{p_j+3}{2}(\al-2),\quad j=1,2.
		\]
		Then, we have from the Pohojaev identities
		\begin{equation}\label{poho-b}
			\begin{split}
				\al\|\dx\ff\|_\lt^2&=2\|\nd\ff_y\|_\lt^2,\\
				\|\nd\ff_y\|_\lt^2&=\ps_1K_1(\ff)+\ps_2K_2(\ff) ,\\
				\ps_1\|\ff\|_\lt^2&=\frac{r_1}{2}\|\dx\ff\|_\lt^2-(\poo-\po)K_2(\ff),\\
				\ps_2\|\ff\|_\lt^2&=\frac{r_2}{2}\|\dx\ff\|_\lt^2-(\po-\poo)K_1(\ff)
			\end{split}
		\end{equation}
		that
		\[
		S(\ff)= \left(\frac{r_j}{2\al\ps_j}+\frac12+\frac1\al-\frac{1}{\ps_j}\right)   \|\nd\ff_y\|_\lt^2.
		\]
		Moreover, we consider $d=(1+\al)b$ to observe that
		\[
		\begin{split}
			\frac{1}{b}\rrr(w)&=\al\left(B^{2\al}+\frac{\al D^2}{2B^2}-\frac{\ps_j(\al+2)(BD)^{\ps_j}}{2}\right)\|\dx\ff\|_\lt^2\\
			&\quad-(\al+2)(-1)^j\left(\ps_2(BD)^{\ps_2}-\ps_1(BD)^{\ps_1}\right)K_j(\ff)
		\end{split}
		\]
		and
		\[
		\begin{split}
			S(w)&=\frac{\al}{2}\left(\frac{r_j}{2\al\ps_j}+\frac{B^{2\al}}{\al}+\frac{D^2}{2B^2}-\frac{(BD)^{\ps_j}}{\ps_j}\right)\|\dx\ff\|_\lt^2\\
			&\quad-(\al+2)\left(\frac{(-1)^j(\poo-\po)}{2\ps_j}+(BD)^{\ps_2}-\frac{\ps_2}{\ps_1}(BD)^{\ps_1}\right)K_j(\ff) 
		\end{split}
		\]
		for $j=1,2$. Now, if $\mu_2<0$, then by using $j=1$ in the above equations and choosing $BD>1$ and near to 1, we find after some calculations for $b>(\al+2)^{-1}$ that $\rrr(w)<0$ and $S(w)<m$
		provided
		\[
		p_1>1+\frac{4\al b}{(\al+2)b-1}
		\]
		and $p_2\leq p_1$. In the case $\mu_2>0$, we get the similar inequalities if 
		\[
		p_1\geq p_2>1+\frac{4\al b}{(\al+2)b-1}.
		\]
		We note from \eqref{virial-m} and $\rrr(w)<0$ that  $R_{b,d-1}(w)>0$ if we choose $b>0$ sufficiently large. Now, by choosing for $\delta>0$ the above function $w$ such that $\|w-\ff\|_\x<\delta/2$, we can find from the density $X^s\hookrightarrow\x$ the initial data $u_0\in X^s$ such that 
		\[
		\|u_0-w\|_\x<\min\{\delta/2,\frac{\delta_1}{C_1+C_2+C_3}\},
		\]
		where $C_j$ with $j=1,2,3$ are as in Lemma \ref{initial} and $\delta_1=\min\{m-S(w),-\rrr(w),R_{b,d-1}(w),\delta/2\}$. Therefore, 
		\[
		\|\ff-u_0\|_\x\leq\|u_0-w\|_\x+\|w-\ff\|_\x<\delta,
		\]
		so that $u_0\in\Lambda_{b,d}^-\cap\Lambda_{b,d-1}^+$.
		Theorem \ref{thm-improved-blow-up} shows that the solution $u(t)$ associated with the initial data $u_0$ blows up in a finite time.
	\end{proof}

	\section{Numerical Results}\label{sect-numerical-s}
	In this section, we introduce a numerical approach that combines the Fourier pseudo-spectral method with the integration factor method and the Runge-Kutta method to solve the generalized KP equation, specifically with $\alpha=1$. This technique has proven successful in solving the Korteweg-de Vries (KdV) equation as demonstrated in \cite{trefethen}, and the Kadomtsev-Petviashvili (KP) equation as shown in \cite{klein-roidot, ksm}. An advantage of employing the integration factor method is its capability to alleviate the impact of the strong stiff linear term, thus enabling the use of larger time steps.
	
	It is worth noting that for the case $\alpha=1$, the equation \eqref{main-fifthkp} can be represented in the following form:
	\begin{equation}\label{kp-evolution}
		u_t + u_{xxx} + (\mu_1 u^{p_1} + \mu_2 u^{p_2})_x + \varepsilon ~ \partial_x^{-1} u{yy} = 0,
	\end{equation}
	where the antiderivative $\partial_x^{-1}$ is uniquely defined by
	\begin{equation}
		\partial_x^{-1} u(x) = \frac{1}{2} \left( \int_{-\infty}^x u(s,y) \dd s - \int_{x}^{\infty} u(s,y)\dd s \right).
	\end{equation}
	We choose the initial data in the Schwartz class $\mathcal{S}(\mathbb{R}^2)$  of rapidly decreasing smooth functions and we assume periodic boundary conditions. If the  Fourier transform is applied to the equation \eqref{kp-evolution},
	we obtain 
	\begin{equation} \label{Fouriertransformedevolution}
		\hat{u}_t-\ii  \left( k_x^3 -\varepsilon \frac{ k_y^2}{k_x} \right)\hat{u} +\ii  k_x (\mu_1  \widehat{u^{p_1}}+\mu_2  \widehat{u^{p_2}})=0,
	\end{equation}
	where  $k_x$ and $k_y$ are dual variables to $x$ and $y$ and $\partial_x^{-1}$ is defined  via its Fourier multiplier $ -i/k_x$.
	Here, the Fourier transform of $u(t,x,y)$ is denoted by $\hat{u}$  instead of $\hat{u}(t,k_x,k_y)$ for simplicity.
	Now, we multiply the equation \eqref{Fouriertransformedevolution} by the integration factor $\ee^{-\ii  \left( k_x^3 -\varepsilon \frac{ k_y^2}{k_x} \right)t}$, one gets 
	\begin{equation} \label{ode1}
		\hat{U}_t + \ii  k_x \ee^{-\ii   \left( k_x^3 -\varepsilon \frac{ k_y^2}{k_x} \right)t}(\mu_1  \widehat{u^{p_1}}+\mu_2  \widehat{u^{p_2}})=0,
	\end{equation}
	where 
	\begin{equation}  \label{ode2}
		\hat{U}= e^{-\ii   \left( k_x^3 -\varepsilon \frac{ k_y^2}{k_x} \right)t} \hat{u}.
	\end{equation}
	In order to avoid the division by zero for $k_x=0$ in the equations \eqref{ode1} and  \eqref{ode2}, we add 
	a small imaginary part to $k_x$ considering the sign of $\varepsilon$.  We use the smallest floating point number such that MATLAB 
	represent as \mbox{$\lambda=2.2 \times 10^{-16}$}. Therefore, we replace $ \displaystyle\frac{1}{k_x}$ by $\displaystyle\frac{1}{k_x+ \ii \varepsilon \lambda}$.
	We use the fourth-order explicit Runge–Kutta method to solve the resulting ODE \eqref{ode1} in time. Finally, we find the approximate solution by using the inverse Fourier transform. 
	
	\par Application of the numerical method requires truncation of the  $xy$-plane  to a
	finite rectangular region  $[-L_x, L_x]\times [-L_y, L_y]$. We approximate the Fourier coefficients by
	discrete Fourier transform which is efficiently computed
	with a fast Fourier transform (FFT).   In order to evaluate the discrete Fourier transform and its inverse, we use the MATLAB functions “fft2” and “ifft2,” respectively.  We use $N_x$ and $N_y$  collocation points in $x$ and $y$, respectively. We assume $N_x$ and $N_y$ are even, positive integers. 
	The time interval $[0, T]$ is divided into $M$ equal
	subintervals with grid spacing $\Delta t=T/M$. The temporal grid points are given by
	$t_{m}=m \Delta t$,  $m=0,1,2,...,M$. The approximate values of $u(x_{j}, y_{k}, t_{m})$,
	is denoted by $U_{j,k}^{m}$.  In order to check the accuracy of our code, we force the  mass conservation error 
	\begin{equation}
		\Delta= \left| 1 -\frac{M(u(t))}{M(u(0))} \right|
	\end{equation}
	is less than $10^{-4}$ at each time step, where  the  the mass integral \eqref{mass}
	is approximated by the trapezoidal rule. 
	\subsection{Numerical experiments}
	In this section, we first test the numerical accuracy of the method through the following example. The scheme is used to solve the Gardner equation given by
	\begin{equation}\label{gardner2}
		u_t+u_{xxx}+u u_x+ \frac{1}{12} u^2 u_x=0.
	\end{equation}
	The exact solution initially centered $x_0=0$  of the equation \eqref{gardner2} is  given by 
	\begin{equation}
		u(x,t)=\frac{6A}{1+\sqrt{1+\frac{A}{2}} \cosh \left(\sqrt{A} (x-At)\right)},
	\end{equation}
	where $A$ is arbitrary in \cite{hamdi}. Here, we set $A=1$. The computations are performed on the rectangular 
	region $[-60, 60] \times [-30, 30]$ for times up to $T = 10$. We
	choose the number of spatial grid points $N_x = 2^{10}, N_y=2^9$, and the number of temporal points $M=10000$. 
	Figure \ref{figure_gardner} shows the variation of the 
	$L_{\infty}$ error between the numerical and exact solution  (left panel) and the difference of mass $\Delta$ for the numerical solution (right panel).
	It shows that our proposed method is capable of high accuracy.
	\begin{figure}[!h]
		\begin{minipage}[t]{0.45\linewidth}
			\centering
			\includegraphics[height=5.5cm,width=7.5cm]{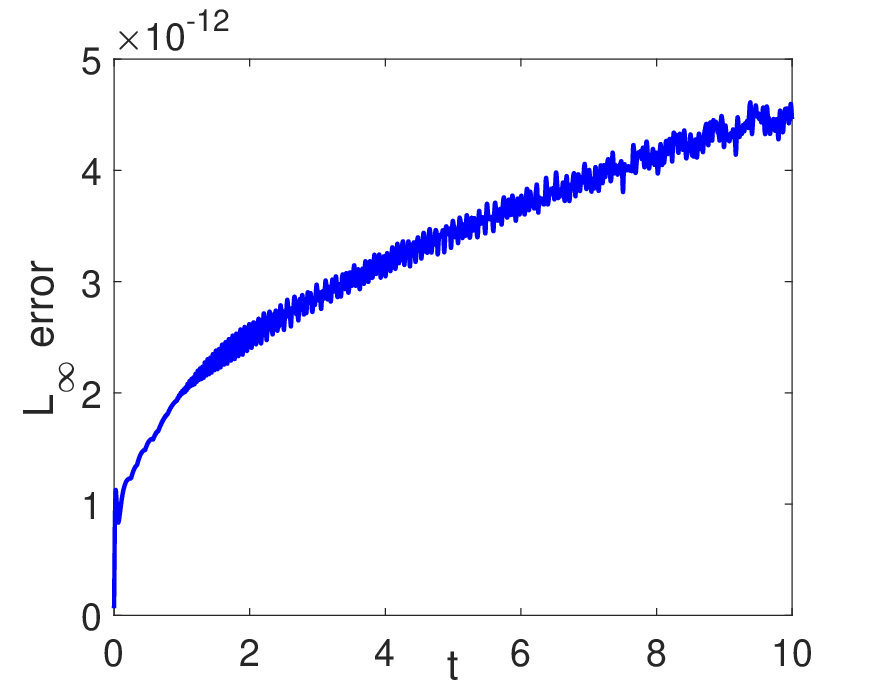}
		\end{minipage}%
		\hspace{20pt}
		\begin{minipage}[t]{0.45\linewidth}
			\centering
			\includegraphics[height=5.5cm,width=7.5cm]{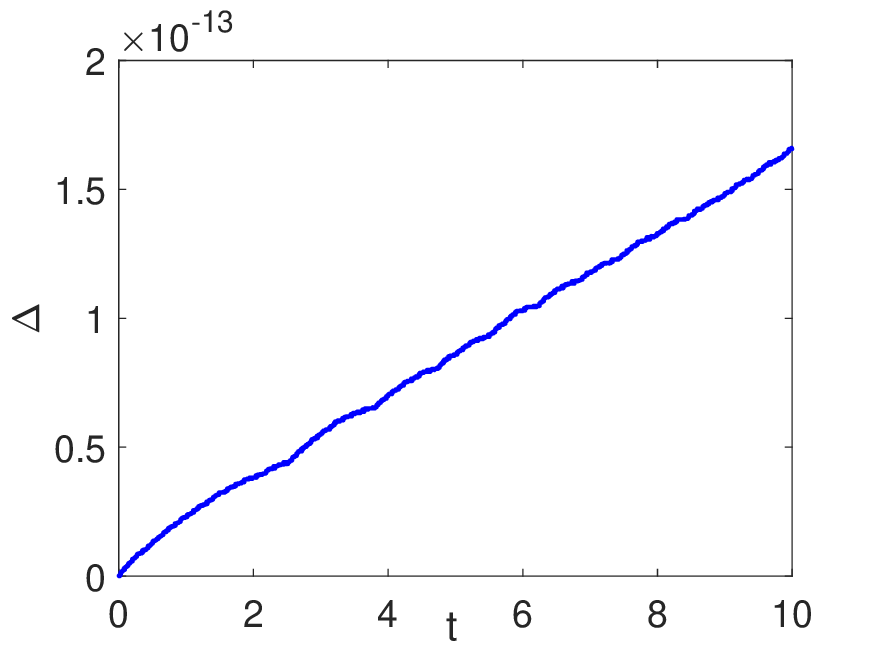}
		\end{minipage}
		\caption{Long-time errors of the numerical method (left panel) and mass conservation error (right panel) }\label{figure_gardner}
	\end{figure}

	\par
	The conditions needed for uniform boundedness or blow-up of the solutions for \eqref{main-fifthkp} have been discussed theoretically in Section $4$ but there remains a gap. Our aim is to remove the gap and to study the qualitative behavior of the solutions for the generalized KP equation with supercritical, subcritical, and critical nonlinearities.

	\subsection{The generalized KP equation with double supercritical nonlinearities }
	
	In this subsection,   we test our scheme for the generalized KP equation with double supercritical nonlinearities given by \mbox{$f(u)=u^5+u^6$} with $\mu_1=1, \mu_2=1$. The problem is solved on the rectangular region $[-5\pi, 5\pi]\times [-2\pi, 2\pi]$
	using ${N_x}=2^{12}$ and ${N_y}=2^{14}$.  We consider the Gaussian initial condition
	\begin{equation}
		u_0(x,y)=-4\left(\ee^ {-({x^2}+{y^2})}\right)_{xx} \label{initialu5u6}
	\end{equation}
	satisfying $E(u_0)=-88048<0$.  The program is stopped at $t=1.7188\times 10^{-6}$ since the quantity, measuring numerical mass conservation, becomes larger than $10^{-4}$.   The accuracy of the approximation in space is also controlled by the Fourier coefficients.
	Figure \ref{u5u6} shows the profile of  $ u $,   $ u_x $ and $ u_y $   and  Fourier coefficients  at the critical time \mbox{$t=1.7188\times 10^{-6}$} for the nonlinearity $f(u)=u^5+u^6$. As is seen from the figure, the gradient of $u$  at the critical time diverges more rapidly than the solution itself.
	The numerical result indicates that the solution blows up in a finite time. This numerical result is in complete agreement with the analytical result given in Theorem \ref{blow-up-theorem-1}, part (i).
	\begin{figure}[!htbp]
		\begin{minipage}[t]{0.45\linewidth}
			\centering
			\includegraphics[height=5.5cm,width=7.5cm]{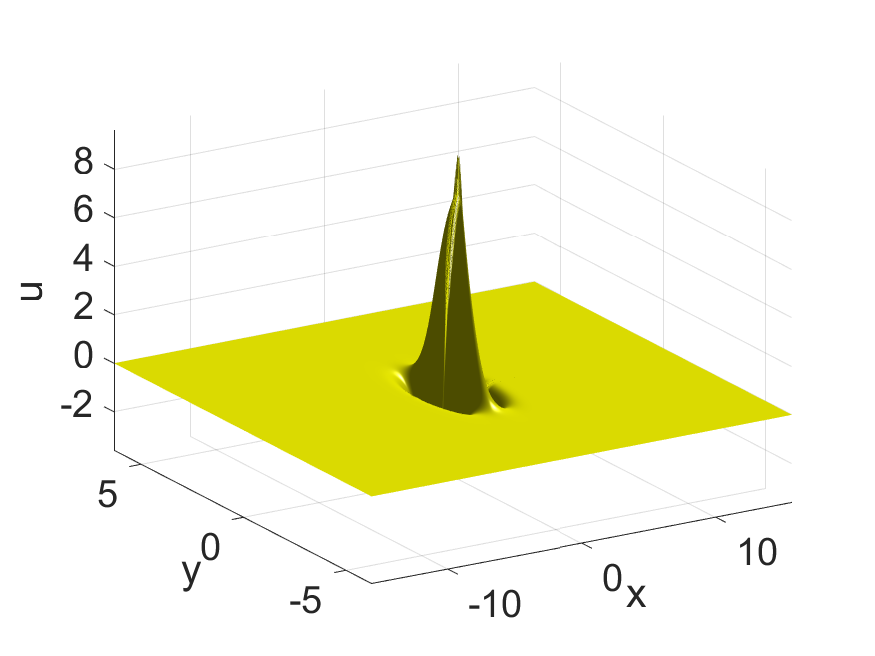}
		\end{minipage}%
		\hspace{20pt}
		\begin{minipage}[t]{0.45\linewidth}
			\centering
			\includegraphics[height=5.5cm,width=7.5cm]{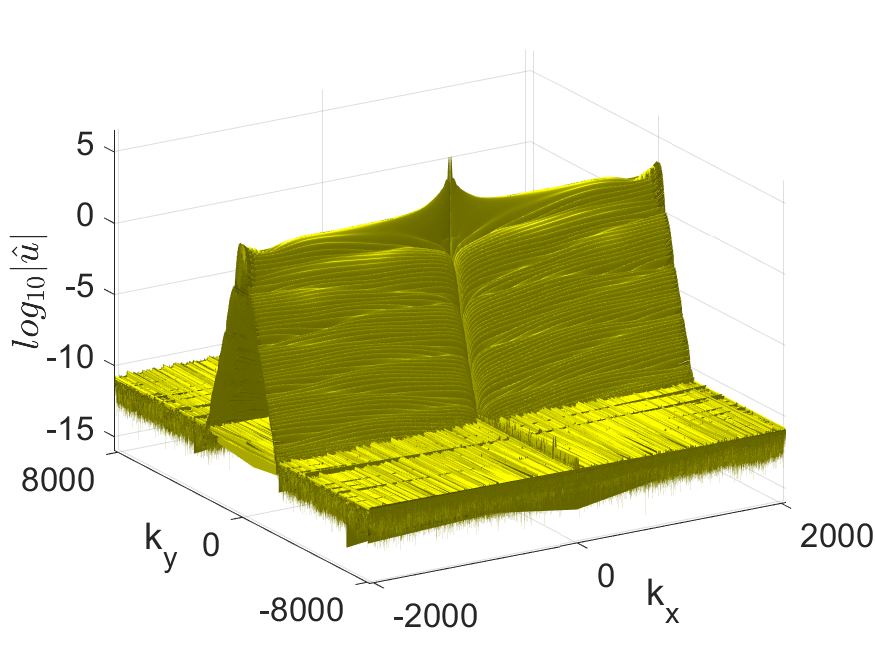}
		\end{minipage}
		\begin{minipage}[t]{0.45\linewidth}
			\centering
			\includegraphics[height=5.5cm,width=7.5cm]{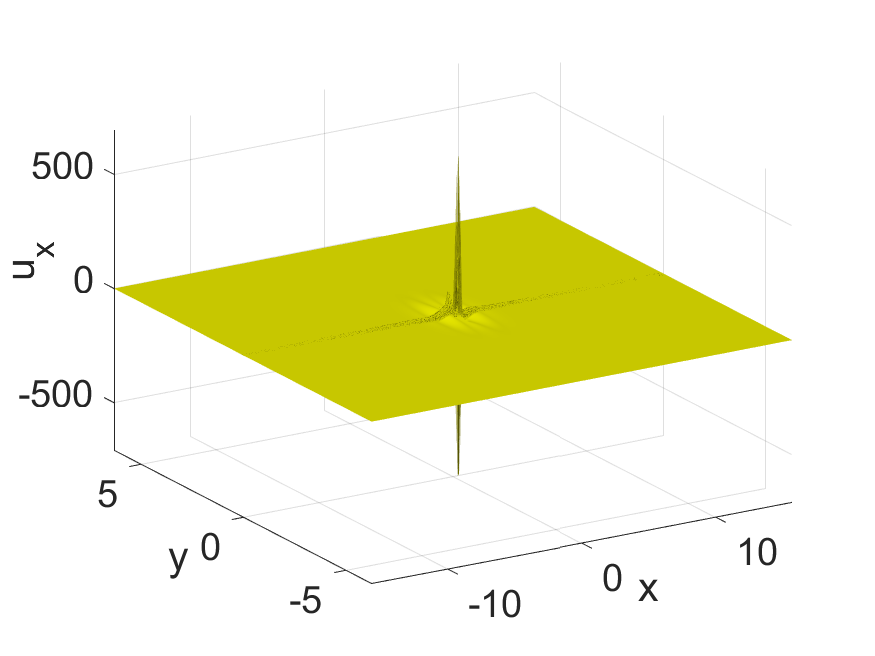}
		\end{minipage}%
		\hspace{40pt}
		\begin{minipage}[t]{0.45\linewidth}
			\centering
			\includegraphics[height=5.5cm,width=7.5cm]{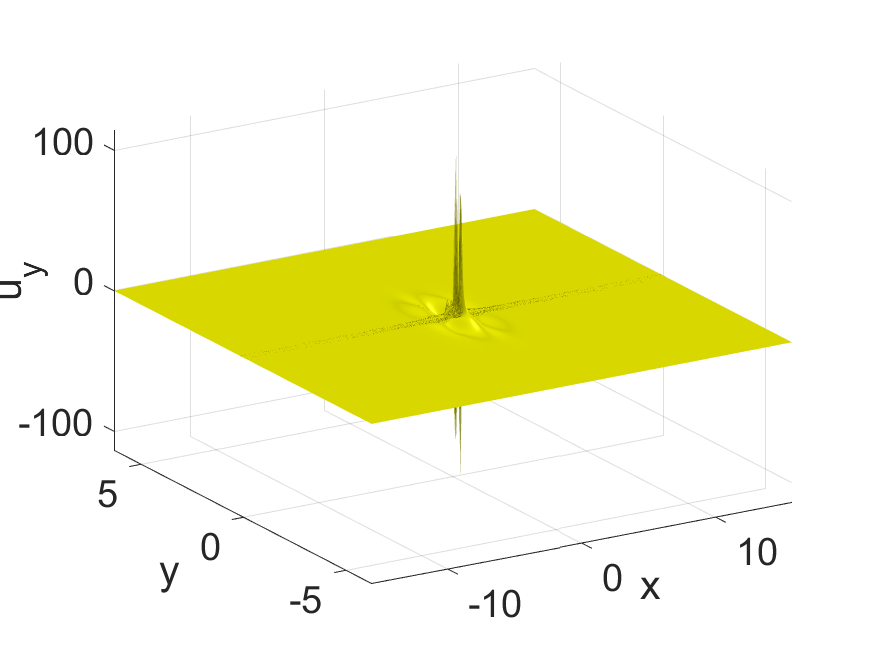}
		\end{minipage}
		\caption{ The profile of  $ u $,   $ u_x $ and $ u_y $  (top left, bottom left, bottom right)  and  Fourier coefficients at $t=1.7188\times 10^{-6}$ (top right) for the nonlinearity $f(u)=u^5+u^6$.}\label{u5u6}
	\end{figure}
	
	\subsection{The   double subcritical nonlinearities }
	For the generalized KP equation  with  double subcritical nonlinearities $f(u)=u^2+u^{3/2}$, the computations are carried out on the larger  rectangular region 
	$[-25\pi, 25\pi]\times [-4\pi, 4\pi]$ using ${N_x}=2^{14}$ and ${N_y}=2^{12}$ by taking the initial condition 
	\begin{equation}
		u_0(x,y)=  \left(\ee^ {-({x^2}+{y^2})}\right)_{xx} . \label{initial-u2u32}
	\end{equation}
	The variation of  $||u||_\infty$ with time and the profile of the numerical solution at  $t=0.5$  are illustrated in Figure \ref{u2u32}. As is seen from the 
	Figure,    $||u||_\infty$ decreases as time increases.   
	Figure \ref{u2u32focus} gives a closer look at the profile. We observe the decreasing oscillations and tails at $t=0.5$. The profile of the KP I equation with
	single quadratic nonlinearity at $t=0.15$ is presented in Figure $30$ of \cite{klein-saut}. The behavior of the solution for the single quadratic
	nonlinearity  $f(u)=u^2/2$  is very similar to the behavior of the generalized KP equation with the nonlinearity $f(u)=u^2+u^{3/2}$.
	There is no indication of a blow-up.  The numerical result is compatible with the Theorem \ref{subcritical-th}.
	
	\begin{figure}
		\begin{minipage}[t]{0.45\linewidth}
			\centering
			\includegraphics[height=5.5cm,width=7.5cm]{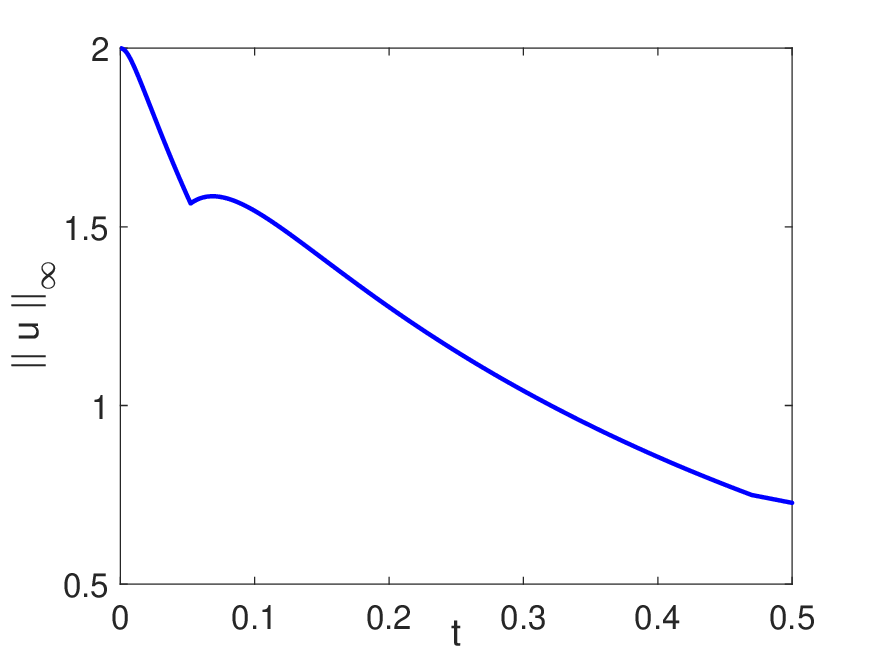}
		\end{minipage}%
		\hspace{40pt}
		\begin{minipage}[t]{0.45\linewidth}
			\centering
			\includegraphics[height=5.5cm,width=7.5cm]{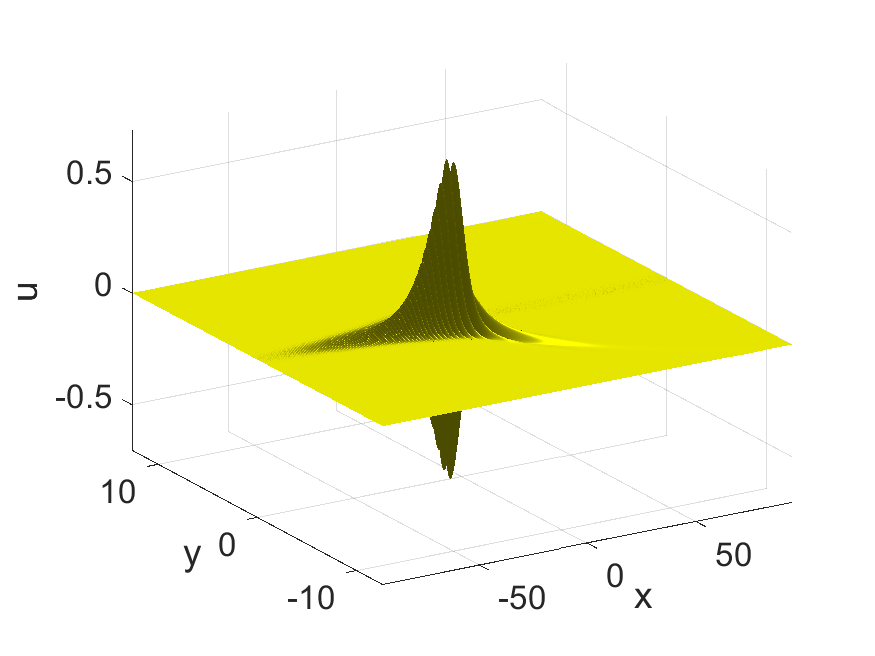}
		\end{minipage}
		\caption{  The variation of  $||u||_\infty$ with time  and  the profile of the numerical solution at  $t=0.5$ for the nonlinearity $f(u)=u^2+u^{3/2}$. }\label{u2u32}
	\end{figure}
	\begin{figure}
		\centering
		\includegraphics[height=5.5cm,width=7.5cm]{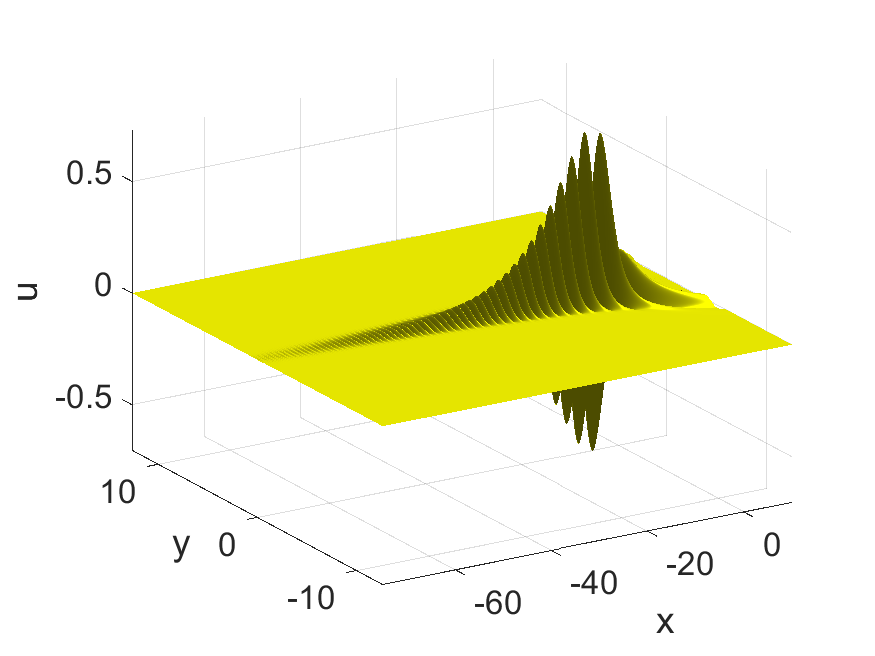}
		\caption{  The close-up look at the profile   $t=0.5$ for the nonlinearity $f(u)=u^2+u^{3/2}$. }\label{u2u32focus}
	\end{figure}
	
	\subsection{The  supercritical and subcritical nonlinearities} 
	Now,   we test our scheme for the generalized KP equation with supercritical and subcritical nonlinearities given by \mbox{$f(u)=-u^2+u^5$} with $\mu_1=-1, \mu_2=1$. The problem is solved on the rectangular region $[-5\pi, 5\pi]\times [-2\pi, 2\pi]$  by taking the initial condition \begin{equation}
		u_0(x,y)=4\left(\ee^ {-({x^2}+{y^2})}\right)_{xx} . \label{initialuminus2u5}
	\end{equation}
	The initial energy is negative such that $E(u_0)=-12918<0$.   Therefore, the conditions for Theorem \ref{blow-up-theorem-1}, part (ii) are fulfilled. 
	Figure \ref{uminus2u5} shows the profile of  $ u $,   $ u_x $ and $ u_y $   and  Fourier coefficients  at the critical time \mbox{$t=3\times 10^{-5}$} for the nonlinearity $f(u)=-u^2+u^5$. Similar to Figure \ref{u5u6}, the gradient of $u$  at the critical time diverges more rapidly than the solution itself. The numerical result indicates that the solution blows up in a finite time. This numerical result is also in complete agreement with the analytical result given in Theorem \ref{blow-up-theorem-1}, part (ii).
	\begin{figure}[!htbp]
		\begin{minipage}[t]{0.45\linewidth}
			\centering
			\includegraphics[height=5.5cm,width=7.5cm]{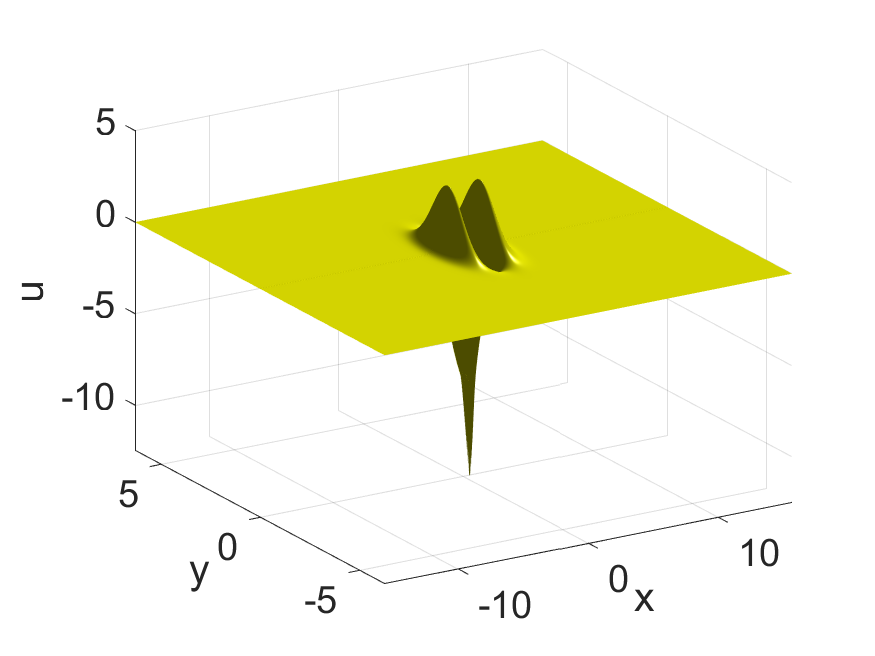}
		\end{minipage}%
		\hspace{20pt}
		\begin{minipage}[t]{0.45\linewidth}
			\centering
			\includegraphics[height=5.5cm,width=7.5cm]{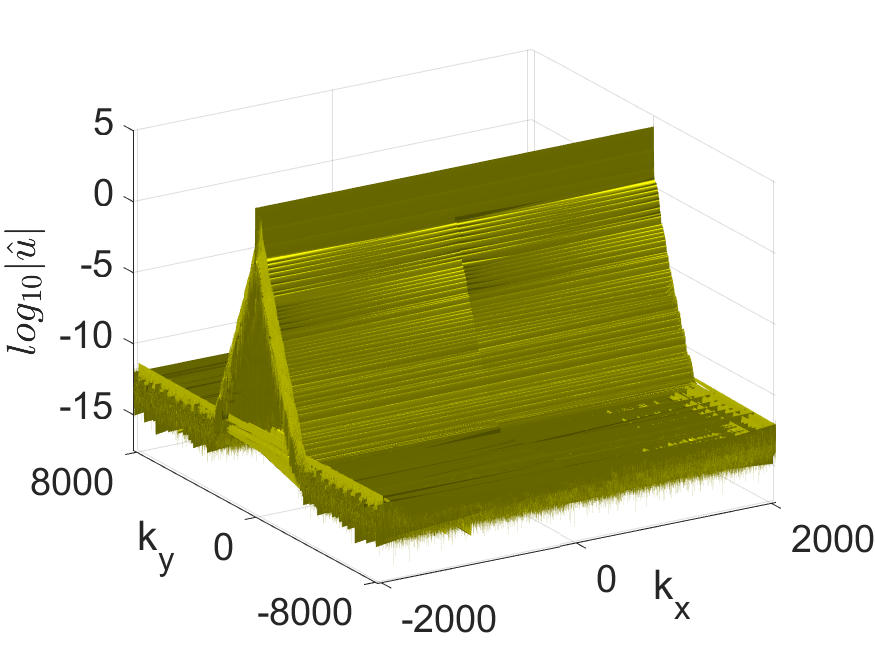}
		\end{minipage}
		\begin{minipage}[t]{0.45\linewidth}
			\centering
			\includegraphics[height=5.5cm,width=7.5cm]{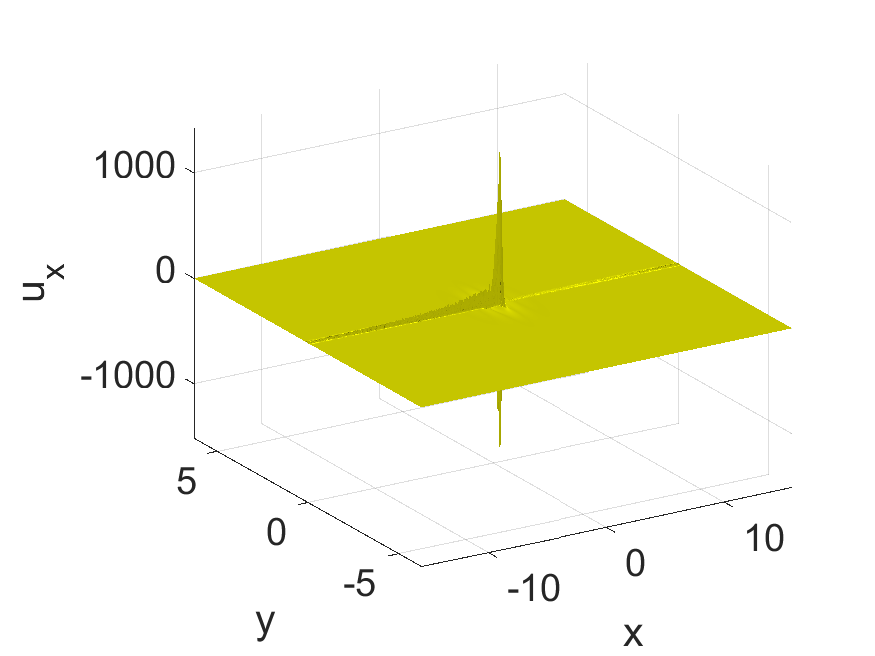}
		\end{minipage}%
		\hspace{40pt}
		\begin{minipage}[t]{0.45\linewidth}
			\centering
			\includegraphics[height=5.5cm,width=7.5cm]{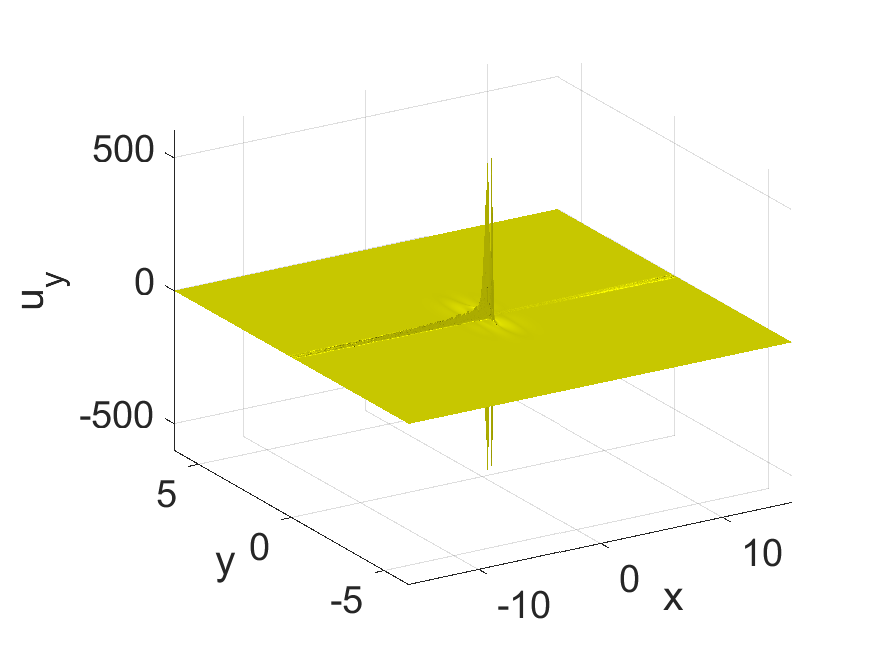}
		\end{minipage}
		\caption{ The profile of  $ u $,   $ u_x $ and $ u_y $  (top left, bottom left, bottom right)  and  Fourier coefficients  at $t=3\times 10^{-5}$ (top right) for the nonlinearity $f(u)=-u^2+u^5$.}\label{uminus2u5}
	\end{figure}

	Now, we focus on the generalized KP equation with different supercritical and subcritical nonlinearities where neither a uniform boundedness result nor a blow-up result is established theoretically. We start with the KP equation with supercritical and subcritical nonlinearities given by $f(u)=u^3-u^2$ with $\mu_1=1, \mu_2=-1$ and \mbox{$f(u)=u^3+u^2$} with $\mu_1=1, \mu_2=1$, respectively.   We choose the initial condition \eqref{initialuminus2u5}
	satisfying $E(u_0)=-359<0$ and  $E(u_0)=-200<0$, corresponding respectively \mbox{$f(u)=u^3-u^2$} and {$f(u)=u^3+u^2$}.
	The numerical experiments are carried out  from $t = 0$ to $t = 0.03$ taking  the number of temporal
	grid points $M = 10000$. The problem is solved on the rectangular region $[-5\pi, 5\pi]\times [-2\pi, 2\pi]$ using ${N_x}=2^{11}$ and ${N_y}=2^{13}$. Figure \ref{u3ueksi2} shows the profile of the numerical solution at different times.  We observe some distortions at two humps and dispersive oscillations in the $x$ direction as time increases. The amplitude of the solution increases very rapidly. Figure  \ref{u3ueksi2} indicates a blow-up.  In Figure \ref{u3u2}, the variation of  $L^{\infty}-$norm of the solution $u$  with time and 
	the profile of the numerical solution near the critical time $t=0.02105$ are depicted for the nonlinearity $f(u)=u^3+u^2$. The numerical results indicate that the solution blows up in a finite time. 
	
	\begin{figure}[!htbp] 
		\begin{minipage}[t]{0.45\linewidth}
			\centering
			\includegraphics[height=5.5cm,width=7.5cm]{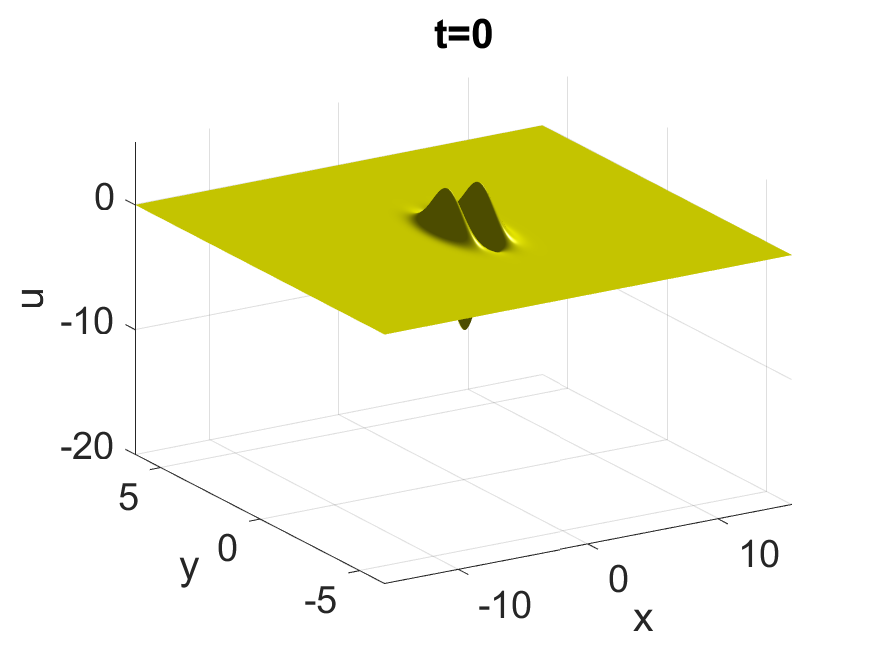}
		\end{minipage}%
		\hspace{20pt}
		\begin{minipage}[t]{0.45\linewidth}
			\centering
			\includegraphics[height=5.5cm,width=7.5cm]{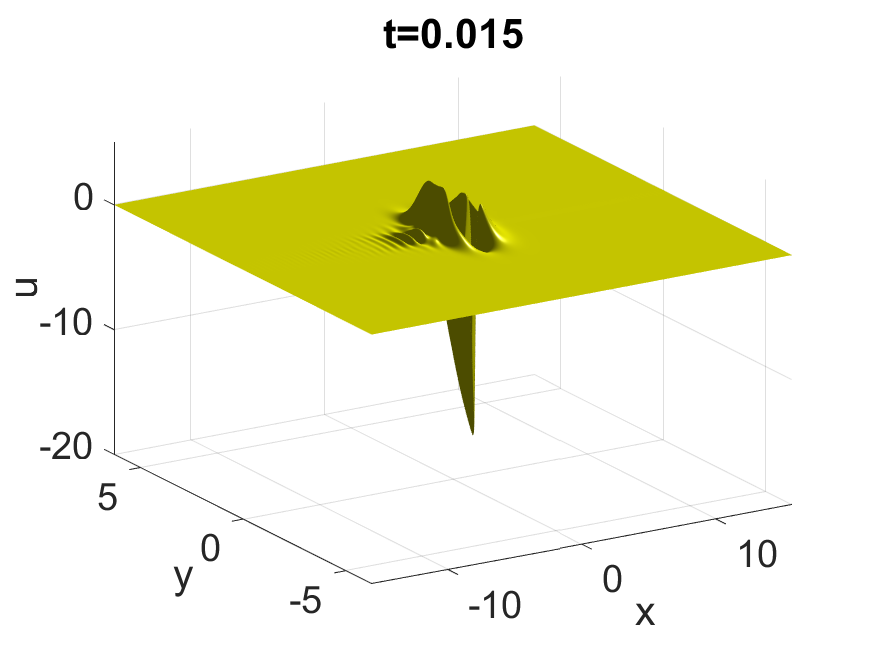}
		\end{minipage}
		\begin{minipage}[t]{0.45\linewidth}
			\centering
			\includegraphics[height=5.5cm,width=7.5cm]{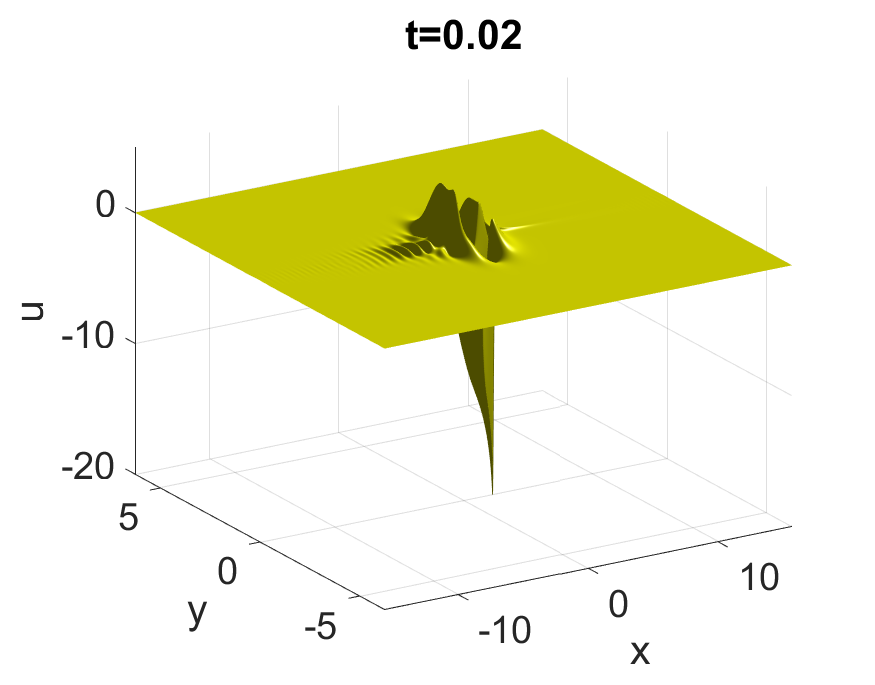}
		\end{minipage}%
		\hspace{40pt}
		\begin{minipage}[t]{0.45\linewidth}
			\centering
			\includegraphics[height=5.5cm,width=7.5cm]{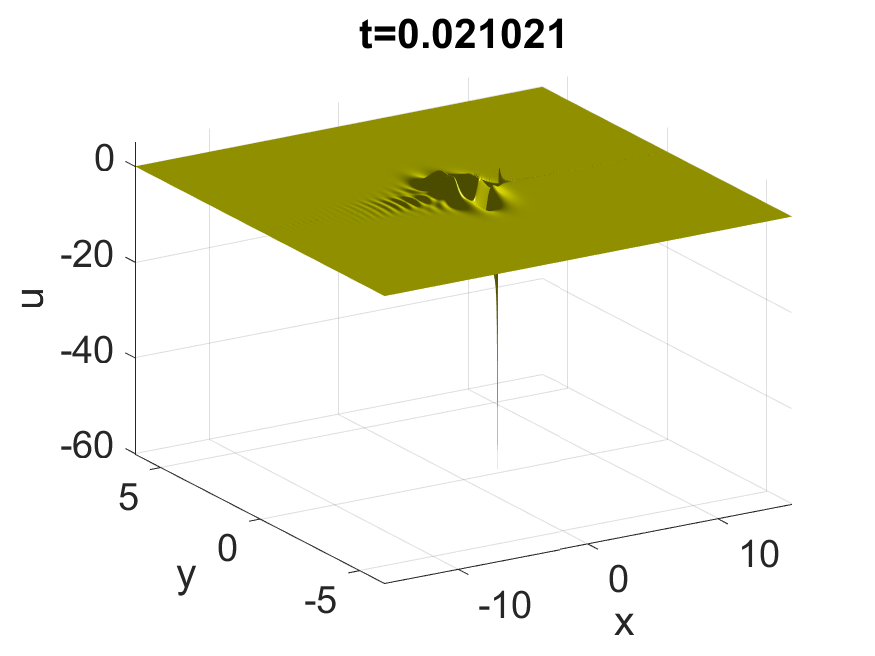}
		\end{minipage}
		\caption{ The variation of $u$ at different times for the nonlinearity $f(u)=u^3-u^2$.  }\label{u3ueksi2}
	\end{figure}
	\begin{figure}[!htbp] 
		\begin{minipage}[t]{0.45\linewidth}
			\centering
			\includegraphics[height=5.5cm,width=7.5cm]{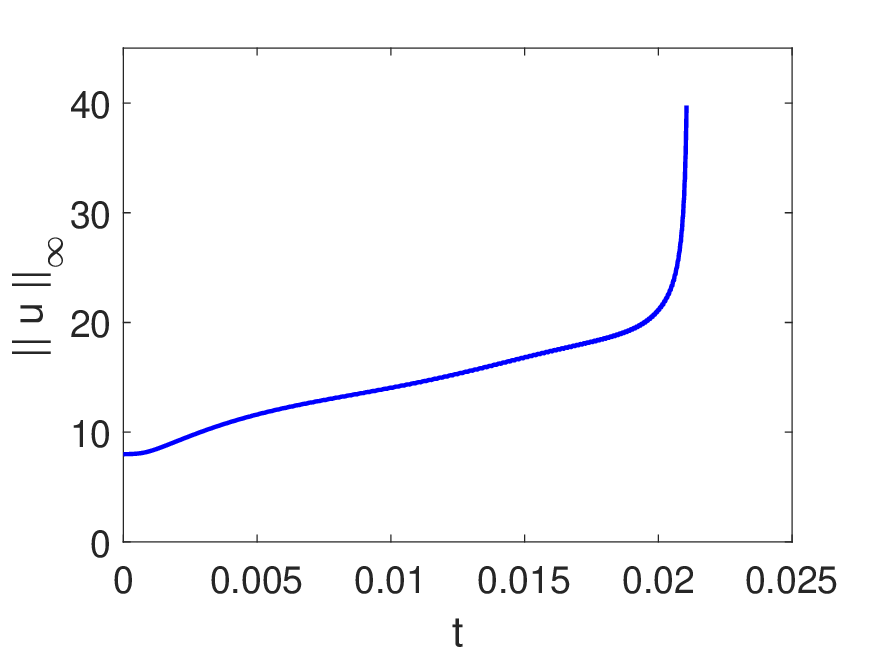}
		\end{minipage}%
		\hspace{20pt}
		\begin{minipage}[t]{0.45\linewidth}
			\centering
			\includegraphics[height=5.5cm,width=7.5cm]{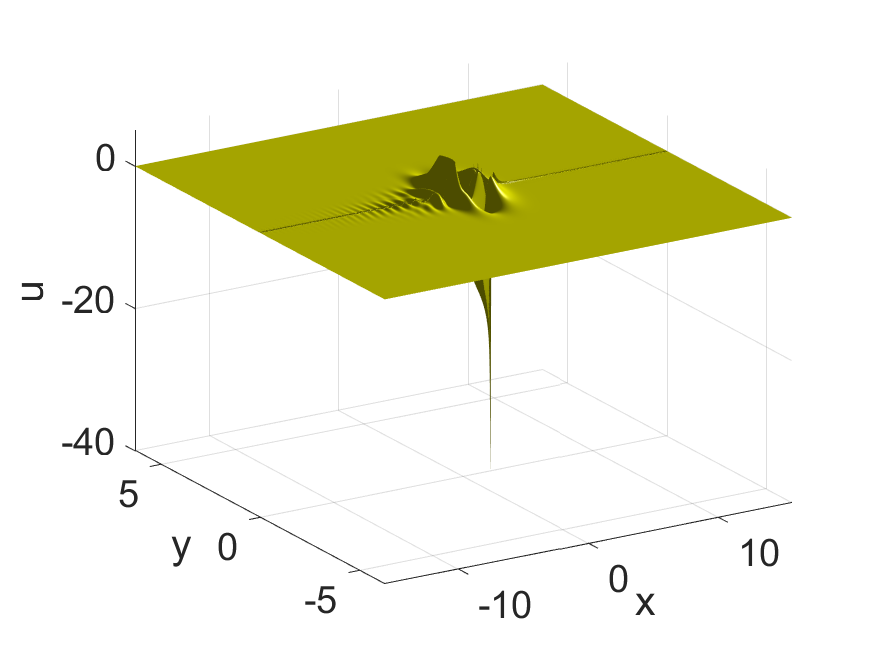}
		\end{minipage}
		
		\caption{ The variation of $||u||_\infty$ with time (left panel)  and  the profile of the numerical solution at $t=0.02105$ (right panel)
			for the nonlinearity $f(u)=u^3+u^2$. }\label{u3u2}
	\end{figure}

		\subsection{The sub-critical and critical nonlinearities case} 
		In the case of combined sub-critical and critical  nonlinearity  $f(u)=-u^2+u^{7/3}$ with $\mu_1=-1, \mu_2=1$,  we  take
		\begin{equation}
			u_0(x,y)={-10}\left(\ee^ {-({x^2}+{y^2})}\right)_{xx} 
		\end{equation}
		satisfying $E(u_0)=-1336<0$.  The experiment is carried out  from $t = 0$ to $t = 0.07$ taking  the number of temporal
		grid points $M = 10000$.
		In the left panel of Figure \ref{u-2u73},
		the variation of $||u||_\infty$ with time is presented. The amplitude of the numerical solution increases as time increases. Figure \ref{u-2u73} indicates that the solution blows up in a finite time.
		The profile of the numerical solution near the blow-up time $t=0.056707$ is depicted in the right panel of Figure \ref{u-2u73}.
		\begin{figure}[!htbp] 
			\begin{minipage}[t]{0.45\linewidth}
				\centering
				\includegraphics[height=5.5cm,width=7.5cm]{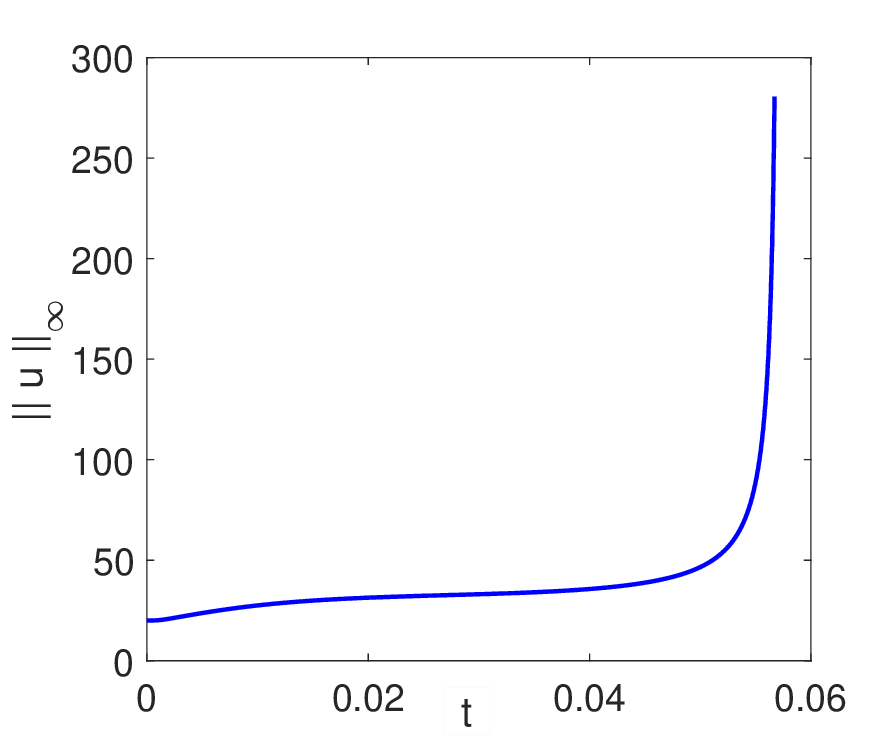}
			\end{minipage}%
			\hspace{40pt}
			\begin{minipage}[t]{0.45\linewidth}
				\centering
				\includegraphics[height=5.5cm,width=7.5cm]{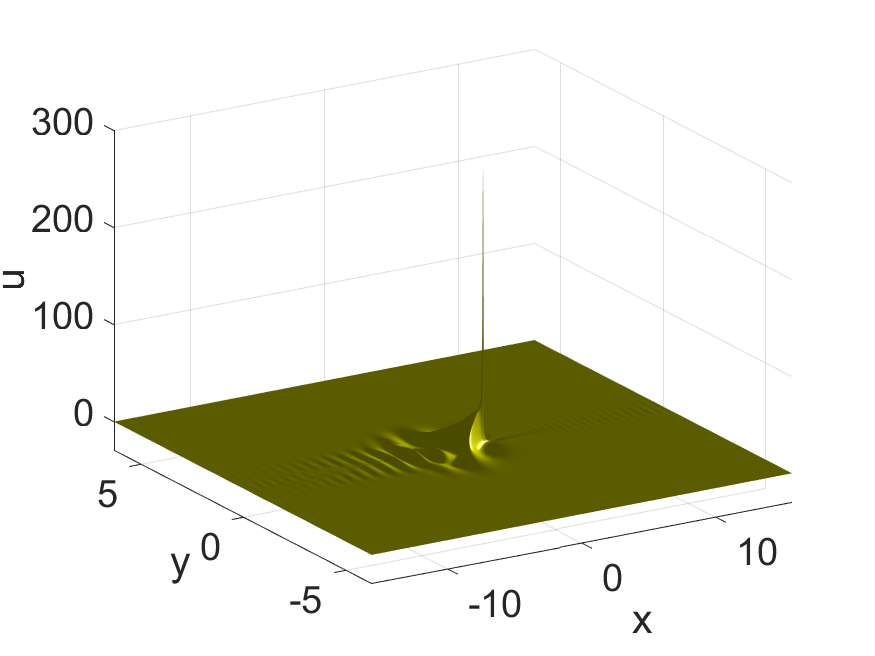}
			\end{minipage}
			\caption{  The variation of  $||u||_\infty$ with time  and  the profile of the numerical solution at  $t=0.056707$ for the nonlinearity $f(u)=-u^2+u^{7/3}$. }\label{u-2u73}
		\end{figure}
		
		
		
		\section*{Acknowledgment}
		
		A. E. is supported by Nazarbayev University under Faculty Development Competitive Research Grants Program for 2023-2025 (grant number 20122022FD4121). G. M. M.  used Computing Resources provided by the National Center for High-Performance Computing of Turkey (UHeM) under grant number 1015922023.
		
		\subsection*{Conflict of interest} The authors declare that they have no conflict of interest. 
		
		\subsection*{Data Availability}
		There is no data in this paper.

	\end{document}